\documentclass[12pt, twoside]{article}
\usepackage{amsmath,amsthm,amssymb}
\usepackage{times}
\usepackage{enumerate}

\def\titlerunning#1{\gdef\titrun{#1}}
\makeatletter
\def\author#1{\gdef\autrun{\def\and{\unskip, }#1}\gdef\@author{#1}}
\def\address#1{{\def\and{\\\hspace*{18pt}}\renewcommand{\thefootnote}{}%
\footnote {#1}}%
\markboth{\autrun}{\titrun}}
\makeatother
\def\email#1{e-mail: #1}
\def\subjclass#1{{\renewcommand{\thefootnote}{}%
\footnote{\emph{Mathematics Subject Classification (2010):} #1}}}
\def\keywords#1{\par\medskip
\noindent\textbf{Keywords.} #1}

\numberwithin{equation}{section}
\newcommand{\beq}{\begin{equation}}
\newcommand{\eeq}{\end{equation}}
\newcommand{\bea}{\begin{eqnarray}}
\newcommand{\eea}{\end{eqnarray}}
\newcommand{\beas}{\begin{eqnarray*}}
\newcommand{\eeas}{\end{eqnarray*}}

\newtheorem{theorem}{Theorem}[section]
\newtheorem{definition}[theorem]{Definition}
\newtheorem{proposition}[theorem]{Proposition}
\newtheorem{corollary}[theorem]{Corollary}
\newtheorem{lemma}[theorem]{Lemma}
\newtheorem{remark}[theorem]{Remark}
\newtheorem{example}[theorem]{Example}
\newtheorem{examples}[theorem]{Examples}
\newtheorem{foo}[theorem]{Remarks}

\newtheorem{assumption}[theorem]{Hypothesis}

\newcommand{\bG}{\mathbb G}

\newcommand{\p}{\partial}
\newcommand{\ee}{\ell}

\newcommand{\bM}{\mathbb M}

\newcommand{\Rn}{\mathbb R^n}
\newcommand{\Om}{\Omega}

\newcommand{\Ho}{\mathcal H}

\newcommand{\di}{\mathfrak h}
\newcommand{\M}{\mathbb M}

\newcommand{\R}{\mathbb R}

\newcommand{\ve}{\varepsilon}

\begin{document}


\baselineskip=17pt


\titlerunning{Curvature-dimension inequalities, etc.}

\title{Curvature-dimension inequalities and Ricci lower bounds for sub-Riemannian manifolds with transverse symmetries}

\author{Fabrice Baudoin
\and 
Nicola Garofalo}

\date{}

\maketitle

\address{Fabrice Baudoin: Department of Mathematics, Purdue University, West Lafayette, IN 47907; \email{fbaudoin@purdue.edu}
\and
Nicola Garofalo: Dipartimento d'Ingegneria Civile e Ambientale (DICEA), Universit\`a di Padova via Trieste 63, 35131 Padova, Italy; \email{nicola.garofalo@unipd.it}}



\begin{abstract}
Let $\M$ be a smooth connected manifold endowed with a smooth measure $\mu$ and a smooth locally subelliptic diffusion operator $L$ satisfying $L1=0$, and which is symmetric with respect to $\mu$. Associated with $L$ one has the \textit{carr\'e du champ} $\Gamma$ and a canonical distance $d$, with respect to which we suppose that $(\M,d)$ be complete. We assume that $\M$ is also equipped with another first-order differential bilinear form $\Gamma^Z$ and we assume that $\Gamma$ and $\Gamma^Z$ satisfy the  Hypothesis \ref{A:exhaustion}, \ref{A:main_assumption}, and \ref{A:regularity} below. With these forms we introduce in \eqref{cdi} below a generalization of the curvature-dimension inequality from Riemannian geometry, see Definition \ref{D:cdi}. In our main results we prove that, using solely \eqref{cdi}, one can develop a theory which parallels the celebrated works of Yau, and Li-Yau on complete manifolds with Ricci curvature bounded from below. We also obtain an analogue of the Bonnet-Myers theorem. In Section \ref{S:appendix} we construct large classes of sub-Riemannian manifolds with transverse symmetries which satisfy the generalized curvature-dimension inequality \eqref{cdi}. Such classes include all Sasakian manifolds whose horizontal Webster-Tanaka-Ricci curvature is bounded from below, all Carnot groups with step two, and wide subclasses of principal bundles over Riemannian manifolds whose Ricci curvature is bounded from below. 
\end{abstract}

\section{Introduction}\label{S:intro}

In the present paper we introduce a generalization of the curvature-dimension inequality from Riemannian geometry which, as we show, is appropriate for some sub-Riemannian geometries. The central objective of our work is developing a program which, through a systematic use of such curvature-dimension inequality, connects the geometry of the ambient manifold, expressed in terms of lower bounds on a generalization of the Ricci tensor, to global properties of solutions of a certain canonical second order diffusion (non elliptic) partial differential operator, a sub-Laplacian, and of its associated heat semigroup. 


In Riemannian geometry the Ricci tensor plays a fundamental role. Its connection with the Laplace-Beltrami operator is provided by the celebrated identity of Bochner which states that if $\bM$ is a $n$-dimensional Riemannian manifold with Laplacian $\Delta$, for any $f\in C^\infty(\bM)$ one has
\begin{equation}\label{rb}
\Delta(|\nabla f|^2) = 2 ||\nabla^2 f||^2 + 2 <\nabla f,\nabla(\Delta f)> + 2 \text{Ric}(\nabla f,\nabla f).
\end{equation}
Consider the following differential forms on
functions $f, g \in C^\infty(\bM)$,
\begin{equation*}
\Gamma(f,g) =\frac{1}{2}(\Delta(fg)-f\Delta g-g\Delta f)=( \nabla f
, \nabla g ),
\end{equation*}
and
\begin{equation*}
\Gamma_{2}(f,g) = \frac{1}{2}\big[\Delta\Gamma(f,g) - \Gamma(f,
\Delta g)-\Gamma (g,\Delta f)\big].
\end{equation*}
When $f=g$, we simply write $\Gamma(f) = \Gamma(f,f)$, $\Gamma_2(f) = \Gamma_2(f,f)$. The functional calculus of these forms was introduced and developed in \cite{Bakry-Emery}.
As an application of the Bochner's formula, which in terms of these functionals can be reformulated as
\[ \Delta \Gamma(f) = 2 ||\nabla^2 f||^2 + 2 \Gamma(f,\Delta f) +
2\ \text{Ric}(\nabla f,\nabla f), \] one obtains
\[
\Gamma_{2}(f)= \| \nabla^2 f \|_2^2 +\text{Ric}(\nabla f, \nabla
f).
\]
Using the Cauchy-Schwarz inequality, which gives $\| \nabla^2 f
\|_2^2\ge \frac{1}{n} (\Delta f)^2$, we thus see that the assumption that the
Riemannian Ricci tensor on $\bM$ be bounded from below by $\rho_1 \in
\mathbb{R}$ implies the so-called \emph{curvature-dimension
inequality} CD$(\rho_1,n)$:
\begin{equation}\label{CDi}
\Gamma_{2}(f) \ge \frac{1}{n} (\Delta f)^2 + \rho_1
\Gamma(f),\ \ \ \ f\in C^\infty(\bM).
\end{equation}
In the hands of D. Bakry, M. Ledoux and their co-authors the inequality  \eqref{CDi} 
has proven a powerful tool in combination with a systematic use of fine properties of the heat semigroup. Among other things, these authors have succeeded in re-deriving, from a purely analytical perspective, several of the
well-known fundamental results which, in Riemannian geometry, are obtained
under the assumption that the Ricci curvature be bounded from below, see for instance \cite{bakry-stflour},  \cite{Bakry-Ledoux1} \cite{ledoux-zurich},
\cite{li}.
It is remarkable that the curvature dimension inequality \eqref{CDi} perfectly  captures the
notion of Ricci curvature lower bound. It was in fact proved by Bakry in Proposition 6.2 in \cite{bakry-stflour} that: \emph{on a $n$-dimensional Riemannian manifold $\bM$ the inequality \emph{CD}$(\rho_1,n)$ implies} Ric $\ge \rho_1$. In conclusion, Ric$\ge \rho_1$ $\Longleftrightarrow$ \text{CD}$(\rho_1,n)$.

Inspired by the ideas contained in the above mentioned works, in the present paper we introduce a generalization of the curvature-dimension inequality \eqref{CDi} which can be successfully used in sub-Riemannian geometry. At this point, we feel it is important to say few words concerning the organization of the paper. The essential contribution of the present work is based on ideas and tools which are purely analytical in nature: as we have mentioned above, we systematically use the heat semigroup to derive new results in sub-Riemannian geometry. On the other hand, an equally important aspect of the present work is the construction of the examples from geometry: as the title indicates, the main class studied in this paper is that of sub-Riemannian manifolds with transverse symmetries. We show that such class is quite large, as it incorporates (but is not limited to) examples which are geometrically as diverse as CR manifolds with vanishing Tanaka-Webster torsion (Sasakian manifolds), graded nilpotent Lie groups of step two, orthonormal frame bundles. To facilitate the perusal of this paper by an audience of analysts we have strived as much as possible to separate the presentation of the analytical part of our work from the geometrical discussion of the examples. With this objective in mind, we have chosen to present the analytical part of the paper in an axiomatic way. By this we mean that all that is asked to a reader less inclined toward geometry is to accept a set of four ``abstract'' assumptions, which are listed as Hypothesis \ref{A:exhaustion}, \ref{A:main_assumption}, Definition \ref{D:cdi} and Hypothesis \ref{A:regularity} below. The geometrical relevance, and the motivation, of such assumptions is unraveled in Section \ref{S:appendix}, where we discuss the examples and we develop the geometric setup.

With this being said, we now introduce the relevant setting. We consider a smooth, connected manifold $\bM$ endowed with a smooth measure $\mu$ and a smooth second-order diffusion operator $L$ with real coefficients satisfying $L1=0$, and which is symmetric with respect to $\mu$ and non-positive. By this we mean that 
\begin{equation}\label{sa}
\int_\bM f L g d\mu=\int_\bM g Lf d\mu,\ \ \ \ \ \ \int_\bM f L f d\mu \le 0,
\end{equation}
for every $f , g \in C^ \infty_0(\bM)$. We make the technical assumption that $L$ be locally subelliptic in the sense of \cite{FP1}.
We associate with $L$ the following symmetric, first-order, differential bilinear form: 
\begin{equation}\label{gamma}
\Gamma(f,g) =\frac{1}{2}(L(fg)-fLg-gLf), \quad f,g \in C^\infty(\bM).
\end{equation}
The expression $\Gamma(f) = \Gamma(f,f)$ is known as the \textit{ carr\'e du champ}. Furthermore,
using the results in \cite{PS1}, locally in the neighborhood of every point $x\in \M$ we can write
\begin{equation}\label{Lrep}
L =- \sum_{i=1}^m X_i^* X_i,
\end{equation}
where the vector fields $X_i$ are Lipschitz continuous (such representation is not unique, but this fact is of no consequence for us). Therefore, for any $x\in \M$ there exists an open neighborhood $U_x$ such that for any $f\in C^\infty(\M)$ we have in $U_x$  
\begin{equation}\label{Grep}
\Gamma(f)  = \sum_{i=1}^m (X_i f)^2.
\end{equation}
This shows that $\Gamma(f)\ge 0$ and it actually only involves differentiation of order one.

 Furthermore, as it is clear from (\ref{gamma}), the value of $\Gamma(f)(x)$ does not depend  on the particular representation \eqref{Lrep} of $L$.

With the operator $L$ we can also associate a canonical distance:
\begin{equation}\label{di}
d(x,y)=\sup \left\{ |f(x) -f(y) | \mid f \in  C^\infty(\bM) , \| \Gamma(f) \|_\infty \le 1 \right\},\ \ \  \ x,y \in \bM,
\end{equation}
where for a function $g$ on $\bM$ we have let $||g||_\infty = \underset{\bM}{\text{ess} \sup} |g|$.  A tangent vector $v\in T_x\M$ is called \emph{subunit} for $L$ at $x$ if   
$v = \sum_{i=1}^m a_i X_i(x)$, with $\sum_{i=1}^m a_i^2 \le 1$, see \cite{FP1}. It turns out that the notion of subunit vector for $L$ at $x$ does not depend on the local representation \eqref{Lrep} of $L$. A Lipschitz path $\gamma:[0,T]\to \M$ is called subunit for $L$ if $\gamma'(t)$ is subunit for $L$ at $\gamma(t)$ for a.e. $t\in [0,T]$. We then define the subunit length of $\gamma$ as $\ell_s(\gamma) = T$. Given $x, y\in \M$, we indicate with 
\[
S(x,y) =\{\gamma:[0,T]\to \M\mid \gamma\ \text{is subunit for}\ L, \gamma(0) = x,\ \gamma(T) = y\}.
\]
In this paper we assume that 
\[
S(x,y) \not= \varnothing,\ \ \ \ \text{for every}\ x, y\in \M.
\]
Under such assumption  it is easy to verify that
\begin{equation}\label{ds}
d_s(x,y) = \inf\{\ell_s(\gamma)\mid \gamma\in S(x,y)\},
\end{equation}
defines a true distance on $\M$. Furthermore, thanks to Lemma 5.43 in \cite{CKS} we know that
\[
d(x,y) = d_s(x,y),\ \ \ x, y\in \mathbb M,
\]
hence we can work indifferently with either one of the distances $d$
or $d_s$. 
Throughout this paper we assume that the metric space $(\M,d)$ be complete.

We also suppose given on $\bM$ a symmetric, first-order differential bilinear form $\Gamma^Z:C^\infty(\bM)\times C^\infty(\bM) \to \R$. Hereafter in this, the term symmetric first-order differential form means that $\Gamma^Z(f,g)=\Gamma^Z(g,f)$ and
\begin{equation}\label{Zleib}
\Gamma^Z(fg,h) = f\Gamma^Z(g,h) + g \Gamma^Z(f,h).
\end{equation}
In particular, we have $\Gamma^Z(1) = 0$, where, as for $\Gamma$, we have set $\Gamma^Z(f) = \Gamma^Z(f,f)$. We assume that $\Gamma^Z(f)\ge 0$.

We will work with four general assumptions. The former three will be listed as Hypotheses \ref{A:exhaustion}, \ref{A:main_assumption} and Definition \ref{D:cdi}, the fourth one will be introduced in Hypothesis \ref{A:regularity} below. 

\begin{assumption}\label{A:exhaustion}
There exists an increasing
sequence $h_k\in C^\infty_0(\bM)$   such that $h_k\nearrow 1$ on
$\bM$, and \[
||\Gamma (h_k)||_{\infty} +||\Gamma^Z (h_k)||_{\infty}  \to 0,\ \ \text{as} \ k\to \infty.
\]
\end{assumption}
We will also assume that the following commutation relation be satisfied.

\begin{assumption}\label{A:main_assumption} 
For any $f \in C^\infty(\bM)$ one has
\[
\Gamma(f, \Gamma^Z(f))=\Gamma^Z( f, \Gamma(f)).
\]
\end{assumption}

Let us notice explicitly that when $\bM$ is a Riemannian manifold, $\mu$ is the Riemannian volume on $\bM$, and $L = \Delta$, then $d(x,y)$ in \eqref{di} is equal to the Riemannian distance on $\bM$. In this situation if  we take $\Gamma^Z\equiv 0$, then Hypothesis \ref{A:exhaustion}, \ref{A:main_assumption} are fulfilled. In fact, Hypothesis \ref{A:main_assumption} is trivially satisfied, whereas Hypothesis \ref{A:exhaustion} is equivalent to assuming that $(\bM,d)$ be a complete metric space, which we are assuming anyhow. More generally, in all the examples of  Section \ref{S:appendix}, Hypothesis \ref{A:exhaustion} is equivalent to assuming that $(\bM,d)$ be a complete metric space (the reason is that in those examples $\Gamma+\Gamma^Z$ is the \textit{carr\'e du champ} of the Laplace-Beltrami operator of a Riemannian structure whose completeness is equivalent to the completeness of  $(\bM,d)$). On the other hand, Hypothesis \ref{A:main_assumption} is also verified as a consequence of the assumptions about the existence of transverse symmetries that we make. 

Before we proceed with the discussion, we pause to stress that, in the generality in which we work the bilinear differential form $\Gamma^Z$,  unlike $\Gamma$, is not a priori canonical. Whereas $\Gamma$ is determined once $L$ is assigned, the form $\Gamma^Z$ in general is not intrinsically associated with $L$. However, in the geometric examples described in this paper (for this see the discussion below and Section \ref{S:appendix}) the choice of $\Gamma^Z$ will be natural and even canonical, up to a constant. This is the case, for instance, of the important example of CR Sasakian manifolds. The reader should think of $\Gamma^Z$ as an orthogonal complement of $\Gamma$: the bilinear form $\Gamma$ represents the square of the length of the gradient in the horizontal directions, whereas $\Gamma^Z$ represents the square of the length of the gradient along the vertical directions. 


Given the sub-Laplacian $L$ and the first-order bilinear forms $\Gamma$ and $\Gamma^Z$ on $\bM$, we now introduce the following second-order differential forms:
\begin{equation}\label{gamma2}
\Gamma_{2}(f,g) = \frac{1}{2}\big[L\Gamma(f,g) - \Gamma(f,
Lg)-\Gamma (g,Lf)\big],
\end{equation}
\begin{equation}\label{gamma2Z}
\Gamma^Z_{2}(f,g) = \frac{1}{2}\big[L\Gamma^Z (f,g) - \Gamma^Z(f,
Lg)-\Gamma^Z (g,Lf)\big].
\end{equation}
Observe that if $\Gamma^Z\equiv 0$, then $\Gamma^Z_2 \equiv 0$ as well. As for $\Gamma$ and $\Gamma^Z$, we will use the notations  $\Gamma_2(f) = \Gamma_2(f,f)$, $\Gamma_2^Z(f) = \Gamma^Z_2(f,f)$.


We are ready to introduce the central character of our paper, a generalization of the above mentioned curvature-dimension inequality \eqref{CDi}.

\begin{definition}\label{D:cdi}
We shall say that $\M$ satisfies the \emph{generalized curvature-dimension inequality} \emph{CD}$(\rho_1,\rho_2,\kappa,d)$ with respect to $L$ and $\Gamma^Z$ if there exist constants $\rho_1 \in \mathbb{R}$, $\rho_2 >0$, $\kappa \ge 0$, and $0<d\le \infty$ such that the inequality 
\begin{equation}\label{cdi}
\Gamma_2(f) +\nu \Gamma_2^Z(f) \ge \frac{1}{d} (Lf)^2 +\left( \rho_1 -\frac{\kappa}{\nu} \right) \Gamma(f) +\rho_2 \Gamma^Z(f)
\end{equation}
 hold for every  $f\in C^\infty(\bM)$ and every $\nu>0$.
\end{definition}

It is worth observing explicitly that if in Definition \ref{D:cdi} we choose $L = \Delta$, $\Gamma^Z \equiv 0$,  $d = n=$ dim($\bM$), and $\kappa = 0$, we obtain the Riemannian curvature-dimension inequality CD$(\rho_1,n)$ in \eqref{CDi} above. Thus, the case of Riemannian manifolds is trivially encompassed by Definition \ref{D:cdi}. 
We also remark that, changing $\Gamma^Z$ into $a \Gamma^Z$, where $a>0$, changes the inequality  CD$(\rho_1,\rho_2,\kappa,d)$ into CD$(\rho_1,a \rho_2, a \kappa,d)$. We express this fact by saying that the quantity $\frac{\kappa}{\rho_2}$ is intrinsic. 
Hereafter, when we say that  $\M$ satisfies the curvature dimension inequality CD$(\rho_1,\rho_2,\kappa,d)$ (with respect to $L$ and $\Gamma^Z$), we will routinely avoid repeating at each occurrence the sentence ``for some $\rho_2>0$, $\kappa\ge 0$ and $d >0$''. Instead, we will explicitly mention whether $\rho_1 = 0$, or $>0$, or simply $\rho_1\in \R$. 
The reason for this is that the parameter $\rho_1$ in the inequality \eqref{cdi} has a special relevance since, in the geometric examples in Section \ref{S:appendix}, it represents the lower bound on a sub-Riemannian generalization of the Ricci tensor. Thus, $\rho_1 = 0$ is, in our framework, the counterpart of  the Riemannian Ric\ $\ge 0$, whereas when $\rho_1 >0$ $(<0)$, we are dealing with the counterpart of the case Ric\ $>0$ (Ric  bounded from below by a negative constant). 

Since, as we have stressed above, we wish to present our results in an axiomatic way, we will also need the following assumption which is necessary to rigorously justify computations on functionals of the heat semigroup. Hereafter, we will denote  by $P_t = e^{tL}$ the semigroup generated by the diffusion operator $L$, see the discussion below, and Section \ref{S:comp}. 
\begin{assumption}\label{A:regularity}
The semigroup $P_t$ is stochastically complete that is, for $t \ge 0$, $P_t 1=1$ and  for every $f \in C_0^\infty(\bM)$ and $T \ge 0$, one has 
\[
\sup_{t \in [0,T]} \| \Gamma(P_t f)  \|_{ \infty}+\| \Gamma^Z(P_t f) \|_{ \infty} < +\infty.
\]

\end{assumption}

In the Riemannian setting ($L = \Delta$ and $\Gamma^Z \equiv 0$), Hypothesis \ref{A:regularity} is satisfied if one assumes the lower bound Ricci$\ \ge \rho$, for some $\rho \in \R$. This can be derived from the paper by Yau \cite{Yau2} and Bakry's note \cite{bakry-CRAS}. It thus follows that, in the Riemannian case, the Hypothesis \ref{A:regularity} is not needed since it can be derived as a consequence of the curvature-dimension inequality CD$(\rho_1,n)$ in \eqref{CDi} above. In this paper we will prove that, more generally, this situation occurs in the sub-Riemannian setting of our work. As a consequence of the results in Section \ref{S:appendix} below, in Theorem \ref{T:on revient tojour a son premiere amour} we prove that, in every sub-Riemannian manifold with transverse symmetries  of Yang-Mills type (for the relevant definitions see Sections \ref{S:appendix} and 3 below), the Hypothesis \ref{A:regularity} is not needed since it follows (in a non-trivial way) from the generalized curvature-dimension inequality CD$(\rho_1,\rho_2,\kappa,d)$ in Definition \ref{D:cdi} above.  

In this connection it is worth observing that, even in the abstract framework of the present work, if we assume that $\Gamma^Z = 0$, then the Hypothesis \ref{A:regularity} becomes redundant since it can be actually obtained a consequence of  CD$(\rho_1,n)$. This can be seen from the results in Chapter 5 of \cite{sobolog}. Whether it is possible to generalize this fact to the genuinely non-Riemannian situation of $\Gamma^Z \not= 0$, we must leave to a future study. Concerning our axiomatic presentation, we  finally mention that, had we chosen to do so, we could have developed our results in an even more abstract setting, as Bakry and Ledoux often do in their works. We could have worked with abstract Markov diffusion generators on measure spaces and replaced Hypothesis \ref{A:exhaustion} and Hypothesis \ref{A:regularity} with the existence of a nice  algebra of functions which is dense in the domain of $L$ (see Definition 2.4.2 in \cite{sobolog} for the precise properties that should be satisfied by this algebra  when $\Gamma^Z=0$). However assuming  the existence of such algebra is a strong assumption that may be difficult to verify in some concrete situations.
 
The above discussion prompts us to underline the distinctive aspect of the theory developed in the present paper: \emph{for the class of complete sub-Riemannian manifolds with transverse symmetries of Yang-Mills type that we study in Section 3, all our results are solely deduced from the curvature-dimension inequality \emph{CD}$(\rho_1,\rho_2,\kappa,d)$ in \eqref{cdi}}.

To introduce our results we recall that in their celebrated work  \cite{LY} Li and Yau, generalizing to the heat equation some fundamental works of Yau, see for instance \cite{Yau},  obtained various a priori gradient bounds for positive solutions of the heat equation on a complete $n$-dimensional Riemannian manifold $\M$. When Ric $\ge 0$ the Li-Yau inequality states that if $u>0$ is a solution of $\Delta u - u_t = 0$ in $\bM\times (0,\infty)$, then
\begin{equation}\label{RLY}
\frac{|\nabla u|^2}{u^2} - \frac{u_t}{u} \le \frac{n}{2t}.
\end{equation}
Notice that in the flat $\Rn$ the Gauss-Weierstrass kernel $u(x,t) = (4\pi t)^{-n/2} \exp(-|x|^2/4t)$ satisfies \eqref{RLY} with equality. The inequality \eqref{RLY} was the central tool for obtaining a scale invariant Harnack inequality for the heat equation and optimal off-diagonal upper bounds for the heat kernel on $\M$, see Corollary 3.1 and Theorem 4.1 in \cite{LY}. The proof of \eqref{RLY} hinges crucially on Bochner's identity \eqref{rb} above, and on the Laplacian comparison theorem which, for a manifold with Ric\ $\ge 0$, states that, given a base point $x_0\in \bM$, and denoted with $\rho(x)$ the Riemannian distance from $x$ to $x_0$, then
\begin{equation}\label{lct}
\Delta \rho(x) \le \frac{n-1}{\rho(x)},
\end{equation}
outside of the cut-locus of $x_0$ (and  globally in $\mathcal D'(\bM)$). 
As it is well-known, see for instance \cite{CLN}, the proof of \eqref{lct} exploits the theory of Jacobi fields. In sub-Riemannian geometry the exponential map is not a local diffeomorphism. As a consequence of this obstacle, a general sub-Riemannian comparison theorem such as \eqref{lct} presently represents \emph{terra incognita}. 

The main thrust of the present work is that, despite such obstructions, we have succeeded in establishing a sub-Riemannian generalization of the Li-Yau inequalities. In our approach, we completely avoid those tools from geometry that appear typically Riemannian, and instead base our analysis on a systematic use of some entropic inequalities for the heat semigroup that are inspired by the works  \cite{bakry-baudoin},  \cite{Bakry-Ledoux}, \cite{baudoin-bonnefont}, and which, as we have stressed above, in the geometric framework of this paper we solely derive from our generalized curvature-dimension inequality  CD$(\rho_1,\rho_2,\kappa,d)$ in \eqref{cdi}. 

More precisely, let $P_t = e^{t L}$ indicate the heat semigroup on $\bM$ associated with the operator $L$. It is well-known that $P_t$ is sub-Markovian, i.e., $P_t 1\le 1$, and it has a positive and symmetric kernel $p(x,y,t)$. If $f\in C^\infty_0(\bM)$ the function 
\[
u(x,t) = P_t f(x) = \int_\bM p(x,y,t) f(y) d\mu(y),
\]
solves the Cauchy problem 
\[
\begin{cases}
\frac{\p u}{\p t} - Lu = 0, \ \ \ \ \text{in}\ \bM\times (0,\infty),
\\
u(x,0) = f(x),\ \ \ \ \ x\in \bM.
\end{cases}
\]
For fixed $x\in \bM$ and $T>0$ we introduce  the functionals
\[
\Phi_1 (t)=P_t \left( (P_{T-t} f) \Gamma (\ln P_{T-t}f) \right)(x),
\]
\[
\Phi_2 (t)=P_t \left( (P_{T-t} f) \Gamma^Z (\ln P_{T-t}f) \right)(x),
\]
which are defined for $0\le t<T$.
The fundamental observation is that, in our framework, the inequality CD$(\rho_1,\rho_2,\kappa,d)$ in \eqref{cdi} leads to the following differential inequality
\begin{equation}\label{ei}
\left(- \frac{b'}{2\rho_2} \Phi_1 +b \Phi_2 \right)' \ge  -\frac{2b' \gamma}{d\rho_2} LP_Tf + \frac{b'  \gamma^2}{d\rho_2}  P_T f,
\end{equation}
where $b$ is any smooth, positive and decreasing function on the time
interval $[0,T]$ and
\[
\gamma=\frac{d}{4} \left( \frac{b''}{b'} +\frac{\kappa}{\rho_2} \frac{b'}{b} +2\rho_1 \right).
\]
Depending on the value of $\rho_1$, a good choice of the function $b$ leads to a generalized Li-Yau type inequality, see Theorem \ref{T:ge} below. In the special case $\rho_1 = 0$ (i.e., our Ric\ $\ge 0$), the latter becomes 
\begin{equation}\label{srLY}
\Gamma (\ln P_t f) +\frac{2 \rho_2}{3}  t \Gamma^Z (\ln P_t f) \le
\left(1+\frac{3\kappa}{2\rho_2}\right) \frac{LP_t f}{P_t f} +\frac{d\left(
1+\frac{3\kappa}{2\rho_2}\right)^2}{2t},
\end{equation}
for every sufficiently nice function $f\ge 0$ on $\bM$.
In the Riemannian case, when $\Gamma^Z \equiv 0$, and $\kappa = 0$, the inequality \eqref{srLY} is precisely the Li-Yau inequality \eqref{RLY}, except that our inequality holds for positive solutions of the heat equation of the type $u = P_t f$, i.e., they arise from an initial datum $f$, whereas in the original Li-Yau inequality \eqref{RLY} such limitation is not present. 

It is worth emphasizing at this point that, even in the Riemannian case, our approach, based on a systematic use of the  entropic inequality \eqref{ei} above, provides a new and elementary proof of several fundamental results for complete manifolds with Ric $\ge 0$. In this framework, in fact, besides the already mentioned Li-Yau gradient estimates, with the ensuing scale invariant Harnack inequality and the Liouville theorem of Yau, see \cite{Yau}, we also obtain an elementary proof of the fundamental monotonicity of Perelman's entropy for the heat equation, see \cite{Per}, and of the volume doubling property on Riemannian manifolds (for the statement of this classical result see for instance \cite{Chavel}). For these aspects we refer the reader to the recent note \cite{BGjga}. The reader more oriented toward analysis and pde's might in fact find somewhat surprising that we can develop the whole local regularity theory for solutions of the relevant heat equation starting from a global object such as the heat semigroup. By this we mean that, at the end of our process, we are able to replace the functions $P_t f$ in \eqref{srLY} with \emph{any} positive solution $u$ of the heat equation. This in a sense reverses the way one normally proceeds, starting from local solutions, and then moving from local to global. 

We are now ready to provide a brief account of our main results.  
\begin{itemize}
\item[1)] \emph{Li-Yau type inequalities} (Theorem \ref{T:ge}): assume Hypothesis \ref{A:exhaustion}, \ref{A:main_assumption} and \ref{A:regularity} hold. If $\M$ satisfies CD$(\rho_1,\rho_2,\kappa,d)$ in \eqref{cdi} with $\rho_1\in \R$, then for any  $f \in C_0^\infty(\bM)$,  $f  \ge 0$, $f \neq 0$, the following inequality holds for $t>0$:
\begin{align*}
 & \Gamma (\ln P_t f) +\frac{2 \rho_2}{3}  t \Gamma^Z (\ln P_t f) \\
  \le & \left(1+\frac{3\kappa}{2\rho_2}-\frac{2\rho_1}{3} t\right)
\frac{LP_t f}{P_t f} +\frac{d\rho_1^2}{6} t -\frac{d \rho_1}{2}\left(
1+\frac{3\kappa}{2\rho_2}\right) +\frac{d\left(
1+\frac{3\kappa}{2\rho_2}\right)^2}{2t}.
\end{align*}
\item[2)] \emph{Scale-invariant parabolic Harnack inequality} (Theorem \ref{T:harnack}): assume Hypothesis \ref{A:exhaustion}, \ref{A:main_assumption} and \ref{A:regularity}. If $\M$ satisfies CD$(\rho_1,\rho_2,\kappa,d)$ with $\rho_1 \ge 0$, then for every $(x,s), (y,t)\in \bM\times (0,\infty)$ with
$s<t$ one has 
\begin{equation*}
u(x,s) \le u(y,t) \left(\frac{t}{s}\right)^{\frac{D}{2}} \exp\left(
\frac{D}{d} \frac{d(x,y)^2}{4(t-s)} \right),
\end{equation*}
with $u(x,t) =
P_t f(x)$, and $f\in
C^\infty(\bM)$ such that $f\ge 0$ and bounded. The number $D>0$, which solely depends on $\rho_2, \kappa$ and $d$, is defined in  \eqref{D} below.
\item[3)] \emph{Off-diagonal Gaussian upper bounds} (Theorem \ref{T:ub}): assume Hypothesis \ref{A:exhaustion}, \ref{A:main_assumption} and \ref{A:regularity}. If $\bM$ satisfies CD$(\rho_1,\rho_2,\kappa,d)$ with $\rho_1 \ge 0$, then for any $0<\ve <1$
there exists a constant $C(\rho_2,\kappa,d,\ve)>0$, which tends
to $\infty$ as $\ve \to 0^+$, such that for every $x,y\in \bM$
and $t>0$ one has
\[
p(x,y,t)\le \frac{C(\rho_2,\kappa,d,\ve)}{\mu(B(x,\sqrt
t))^{\frac{1}{2}}\mu(B(y,\sqrt t))^{\frac{1}{2}}} \exp
\left(-\frac{d(x,y)^2}{(4+\ve)t}\right).
\]
\item[4)] \emph{Liouville type theorem} (Theorem \ref{T:liouville}): assume Hypothesis \ref{A:exhaustion}, \ref{A:main_assumption} and \ref{A:regularity}. If $\M$ satisfies CD$(\rho_1,\rho_2,\kappa,d)$ with $\rho_1 \ge 0$, then there exists no entire bounded solution of $Lf=0$. 
\item[5)]  \emph{Bonnet-Myers type theorem} (Theorem \ref{T:BM}): assume Hypothesis \ref{A:exhaustion}, \ref{A:main_assumption}, \ref{A:regularity}, and suppose that $\M$ satisfy CD$(\rho_1,\rho_2,\kappa,d)$ with $\rho_1 > 0$. 
Then, the
metric space $(\mathbb{M},d)$ is compact in the metric topology,  and we have \[ \text{diam}\ \bM \le 2\sqrt{3} \pi \sqrt{
\frac{\kappa+\rho_2}{\rho_1\rho_2} \left(
1+\frac{3\kappa}{2\rho_2}\right)d }.
\]
\end{itemize}
Concerning the Gaussian upper bound in 3), we mention that a similar bound was obtained for sub-Laplacians on Lie groups \cite{VSC}. Our approach  is totally different since it does not use the uniform doubling condition on the volume of the metric balls which is a key assumption in that work. We should also mention that in the sequel paper \cite{BBG} we have in fact established a uniform global doubling condition under non negative lower bound on the sub-Riemannian Ricci tensor $(\rho_1 \ge 0)$.

Concerning the sub-Riemannian Bonnet-Myers theorem in 5) we emphasize that, similarly to the Laplacian comparison theorem \eqref{lct}, the proof of its classical Riemannian predecessor is based on the theory of Jacobi fields. Our proof of Theorem \ref{T:BM} is, instead, purely analytical and exploits in a subtle way some sharp entropic inequalities which, in the case $\rho_1>0$, we are able to derive from the inequality CD$(\rho_1,\rho_2,\kappa,d)$ in \eqref{cdi}.

Having presented the main results of the paper, we now turn to the fundamental question of the examples. This aspect is dealt with in Section \ref{S:appendix}, which is devoted to constructing large classes of sub-Riemannian manifolds to which our general results apply. 
We begin with a discussion in Section \ref{SS:3dimmod} of a class of Lie groups which carry a natural CR structure, and which, in our framework, 
are the $3$-dimensional sub-Riemannian CR Sasakian model spaces  with constant curvature (see Hughen \cite{Hu} for a precise meaning of the notion of model spaces). Entropic inequalities on such model spaces were studied in \cite{bakry-baudoin}, and these inequalities constituted a first motivation for our theory. 

Given a $\rho_1\in \R$ we consider a Lie group $\bG(\rho_1)$ whose Lie algebra $\mathfrak g$ admits a basis of generators $X, Y, Z$ satisfying the commutation relations
\[
[X,Y] = Z,\ \ \ [X,Z] = - \rho_1 Y,\ \ \ [Y,Z] = \rho_1 X.
\]
The group $\bG(\rho_1)$ can be endowed with a natural CR structure $\theta$ with respect to which the Reeb vector field is given by $-Z$. A sub-Laplacian  on $\bG(\rho_1)$ with respect to such structure is thus given by $L = X^2 + Y^2$. The pseudo-hermitian Tanaka-Webster torsion of $\bG(\rho_1)$ vanishes (see Definition \ref{D:sasakian} below), and thus $(\bG(\rho_1),\theta)$ is a Sasakian manifold. In the smooth manifold $\bM = \bG(\rho_1)$ with sub-Laplacian $L$ we introduce the differential forms $\Gamma$ and $\Gamma^Z$  defined by 
\[
\Gamma(f,g) = Xf Xg + Yf Yg,\ \ \ \ \ \ \ \Gamma^Z(f,g) = Zf Zg.
\]
These forms satisfy the Hypothesis \eqref{A:exhaustion}, \eqref{A:main_assumption}. It is worth observing that, since $-Z$ is the Reeb vector field of the CR structure $\theta$, then the above choice of $\Gamma^Z$ is canonical. It is also worth remarking at this point that for the CR manifold $(\bG(\rho_1),\theta)$ the Tanaka-Webster horizontal sectional curvature is constant and equals $\rho_1$. Having noted these facts, in Section \ref{SS:3dimmod} we prove the following proposition. 

\begin{proposition}\label{P:Grho}
The sub-Laplacian $L$ on the Lie group $\mathbb{G}(\rho_1)$ satisfies the  generalized curvature-dimension inequality \emph{CD}$(\rho_1,\frac{1}{2},1,2)$.
\end{proposition}
The relevance of the model space $\bG(\rho_1)$ is illustrated by the Lie groups:
\begin{itemize}
\item[(i)] $\mathbb{SU}(2)$;
\item[(ii)] the ``flat'' Heisenberg group $\mathbb H^1$;
\item[(iii)]  $\mathbb{SL}(2,\R)$.
\end{itemize}
In Section \ref{SS:3dimmod} we note that the Lie groups (i)-(iii) are special instances of the model CR manifold $\bG(\rho_1)$ corresponding, respectively, to the cases $\rho_1 = 1,$ $\rho_1 = 0$ and $\rho_1 = -1$. 

After introducing these motivating examples, in Section \ref{SS:subR} we turn our attention to the construction of a large class of $C^\infty$ manifolds carrying a natural sub-Riemannian structure for which our generalized curvature-dimension inequality \eqref{cdi} holds. As a consequence, in these spaces all the above mentioned results 1)-6) are valid as well.  Let $\M$ be a smooth, connected  manifold equipped with a bracket generating distribution $\mathcal{H}$ of dimension $d$ and a fiberwise inner product $g$ on $\mathcal H$. The distribution $\mathcal{H}$ will be referred to as the set of \emph{horizontal directions}. 

We indicate with $\mathfrak{iso}$ the finite-dimensional Lie algebra of all sub-Riemannian Killing vector fields on $\bM$. It is readily seen that $Z\in \mathfrak{iso}$ if and only if:
 \begin{itemize}
 \item[(1)] For every $x\in \bM$, and any $u,v \in \mathcal{H}(x)$, $\mathcal{L}_Z g (u,v)=0$;
 \item[(2)] If $X\in \mathcal H$, then $[Z, X]\in \mathcal H$.
 \end{itemize}
 In (1) we have denoted by $\mathcal L_Z g$ the Lie derivative of $g$ with respect to $Z$.  Our main geometric assumption is the following:
 
 \begin{assumption}
 There exists a Lie sub-algebra $\mathcal{V} \subset \mathfrak{iso}$, such that for every $x \in \bM$, 
 \[
 T_x \bM= \mathcal{H}(x) \oplus \mathcal{V}(x).
 \]
 \end{assumption}
 
 The sub-bundle of transverse symmetries will be referred  to as the set of \emph{vertical directions}. The dimension of $\mathcal{V}$ will be denoted by $\di$. 

The horizontal distribution $\mathcal H$ with its fiberwise inner product $g$, plus the Lie algebra $\mathcal V$ are the essential data of the construction in Section \ref{SS:subR}. By this we mean that the relevant geometric objects that we introduce, namely the sub-Laplacian, the canonical connection $\nabla$ and the  tensor $\mathcal R$, respectively defined in Section \ref{SSS:cancon} and equation \eqref{Ro} in Definition \ref{D:RS} below,   
solely depend on $(\mathcal H,g)$ and $\mathcal V$, but not on the choice of an inner product on $\mathcal V$. As a consequence, in those situations in which the choice of $\mathcal V$ is canonical, then our analysis will depend only on the choice of $(\mathcal H,g)$. This is the case, for instance, of the basic example of Sasakian manifolds. 


Our ultimate objective in Section \ref{SS:subR} is proving that the smooth manifold $\bM$, with a given sub-Riemannian geometry $(\mathcal H,g)$ and a vertical distribution of transverse symmetries $\mathcal V$, satisfies a generalized curvature-dimension inequality such as \eqref{cdi} as soon as some intrinsic geometric conditions are satisfied. To achieve this objective we find it expedient introducing in Section \ref{SSS:cancon} a canonical connection $\nabla$. By means of such connection we define in Definition \ref{D:RS} a generalization of the Riemannian Ricci tensor, which we denote by $\mathcal R$. In Theorem \ref{T:bochner} we prove two Bochner identities which intertwine the tensor $\mathcal R$ with the forms $\Gamma$ and $\Gamma^Z$. With such Bochner identities in hand in Theorem \ref{T:cd} we finally show that, under the geometric assumptions in \eqref{riccibounds}, the manifold $\bM$ satisfies the generalized curvature-dimension inequality CD$(\rho_1,\rho_2,\kappa,d)$. In Proposition \ref{P:tres_beau} we prove that, remarkably, the generalized curvature-dimension inequality implies the geometric bounds \eqref{riccibounds}, and therefore: \emph{on any sub-Riemannian manifold with transverse symmetries we have} CD$(\rho_1,\rho_2,\kappa,d)$ $\Longleftrightarrow$ \eqref{riccibounds}.

The remaining part of Section \ref{S:appendix} is devoted to presenting some basic examples of manifolds which fall within the geometric framework of Section \ref{SS:subR}. In Section \ref{SS:carnot} we prove that all Carnot groups of step two satisfy the curvature-dimension inequality CD$(0,\rho_2,\kappa,d)$, for some appropriate values of $\rho_2$ and $\kappa$, see Proposition \ref{P:carnotCD}. Here, $d$ indicates the dimension of the bracket-generating layer of their Lie algebra. This result shows, in particular, that in our framework \emph{all} Carnot groups of step two are sub-Riemannian manifolds of nonnegative Ricci tensor, since $\rho_1 = 0$.  In Section \ref{SS:sasakian} we analyze another important class of manifolds which falls within the scope of our work, namely Sasakian manifolds endowed with their CR sub-Laplacian. These are CR manifolds of real hypersurface type for which the Tanaka-Webster pseudo-hermitian torsion vanishes in an appropriate sense. Concerning Sasakian manifolds we prove the following basic result.

\begin{theorem}\label{T:sasakiani}
Let $(\bM,\theta)$ be a complete \emph{CR} manifold  with real dimension $2n+1$ and vanishing Tanaka-Webster torsion, i.e., a Sasakian manifold. If 
for every $x\in \bM$ the Tanaka-Webster Ricci tensor satisfies the bound  
\[
\emph{Ric}_x(v,v)\ \ge \rho_1|v|^2,
\]
for every horizontal vector $v\in \mathcal H_x$,
then, for the \emph{CR} sub-Laplacian of $\bM$  the curvature-dimension inequality \emph{CD}$(\rho_1,\frac{d}{4},1,d)$ holds, with $d = 2n$ and the Hypothesis  \ref{A:exhaustion}, \ref{A:main_assumption} and \ref{A:regularity} are satisfied..
\end{theorem}

Thanks to this result, the above listed results 1)-5) are valid for \emph{all} Sasakian manifolds.

We close this introduction by mentioning that, for general metric measure spaces, a different notion of lower bounds on the Ricci tensor based on the theory of optimal transport has been recently proposed independently by Sturm \cite{sturm1}, \cite{sturm2}, and by Lott-Villani \cite{villani-lott}, see also the paper by  Ollivier \cite{ollivier}.  However, as pointed out by Juillet in \cite{juillet}, the remarkable theory developed in these papers does not appear to be suited for sub-Riemannian manifolds. For instance, in this theory the flat Heisenberg group $\mathbb H^1$ has curvature $= - \infty$.  In their preprint \cite{AL}  Agrachev and Lee have used a notion of Ricci tensor, denoted by $\mathfrak{Ric}$, which was introduced by the first author in \cite{A}. They study three-dimensional contact manifolds and, under the assumption that the manifold be Sasakian, they prove that a lower bound on $\mathfrak{Ric}$ implies the so-called measure-contraction property. In particular, when $\mathfrak{Ric} \ge 0$, then the manifold $\M$ satisfies a global volume growth similar to the Riemannian Bishop-Gromov theorem. An analysis shows that, interestingly, our notion of Ricci tensor coincides, up to a scaling factor, with theirs. 

We also mention that for three-dimensional contact manifolds, the sub-Riemannian geometric  invariants were computed by Hughen in his unpublished Ph.D. dissertation, see \cite{Hu}. In particular, with his notations, the  CR Sasakian structure corresponds to the case $a_1^2+a_2^2=0$ and, up to a scaling factor, his $K$ is the Tanaka-Webster Ricci curvature. In such respect, the Bonnet-Myers type theorem obtained  by Hughen (Proposition 3.5 in \cite{Hu})  is the exact analogue (with a better constant) of our Theorem \ref{T:BM}, applied to the case of three-dimensional  Sasakian manifolds. Let us finally mention that a Bonnet-Myers type theorem on general three-dimensional CR manifolds was first obtained by Rumin in \cite{rumin}. The methods of Rumin and Hughen are close as they both rely on the analysis of the second-variation formula for sub-Riemannian geodesics.

\


\

\textbf{Acknowledgments:} The authors would like to thank F.Y. Wang for pointing to our attention an oversight in a previous version of the paper. His constructive criticism has led us to improve the presentation and also add new results. We would also like to thank the anonymous referees for their careful reading of the manuscript and for several helpful comments.

\section{Examples}\label{S:appendix}

In this section we present several classes of sub-Riemannian spaces satisfying the generalized curvature-dimension inequality in Definition \ref{D:cdi}. These examples constitute the central motivation of the present work.

\subsection{Riemannian manifolds}

As we have mentioned in the introduction, when $\bM$ is a $n$-dimensional complete Riemannian manifold with Riemannian distance $d_R$, Levi-Civita connection $\nabla$ and Laplace-Beltrami operator $\Delta$, our main assumptions hold trivially. It suffices in fact to choose $\Gamma^Z = 0$ to satisfy Hypothesis \ref{A:main_assumption} in a trivial fashion. Hypothesis \ref{A:exhaustion} is also satisfied since it is equivalent to assuming that $(\bM,d_R)$ be complete, see \cite{GW} (observe in passing that the distance \eqref{di} coincides with $d_R$). Finally, with the choice $\kappa = 0$ the curvature-dimension inequality \eqref{cdi} reduces to \eqref{CDi}, which, as we have shown, is implied by (and it is in fact equivalent to) the assumption Ric $\ \ge \rho_1$.

\subsection{The three-dimensional Sasakian models}\label{SS:3dimmod}

The purpose of this section is providing a first basic sub-Riemannian example which fits the framework of the present paper. This example was first studied in \cite{bakry-baudoin}. Given a number $\rho_1\in \R$, suppose that $\bG(\rho_1)$ be a three-dimensional Lie group whose Lie algebra $\mathfrak{g}$ has a
 basis $\left\{ X,Y ,Z \right\}$ satisfying:
\begin{itemize}
\item[(i)] $[X,Y]=Z$,
\item[(ii)] $[X,Z]= -\rho_1 Y$,
\item[(iii)] $[Y,Z]=\rho_1 X$.
\end{itemize} 
A sub-Laplacian on  $\bG(\rho_1)$ is  the left-invariant, second-order differential operator
\begin{equation}\label{slmodel}
L= X^{2}  + Y^{2}.
\end{equation}
In view of (i)-(iii) H\"ormander's theorem, see \cite{Ho}, implies that $L$ be hypoelliptic, although it fails to be elliptic at every point of $\bG(\rho_1)$. 
From \eqref{gamma} we find in the present situation 
 \[
 \Gamma(f) =\frac{1}{2}\big(L(f^2)-2fLf)= (Xf)^{2}+ (Yf)^{2}.
 \] 
If we define 
\[
\Gamma^Z(f,g)=Zf Zg,
\]
then from (i)-(iii) we easily verify that
\[
\Gamma(f,\Gamma^Z(f)) = \Gamma^Z(f,\Gamma(f)).
\]
We conclude that the Hypothesis \ref{A:main_assumption} is satisfied. It is not difficult to show that the Hypothesis \ref{A:exhaustion} is also fulfilled.

Using (i)-(iii) we  leave it to the reader to verify that
\begin{equation}\label{commLZ} 
[L,Z] = 0.
\end{equation}
By means of \eqref{commLZ} we easily find 
\begin{align*}
\Gamma_2^Z(f) &= \frac 12 L(\Gamma^Z(f)) - \Gamma^Z(f,Lf) = Zf [L,Z]f + (XZf)^2 +(YZf)^2
\\
& =  (XZf)^2 +(YZf)^2.
\end{align*}
Finally, from definition \eqref{gamma2} and from (i)-(iii) we obtain
\begin{align*}\label{Gamma2_3f}
\Gamma_2 (f) & = \frac 12 L(\Gamma(f)) - \Gamma(f,Lf)
\\
& = \rho_1 \Gamma(f) + (X^2 f)^2 + (YXf)^2 + (XYf)^2 +(Y^2 f)^2
\\
& + 2 Yf (XZf) - 2 Xf (YZ f).
\end{align*}
We now notice that 
\[
(X^2f)^2 + (YXf)^2 + (XYf)^2 +(Y^2 f)^2 = ||\nabla^2_H f||^2 + \frac 12 \Gamma^Z(f),
\]
where we have denoted by 
\[
\nabla^2_H f= \begin{pmatrix} X^2 f & \frac 12 (XY f + YXf)
\\
\frac 12 (XY f + YXf) & Y^2 f
\end{pmatrix}
\]
the symmetrized Hessian of $f$ with respect to the horizontal distribution generated by $X, Y$. 
Substituting this information in the above formula we find
\[
\Gamma_2 (f)  =  ||\nabla^2_H f||^2 + \rho_1 \Gamma(f) + \frac 12 \Gamma^Z(f) + 2 \big(Yf (XZf) - Xf (YZ f)\big).
\]
By the above expression for $\Gamma_2^Z(f)$, using Cauchy-Schwarz inequality, we obtain for every $\nu >0$
\[
|2 Yf (XZf) - 2 Xf (YZ f)| \le \nu \Gamma^Z_2(f) + \frac 1\nu \Gamma(f).
\]
Similarly, one easily recognizes that
\[
||\nabla^2_H f||^2 \ge \frac 12 (Lf)^2.
\]
Combining these inequalities,
we conclude that we have proved the following result. 
\begin{proposition}\label{P:cdmodels}
For every $\rho_1\in \R$ the Lie group $\mathbb{G}(\rho_1)$, with the sub-Laplacian $L$ in \eqref{slmodel}, satisfies the generalized  curvature dimension inequality \emph{CD}$(\rho_1,\frac{1}{2},1,2)$. Precisely, for every $f\in C^\infty(\bG(\rho_1))$ and any $\nu>0$ one has:
 \[
 \Gamma_{2}(f)+\nu \Gamma^Z_{2}(f) \ge \frac{1}{2} (Lf)^2 +\left(\rho_1 -\frac{1}{\nu}\right)  \Gamma (f)
 +\frac{1}{2} \Gamma^Z (f).
 \]
\end{proposition}

Proposition \ref{P:cdmodels} provides a basic motivation for Definition \ref{D:cdi}. It is also important to observe at this point that the Lie group $\bG(\rho_1)$ can be endowed with a natural CR structure. Denoting in fact with $\mathcal H$ the subbundle of $T\bG(\rho_1)$ generated by the vector fields $X$ and $Y$, the endomorphism $J$ of $\mathcal H$ defined by 
\[
J(Y) = X,\ \ \ \ J(X) = - Y,
\]
satisfies $J^2 = - I$, and thus defines a complex structure on $\bG(\rho_1)$. By choosing $\theta$ as the form such that \[
\text{Ker}\ \theta = \mathcal H,\   \  \ \text{and}\ \  \  d\theta(X,Y) = 1,
\]
we obtain a CR structure on $\bG(\rho_1)$ whose Reeb vector field is $-Z$. Thus, the above choice of $\Gamma^Z$ is canonical.

The pseudo-hermitian Tanaka-Webster torsion of $\bG(\rho_1)$ vanishes (see Definition \ref{D:sasakian} below), and thus $(\bG(\rho_1),\theta)$ is a Sasakian manifold. It is also easy to verify that for the CR manifold $(\bG(\rho_1),\theta)$ the Tanaka-Webster horizontal sectional curvature is constant and equals $\rho_1$. 
The following three model spaces correspond respectively to the cases $\rho_1 = 1, \rho_1 = 0$ and $\rho_1 = -1$.  

\begin{example}
The Lie group $\mathbb{SU} (2)$ is the group of
$2 \times 2$, complex, unitary matrices of determinant $1$. Its
Lie algebra $\mathfrak{su} (2)$ consists of $2 \times 2$, complex,
skew-hermitian matrices with trace $0$. A basis of $\mathfrak{su} (2)$
is formed by the following matrices  $X = \frac{i}{2} \sigma_1$, $Y = \frac i2 \sigma_2$, $Z = \frac i2 \sigma_3$, where $\sigma_k$, $k=1,2,3$, are the Pauli matrices:
 \[
\text{ }X=\frac{1}{2}\left(
\begin{array}{cc}
~0~ & ~1~ \\
-1~& ~0~
\end{array}
\right),\ \text{ }Y=\frac{1}{2}\left(
\begin{array}{cc}
~0~ & ~i~ \\
~i~ & ~0~
\end{array}
\right),\ 
Z=\frac{1}{2}\left(
\begin{array}{cc}
~i~ & ~0~ \\
~0~ & -i~
\end{array}
\right).
\]
One easily verifies 
\begin{align}\label{Liestructure}
[X,Y]=Z, \quad [X,Z]=-Y, \quad [Y,Z]=X,
\end{align}
and thus $\rho_1 = 1$.
\end{example}

\begin{example}
The Heisenberg group $\mathbb{H}$ is the group of $3\times3$ matrices:
\[
\left(
\begin{array}
[c]{ccc}
~1~ & ~x~   & ~z ~\\
~0~ & ~1~   & ~y ~\\
~0~ & ~0~   & ~1 ~
\end{array}
\right)  ,\text{ \ }x,y,z\in\mathbb{R}.
\]
The Lie algebra of $\mathbb{H}$ is spanned by the matrices
\[
X=\left(
\begin{array}
[c]{ccc}
~0~ & ~1~ & ~0~\\
~0~ & ~0~ & ~0~\\
~0~ & ~0~ & ~0~
\end{array}
\right),\ Y=\left(
\begin{array}
[c]{ccc}%
~0~ & ~0~ & ~0~\\
~0~ & ~0~ & ~1~\\
~0~ & ~0~ & ~0~
\end{array}
\right),\  Z=\left(
\begin{array}
[c]{ccc}%
~0~ & ~0~ & ~1~\\
~0~ & ~0~ & ~0~\\
~0~ & ~0~ & ~0~
\end{array}
\right)  ,
\]
for which the following commutation relations hold
\[
\lbrack X,Y]=Z,\text{ }[X,Z]=[Y,Z]=0.
\]
We thus have $\rho_1 = 0$ in this case.
\end{example}

\begin{example}
The Lie group $\mathbb{SL} (2)$ is the group of
$2 \times 2$, real matrices of determinant $1$. Its
Lie algebra $\mathfrak{sl} (2)$ consists of $2 \times 2$ matrices of trace $0$. 
A basis of $\mathfrak{sl} (2)$ is formed by the matrices:
 \[
X=\frac{1}{2}\left(
\begin{array}{cc}
~1~ & ~0~ \\
~0~ & -1~
\end{array}
\right) 
,\text{ }Y=\frac{1}{2}\left(
\begin{array}{cc}
~0~ & ~1~ \\
~1~ & ~0~
\end{array}
\right),
\text{ }Z=\frac{1}{2}\left(
\begin{array}{cc}
~0~ & ~1~ \\
-1~& ~0~
\end{array}
\right) 
,
\]
for which the following commutation relations hold
\begin{align}\label{Liestructure}
[X,Y]=Z, \quad [X,Z]=Y, \quad [Y,Z]=-X.
\end{align}
We thus have $\rho_1 = -1$ in this case.
\end{example}

\subsection{Sub-Riemannian manifolds with transverse symmetries}\label{SS:subR}

We now turn our attention to a large class of sub-Riemannian manifolds, encompassing the three-dimensional model spaces discussed in the previous section. The central objective of the present section is proving Theorem \ref{T:cd} below. The latter states that for these sub-Riemannian manifolds the generalized curvature-dimension inequality \eqref{cdi} does hold under some natural geometric assumptions which, in the Riemannian case, reduce to requiring a lower bound for the Ricci tensor. To achieve this result, we will need to establish some new Bochner type identities. This is done in Theorem \ref{T:bochner} below. 

Let $\M$ be a smooth, connected  manifold. We assume that $\bM$ is equipped with a bracket generating distribution $\mathcal{H}$ of dimension $d$ and a fiberwise inner product $g$ on that distribution. The distribution $\mathcal{H}$ will be referred to as the set of \emph{horizontal directions}. 

We indicate with $\mathfrak{iso}$ the finite-dimensional Lie algebra of all sub-Riemannian Killing vector fields on $\bM$ (see \cite{Strichartz}). A vector field $Z\in \mathfrak{iso}$ if the one-parameter flow generated by it locally preserves the sub-Riemannian geometry defined by $(\mathcal{H},g )$. This amounts to saying that:
 \begin{itemize}
 \item[(1)] For every $x\in \bM$, and any $u,v \in \mathcal{H}(x)$, $\mathcal{L}_Z g (u,v)=0$;
 \item[(2)] If $X\in \mathcal H$, then $[Z, X]\in \mathcal H$.
 \end{itemize}
 In (1) we have denoted by $\mathcal L_Z g$ the Lie derivative of $g$ with respect to $Z$.  Our main geometric assumption is the following:
 
 \begin{assumption}
 There exists a Lie sub-algebra $\mathcal{V} \subset \mathfrak{iso}$, such that for every $x \in \bM$, 
 \[
 T_x \bM= \mathcal{H}(x) \oplus \mathcal{V}(x).
 \]
 \end{assumption}
 
 The distribution $\mathcal{V}$ will be referred  to as the set of \emph{vertical directions}. The dimension of $\mathcal{V}$ will be denoted by $\di$. 

The choice of an inner product on the Lie algebra $\mathcal{V}$ naturally endows $\bM$ with a Riemannian extension $g_R$ of $g$ that makes the decomposition $\mathcal{H}(x) \oplus \mathcal{V}(x)$ orthogonal. Although  $g_R$ will be useful for the purpose of computations, the geometric objects that we will introduce, like the sub-Laplacian $L$, the canonical connection $\nabla$ and the "Ricci" tensor $\mathcal R$, ultimately will not depend  on the choice of an inner product on $\mathcal{V}$. 

The Riemannian measure of $(\bM,g_R)$ will  be denoted by $\mu$ and, for notational convenience, we will often use the notation $\langle \cdot, \cdot \rangle$ instead of $g_R$.

 \begin{remark}
 If the  Lie group $\mathbb{V}$ generated by $\mathcal{V}$ acts properly on $\bM$, then we have a natural Riemannian submersion  $\bM \to \bM / \mathbb{V}$. In the case $\mathbb{SU}(2)$ studied in the previous section, we obtain the Hopf fibration $\mathbb{S}^3 \to \mathbb{S}^2$, see \cite{Montgomery}.
 \end{remark}
 
 The above assumptions imply that, in a sufficiently small neighborhood of every point $x\in \bM$, we can find a frame of vector fields $\{X_1,\cdots,X_d, Z_1, \cdots, Z_\di\}$ such that: 
 \begin{itemize}
 \item[(a)] $Z_1, \cdots, Z_\di \in \mathcal{V}$;
\item[(b)] $\{X_1(x),\cdots,X_d(x)\}$ is an orthonormal basis of $\mathcal{H}(x)$;
\item[(c)] $\{Z_1(x), \cdots, Z_\di(x)\}$ is an orthonormal basis of $\mathcal{V}(x)$;
\item[(d)]  the following commutation relations hold:
\begin{equation}\label{bra1}
[X_i,X_j]=\sum_{\ee=1}^d \omega_{ij}^\ee X_\ee +\sum_{m=1}^{\di}
\gamma_{ij}^{m} Z_{m},
\end{equation}
\begin{equation}\label{bra2} 
[X_i,Z_{m}]=\sum_{\ee=1}^d \delta_{im}^\ee X_\ee,
\end{equation}
\end{itemize}
for 
smooth functions $ \omega_{ij}^\ee $, $ \gamma_{ij}^{m}$ and
$\delta_{im}^\ee $ such that
\begin{equation}\label{deltas}
\delta_{im}^\ee=-\delta_{\ee m}^i,\ \ i, \ee = 1,...,d,\
\text{and}\ m=1,...,\di.
\end{equation}
We mention explicitly that the equation \eqref{deltas} follows from the property of $Z_m$ being sub-Riemannian Killing, see conditions $(1)$ and $(2)$ above.
By convention,
$\omega_{ij}^\ee =-\omega_{ji}^\ee$ and
$\gamma_{ij}^{m}=-\gamma_{ji}^{m}.$ 

\begin{definition}\label{D:adapted}
A local frame such as in (a)-(d) above will  be called an \emph{adapted frame}.
\end{definition}
 
We define the horizontal gradient $\nabla_\mathcal{H} f$ of a function $f$ as the projection of the Riemannian gradient of $f$ on the horizontal bundle. Similarly, we define the vertical gradient $\nabla_\mathcal{V} f$ of a function $f$ as the projection of the Riemannian gradient of $f$ on the vertical bundle.
In an adapted frame,
\[
\nabla_\mathcal{H} f=\sum_{i=1}^d (X_i f) X_i,
\]
\[
\nabla_\mathcal{V} f =\sum_{m=1}^\di (Z_m f) Z_m.
\]
The canonical sub-Laplacian in this structure is, by definition, the diffusion operator $L$ on $\bM$ which is symmetric  on $C^\infty_0 (\bM)$ with respect  to the measure $\mu$ and such that (see \eqref{gamma}):
\[
\Gamma(f,g) = \frac{1}{2}(L(fg)-fL g-gL f)= \langle \nabla_\mathcal{H} f , \nabla_\mathcal{H} g \rangle.
\]
It is readily seen that in an adapted frame, one has
\begin{equation*}
L= -\sum_{i=1}^d X_i^*X_i,
\end{equation*}
where $X_i^*$ is the formal adjoint of $X_i$ with respect to the measure $\mu$. From the commutation relations in an adapted frame, we obtain that
\[
X_i^*=-X_i+\sum_{k=1}^d \omega_{ik}^k,
\]
so that, in an adapted frame
\begin{align}\label{L}
L=\sum_{i=1}^d X_i^2 +X_0,
\end{align}
where
\begin{equation}\label{X0}
X_0=-\sum_{i,k=1}^d \omega_{ik}^k X_i.
\end{equation}
We also note that since $\mathcal{H}$ is supposed to be bracket generating, from H\"ormander's theorem, $L$ is  a hypoelliptic operator.

In the present setting, from the very definition of $L$, one readily recognizes that the canonical bilinear form introduced in \eqref{gamma} above is given by
\[
\Gamma(f,g)= \langle \nabla_\mathcal{H} f , \nabla_\mathcal{H} g \rangle.
\]

\begin{definition}\label{D:gammaZ}
We define for every $f, g\in C^\infty(\bM)$
\[
\Gamma^Z(f,g)= \langle \nabla_\mathcal{V} f , \nabla_\mathcal{V} g \rangle.
\]
\end{definition}
Our first step is verifying that the differential forms $\Gamma$ and $\Gamma^Z$ satisfy the Hypothesis \ref{A:main_assumption} in the introduction. This is  the content of the next result.

\begin{lemma} For $f,g \in C^\infty(\bM)$,
 \[
 \Gamma( f, \Gamma^Z(f))=\Gamma^Z( f, \Gamma(f)).
 \]
 \end{lemma}
 
 \begin{proof}
 It is readily checked in a local adapted frame $\{X_1,...,X_d,Z_1,...,Z_\di\}$.
 We have
 \begin{align*}
 \Gamma^Z(f,\Gamma(f)) & = 2 \sum_{m = 1}^\di Z_m f \sum_{i=1}^d X_i f Z_m(X_i f) 
 \\
 & = 2 \sum_{i=1}^d X_i f \sum_{m = 1}^\di Z_m f X_i(Z_m f) - 2  \sum_{i=1}^d X_i f \sum_{m = 1}^\di Z_m f [X_i,Z_m] f
 \\
 & = \Gamma(f,\Gamma^Z(f)) - 2 \sum_{m = 1}^\di Z_m f \sum_{i,\ell=1}^d \delta_{im}^\ell X_i f X_\ell f
 \\
 & = \Gamma(f,\Gamma^Z(f)),
 \end{align*}
 where in the last two equalities we have used \eqref{bra2} and \eqref{deltas}.

 \end{proof}

Another property that will be important for us is that  $\mathcal{V}$ is a Lie algebra of symmetries for the sub-Laplacian $L$.
\begin{lemma}\label{L:comm}
For any $Z \in \mathcal{V}$ one has $[L,Z] = 0$.
\end{lemma}

\begin{proof}
Since $Z$ is a Killing vector field, $[L,Z] $ is a first-order differential operator and therefore a vector field. Since $Z^*=-Z+c$, where $Z^*$ denotes the formal adjoint of $Z$ and $c$ a constant, we obtain that $[L,Z]^*=[L,Z]$. Since a symmetric vector field must vanish identically, we obtain the desired conclusion.
\end{proof}

\subsubsection{\textbf{The canonical connection}}\label{SSS:cancon}

Our ultimate objective (see Theorem \ref{T:cd} in Section \ref{SSS:generalizedCD}) will be establishing natural geometric conditions under which the manifold $\bM$, endowed with the above defined sub-Laplacian $L$, and with the differential bilinear form $\Gamma^Z$, satisfy the generalized curvature-dimension inequality CD$(\rho_1,\rho_2,\kappa,d)$ in Definition \ref{D:cdi}. A useful ingredient in the realization of this objective is the existence of a canonical connection on $\bM$.

\begin{proposition}\label{P:horconn}
There exists a unique affine connection $\nabla$ on $\mathbb{M}$ satisfying the following properties:
\begin{itemize}
\item[(i)] $\nabla g =0$;
\item[(ii)] if $X$ and $Y$ are horizontal vector fields, $\nabla_{X} Y$ is horizontal;
\item[(iii)] if $Z \in \mathcal{V}$, $\nabla Z=0$;
\item[(iv)] if $X,Y$ are horizontal vector fields and $Z \in \mathcal{V}$,
the torsion vector field \emph{T}$(X,Y)$ is vertical and \emph{T}$(X, Z)=0$.
\end{itemize}
\end{proposition}

\begin{proof}
If we indicate with $\nabla^R$ the Riemannian Levi-Civita connection on $\bM$, the existence of the connection $\nabla$ follows by prescribing the relations
\[
\nabla_Z X=[Z,X], \quad  \nabla_X Y=\pi_\mathcal{H} (\nabla^R_X Y), \quad  \nabla Z=0,
\]
where $X,Y\in \mathcal H$, $Z \in \mathcal{V}$,  and $\pi_\mathcal{H}$ the projection onto the horizontal bundle. The uniqueness of $\nabla$ follows in a standard fashion.

\end{proof}

\begin{remark}
It is worth noting that the connection $\nabla$ does not depend on the choice of  the inner product on $\mathcal{V}$.
\end{remark}

\begin{remark}\label{R:riemsasak}
It is also worth observing here that in the Riemannian case we simply have $\mathcal H = T\bM$, and $\nabla$ is just the Levi-Civita connection on $\bM$.
\end{remark}

\begin{remark}\label{R:nablas}
For later use we observe that, in a local adapted frame, one has:
\begin{equation}\label{christoffel1}
\nabla_{X_i} X_j= \sum_{k=1}^d
\frac{1}{2} \left(\omega_{ij}^k +\omega_{ki}^j +\omega_{kj}^i
\right) X_k,
\end{equation}
\begin{equation}\label{christoffel2}
\nabla_{Z_{m}} X_i=-\sum_{\ee=1}^d \delta_{im}^\ee X_\ee,
\end{equation}
\begin{equation}\label{christoffel3}
\nabla Z_{m}=0.
\end{equation}
\end{remark}

We also note that, thanks to \eqref{L} and \eqref{X0}, in a local adapted frame we have
\[
L=\sum_{i=1}^d X^2_i-\nabla_{X_i} X_i,
\]
so that
\[
Lf=\text{div}(\nabla_\mathcal{H} f).
\]

\subsubsection{\textbf{Generalized Bochner identities}}\label{SS:bochner}

As we have recalled in the opening of the present paper at the hearth of the Riemannian curvature-dimension inequality CD$(\rho_1,n)$ there is the Bochner identity. It is then only natural that our first step in the formulation of the generalized curvature-dimension inequality in Definition \ref{D:cdi} above was understanding appropriate versions of the identity of Bochner. This is accomplished by Theorem \ref{T:bochner} below, which represents the central result of this section. This result contains two Bochner identities: one for the horizontal directions, see \eqref{bochner}, and the other for the vertical ones, see \eqref{vbf} below. One of the essential points of the program laid in this paper is that, to formulate a notion of Ricci that works well for sub-Riemannian spaces, one needs to appropriately intertwine these identities. As a final comment we mention that, as it will be clear from the proof of Theorem \ref{T:bochner}, the vertical Bochner formula is incredibly easier than the horizontal one, but this is in the nature of things, and should come as no surprise.  

We are ready to introduce the relevant geometric quantities. 

\begin{definition}\label{D:RS}
Let $\nabla$ be the affine connection introduced by Proposition \ref{P:horconn}, and indicate with \emph{Ric} and \emph{T} respectively the Ricci and torsion tensors with respect to $\nabla$. 
For $f \in C^\infty(\bM)$ we define:
\begin{align}\label{Ro}
\mathcal{R} (f)=  \emph{Ric}(\nabla_\mathcal{H} f , \nabla_\mathcal{H} f) +\sum_{\ee,k=1}^d \left( -( (\nabla_{X_\ee} \emph{T}) (X_\ee,X_k)f )(X_k f)  +\frac{1}{4}
\left( \emph{T}(X_\ee,X_k)f\right)^2 \right).
\end{align}
where $\{X_1,\cdots, X_d\}$ is a local frame of horizontal vector fields. We also define the following second-order differential form  by the formula:
\begin{align}\label{S}
\mathcal{S}(f)=-2\sum_{i=1}^d \langle \nabla_{X_i} \nabla_\mathcal{V} f, \emph{T}(X_i, \nabla_\mathcal{H} f)\rangle.
\end{align}
\end{definition}

\begin{remark}\label{R:RS}
The expressions \eqref{Ro}, \eqref{S} do not depend on the choice of the frame, thus they define intrinsic differential forms on $\bM$. Also, we observe that since the connection $\nabla$ does not depend on the choice of an inner product on $\mathcal{V}$, it is easy to check that $ \mathcal{R}$  and $\mathcal{S}$ do not depend on this choice either. We note explicitly that in the Riemannian case we have $\mathcal H = TM$, $\nabla$ is just the Levi-Civita connection of $\bM$, and therefore \emph{T}$\equiv 0$. In such case, $\mathcal R(f) =$\ \emph{Ric}$(\nabla f,\nabla f)$, where now \emph{Ric} is the Riemannian Ricci tensor.
\end{remark}

The following lemma provides a useful expression of the differential forms $\mathcal R(f)$ and $\mathcal S(f)$ in a local adapted frame.

\begin{lemma}\label{L:RS}
Let $\{X_1,...,X_d,Z_1,...,Z_\di\}$ be a local adapted frame. Then, we have:
\begin{align}\label{R2}
\mathcal R(f) & = \sum_{k,\ee=1}^d \bigg\{\bigg(\sum_{j=1}^d
\sum_{m=1}^\di \gamma_{kj}^{m} \delta_{jm}^\ee\bigg) +
\sum_{j=1}^d
(X_\ee\omega^j_{kj} - X_j\omega^k_{\ee j}) \\
& + \sum_{i,j=1}^d \omega_{ji}^i \omega^\ee_{k j} - \sum_{i=1}^d
\omega_{k i}^i \omega_{\ee i}^i + \frac{1}{2} \sum_{1\le i<j\le d}
\bigg(\omega^\ee_{ij} \omega^k_{ij} - (\omega_{\ee j}^i +\omega_{\ee
i}^j)(\omega^i_{kj} + \omega^j_{ki})\bigg)\bigg\}X_k f X_\ee f
\notag\\
&   +  \sum_{k=1}^d \sum_{m=1}^\di \bigg(\sum_{\ell,j=1}^d
\omega_{j\ee}^\ee \gamma_{kj}^{m} + \sum_{1\le \ee<j\le d}
\omega^k_{\ee j} \gamma^{m}_{\ee j}  - \sum_{j=1}^d X_j
\gamma^{m}_{kj}\bigg) Z_{m} f X_k f \notag\\
& + \frac{1}{2} \sum_{1\le \ee<j\le d}\bigg(\sum_{m=1}^\di
\gamma^{m}_{\ee j} Z_{m} f\bigg)^2, \notag
\end{align}
and
\begin{equation}\label{S2}
\mathcal{S}(f)=-2 \sum_{i,j=1}^d \sum_{m=1}^\di \gamma_{ij}^{m} (X_j Z_{m} f) (X_i f).
\end{equation}
\end{lemma}

\begin{proof}
It is a standard but lengthy computation using an adapted frame.

\end{proof}

In the following we denote by $\| \nabla_\mathcal{H}^2 f \|^2$ the Hilbert-Schmidt norm of the symmetrized horizontal Hessian of a function $f$. In a local adapted frame
\[
\| \nabla_\mathcal{H}^2 f \|^2 = \sum_{\ee=1}^d \left( X^2_\ee f-\sum_{i=1}^d
\omega_{i\ee}^\ee X_i f \right)^2 +  2 \sum_{1 \le \ee<j \le d}
\left( \frac  {X_j X_\ee + X_\ee X_j}{2}f -\sum_{i=1}^d \frac{\omega_{ij}^\ee
+\omega_{i\ee}^j}{2} X_i f \right)^2.
\]
Also, we will denote $\| \nabla_\mathcal{H} \nabla_\mathcal{V} f \|^2=\sum_{i=1}^d \sum_{m=1}^{\di}  (X_i Z_m f)^2$, an expression which is seen to be independent from the local adapted frame.
The next theorem constitutes one of the central results of Section \ref{SS:subR}.

\begin{theorem}\label{T:bochner}
For every $f\in C^\infty(\bM)$ the following formulas hold:
\begin{align}\label{bochner}
\Gamma_{2}(f)=\| \nabla_\mathcal{H}^2 f \|^2 +\mathcal{R}(f)+\mathcal{S}(f);\ \ \ \ \emph{(Horizontal Bochner formula)}
\end{align}
\begin{equation}\label{vbf}
\Gamma^Z_{2}(f) = \| \nabla_\mathcal{H} \nabla_\mathcal{V} f \|^2.  \ \ \ \ \ \ \ \ \ \ \ \emph{(Vertical Bochner formula)}
\end{equation}
\end{theorem}

\begin{proof}
It is enough to prove \eqref{bochner} and \eqref{vbf} in a local adapted frame $\{X_1,...,X_d,Z_1,...,Z_\di\}$. We begin with the vertical Bochner formula \eqref{vbf}, which is quite simple. Such formula follows immediately by a direct computation starting from the definition \eqref{gamma2Z} of $\Gamma^Z_2$, and using the fact that $L$ and $Z_m$ commute, see Lemma \ref{L:comm}.

The proof of the horizontal Bochner formula \eqref{bochner} is not as straightforward. In order to avoid long and cumbersome computations we will omit the intermediate details and only provide the essential identities. With such identities the interested reader should be able to fill in the gaps.
Let us preliminarily observe that 
\[
X_i X_j f = f_{,ij} + \frac 12 [X_i,X_j]f,
\]
where we have let
\begin{equation}\label{symmder}
f_{,ij} = \frac 12 (X_i X_j + X_j X_i) f.
\end{equation}
Using \eqref{bra1}, we find 
\begin{equation}\label{nonsym}
X_iX_jf = f_{,ij} + \frac{1}{2} \sum_{\ee=1}^d \omega^\ee_{ij} X_\ee
f + \frac{1}{2} \sum_{m=1}^\di \gamma^{m}_{ij} Z_{m} f.
\end{equation}
Now, starting from the definition \eqref{gamma2} of $\Gamma_2(f)$, we obtain
\begin{align*}
\Gamma_2(f) & =  \sum_{i=1}^d X_i f [X_0,X_i]f - 2
\sum_{i,j=1}^d X_if [X_i,X_j]X_j f 
\\
& + \sum_{i,j=1}^d X_i f
[[X_i,X_j],X_j]f  +  \sum_{i,j=1}^d (X_jX_i f)^2,
\end{align*}
where $X_0$ is defined by \eqref{X0}. From \eqref{nonsym} we have
\begin{align*}
\sum_{i,j=1}^d (X_jX_i f)^2  & =  \sum_{i,j=1}^d f_{,ij}^2 + \frac{1}{2} \sum_{1\le
i<j\le d}\left(\sum_{\ee=1}^d \omega^\ee_{ij} X_\ee f\right)^2 +
\frac{1}{2} \sum_{1\le i<j\le d}\left(\sum_{m=1}^\di
\gamma^{m}_{ij} Z_{m} f\right)^2 \\
& + \sum_{1\le i<j\le d}\sum_{\ee=1}^d \sum_{m=1}^\di
\omega^\ee_{ij} \gamma^{m}_{ij} Z_{m}f X_\ee f. 
\notag
\end{align*}
and therefore, 
\begin{align}\label{bochner1}
\Gamma_2(f) & =  \sum_{i,j=1}^d f_{,ij}^2- 2
\sum_{i,j=1}^d X_if [X_i,X_j]X_j f + \sum_{i,j=1}^d X_i f
[[X_i,X_j],X_j]f 
\\
& + \sum_{i=1}^d X_i f [X_0,X_i]f  + \frac{1}{2} \sum_{1\le
i<j\le d}\left(\sum_{\ee=1}^d \omega^\ee_{ij} X_\ee f\right)^2 +
\frac{1}{2} \sum_{1\le i<j\le d}\left(\sum_{m=1}^\di
\gamma^{m}_{ij} Z_{m} f\right)^2 
\notag\\
& + \sum_{1\le i<j\le d}\sum_{\ee=1}^d \sum_{m=1}^\di
\omega^\ee_{ij} \gamma^{m}_{ij} Z_{m}f X_\ee f. 
\notag
\end{align}
To complete the proof we need to
recognize that the right-hand side in \eqref{bochner1} coincides
with that in \eqref{bochner}.
 With this objective in mind, using
\eqref{bra1}  we obtain after a computation
\begin{align*}
& \sum_{i,j=1}^d f_{,ij}^2  - 2 \sum_{i,j=1}^d X_i f [X_i,X_j]X_j f
\\
& =  \sum_{\ell=1}^d \left(f_{,\ell\ell}^2 - 2 \left(\sum_{i=1}^d
\omega^\ee_{i\ee} X_i f\right) f_{,\ee \ee}\right)
\\
& + 2  \sum_{1 \le \ee<j \le d}\left( f_{,j\ell}^2  - 2 \sum_{1\le
\ee<j\le d} \left(\sum_{i=1}^d \frac{\omega_{ij}^\ee +
\omega_{i\ee}^j}{2}
X_i f\right) f_{,\ee j} \right)
\\
&  -  \sum_{i,j=1}^d \sum_{\ee, k=1}^d \omega_{ij}^\ee \omega^k_{\ee
j} X_k f X_i f -  \sum_{i,j=1}^d \sum_{\ee=1}^d \sum_{m=1}^\di
\omega_{ij}^\ee \gamma^{m}_{\ee j} Z_{m} f\ X_i f
\\
& - 2 \sum_{i,j=1}^d  \sum_{m=1}^\di \gamma_{ij}^{m} Z_{m}X_j f\
X_i f.
\end{align*}
Completing the squares in the latter expression we find
\begin{align*}
& \sum_{i,j=1}^d f_{,ij}^2  - 2 \sum_{i,j=1}^d X_i f [X_i,X_j]X_j f
\\
& = \sum_{\ee=1}^d \left( f_{,\ell\ell} -\sum_{i=1}^d
\omega_{i\ee}^\ee X_i f \right)^2 +  2 \sum_{1 \le \ee<j \le d}
\left( f_{,j\ell} -\sum_{i=1}^d \frac{\omega_{ij}^\ee
+\omega_{i\ee}^j}{2} X_i f \right)^2
\notag\\
& - \sum_{\ee=1}^d \left(\sum_{i=1}^d \omega_{i\ee}^\ee X_i f
\right)^2 -  2 \sum_{1 \le \ee<j \le d} \left(\sum_{i=1}^d
\frac{\omega_{ij}^\ee +\omega_{i\ee}^j}{2} X_i f \right)^2
\notag\\
&  -  \sum_{i,j,k,\ee=1}^d  \omega_{ij}^\ee \omega^k_{\ee j} X_k f
X_i f -  \sum_{i,j=1}^d \sum_{\ee=1}^d \sum_{m=1}^\di
\omega_{ij}^\ee \gamma^{m}_{\ee j} Z_{m} f\ X_i f
\notag\\
& - 2 \sum_{i,j=1}^d \sum_{m=1}^\di \gamma_{ij}^{m} X_j Z_{m}f\
X_i f - 2 \sum_{i,j=1}^d  \sum_{m=1}^\di \gamma_{ij}^{m}
[Z_{m},X_j] f\ X_i f. \notag
\end{align*}
Next, we have from \eqref{X0}
\begin{align*}
\sum_{i=1}^d X_i f [X_0,X_i]f  & =  \sum_{i,j,k,\ee=1}^d
\omega^k_{jk} \omega^\ee_{ij} X_\ee f X_i f
\\
& + \sum_{i=1}^d \sum_{j,k=1}^d \sum_{m=1}^\di \omega_{jk}^k
\gamma_{ij}^{m} Z_{m} f  X_i f + \sum_{i=1}^d \sum_{j,k=1}^d (X_i
\omega^k_{jk}) X_if X_j f. 
\notag
\end{align*} 
Using \eqref{bra1} we find
\begin{align*}
\sum_{i,j=1}^d X_i f [[X_i,X_j],X_j]f  & = \sum_{i,j=1}^d
\sum_{\ee,k=1}^d \omega^\ee_{ij} \omega^k_{\ee j} X_i f X_kf +
\sum_{i,j,\ee=1}^d \sum_{m=1}^\di \omega^\ee_{ij} \gamma^{m}_{\ee
j} Z_{m}f X_i f
\\
& + \sum_{i,j=1}^d \sum_{m=1}^\di \gamma^{m}_{ij} [Z_{m},X_j]f
X_i f - \sum_{i,j=1}^d \sum_{m=1}^\di (X_j \gamma^{m}_{ij})
Z_{m} f
X_i f \notag\\
& - \sum_{i,j=1}^d \sum_{\ee=1}^d (X_j\omega^\ee_{ij}) X_if X_\ee f.
\notag
\end{align*}
Substituting the latter three equations in \eqref{bochner1} we thus obtain 
\begin{align*}
\Gamma_2(f) & = \sum_{\ee=1}^d \left( f_{,\ell\ell} -\sum_{i=1}^d
\omega_{i\ee}^\ee X_i f \right)^2 +  2 \sum_{1 \le \ee<j \le d}
\left( f_{,j\ell} -\sum_{i=1}^d \frac{\omega_{ij}^\ee
+\omega_{i\ee}^j}{2} X_i f \right)^2
\\
& - 2 \sum_{i,j=1}^d  \sum_{m=1}^\di \gamma_{ij}^{m} X_j Z_{m}f\
X_i f + \mathfrak{M}
\end{align*}
where we have let
\begin{align}\label{Monster}
\mathfrak{M} & = - \sum_{\ee=1}^d \left(\sum_{i=1}^d
\omega_{i\ee}^\ee X_i f \right)^2 -  2 \sum_{1 \le \ee<j \le d}
\left(\sum_{i=1}^d \frac{\omega_{ij}^\ee +\omega_{i\ee}^j}{2} X_i f
\right)^2
\\
&  + \sum_{i,j,k,\ee=1}^d \omega^k_{jk} \omega^\ee_{ij} X_\ee f X_i
f -  \sum_{i,j,k,\ee=1}^d  \omega_{ij}^k \omega^\ee_{k j} X_\ee f
X_i f -  \sum_{i,j=1}^d \sum_{\ee=1}^d \sum_{m=1}^\di
\omega_{ij}^\ee \gamma^{m}_{\ee j} Z_{m} f\ X_i f
\notag\\
&  - \sum_{i,j=1}^d  \sum_{m=1}^\di \gamma_{ij}^{m} [Z_{m},X_j]
f\ X_i f  + \sum_{i=1}^d \sum_{j,k=1}^d \sum_{m=1}^\di
\omega_{jk}^k
\gamma_{ij}^{m} Z_{m} f  X_i f \notag\\
& + \sum_{i=1}^d \sum_{j,k=1}^d (X_i \omega^k_{jk}) X_if X_j f +
\sum_{i,j=1}^d \sum_{\ee,k=1}^d \omega^\ee_{ij} \omega^k_{\ee j} X_i
f X_kf + \sum_{i,j,\ee=1}^d \sum_{m=1}^\di \omega^\ee_{ij}
\gamma^{m}_{\ee j} Z_{m}f X_i f
\notag\\
&  - \sum_{i,j=1}^d \sum_{m=1}^\di (X_j \gamma^{m}_{ij}) Z_{m} f
X_i f  - \sum_{i,j=1}^d \sum_{\ee=1}^d (X_j\omega^\ee_{ij}) X_if
X_\ee f
\notag\\
& + \frac{1}{2} \sum_{1\le i<j\le d}\left(\sum_{\ee=1}^d
\omega^\ee_{ij} X_\ee f\right)^2 + \frac{1}{2} \sum_{1\le i<j\le
d}\left(\sum_{m=1}^\di
\gamma^{m}_{ij} Z_{m} f\right)^2 \notag\\
& + \sum_{1\le i<j\le d}\sum_{\ee=1}^d \sum_{m=1}^\di
\omega^\ee_{ij} \gamma^{m}_{ij} Z_{m}f X_\ee f.\notag
\end{align}
Simplifying the latter expression we obtain
\begin{align}\label{Monster2}
\mathfrak{M} & = - \sum_{k,\ee=1}^d \sum_{i=1}^d \omega_{k i}^i
\omega_{\ee i}^i X_k f X_\ee f \notag\\
& -  \frac{1}{2} \sum_{k,l=1}^d \sum_{1 \le i<j \le d} (\omega_{\ee
j}^i +\omega_{\ee i}^j)(\omega^i_{kj} + \omega^j_{ki}) X_k f X_\ee f
\\
& + \sum_{k,\ee=1}^d \sum_{j=1}^d (X_\ee\omega^j_{kj} -
X_j\omega^k_{\ee j}) X_kf X_\ee f    +  \sum_{i,j,k,\ee=1}^d
\omega_{ji}^i \omega^\ee_{k j}  X_k f X_\ee f \notag\\
& + \frac{1}{2} \sum_{k,\ee=1}^d \sum_{1\le i<j\le d}
\omega^\ee_{ij} \omega^k_{ij} X_k f X_\ee f + \sum_{k,j=1}^d
\sum_{m=1}^\di \gamma_{kj}^{m} [X_j,Z_{m}] f\ X_k f
\notag\\
&   + \sum_{i=1}^d \sum_{j,k=1}^d \sum_{m=1}^\di \omega_{jk}^k
\gamma_{ij}^{m} Z_{m} f  X_i f + \sum_{1\le i<j\le
d}\sum_{\ee=1}^d \sum_{m=1}^\di \omega^\ee_{ij} \gamma^{m}_{ij}
Z_{m}f X_\ee f
\notag\\
&  - \sum_{i,j=1}^d \sum_{m=1}^\di (X_j \gamma^{m}_{ij}) Z_{m} f
X_i f  + \frac{1}{2} \sum_{1\le i<j\le d}\left(\sum_{m=1}^\di
\gamma^{m}_{ij} Z_{m} f\right)^2. \notag
\end{align}
At this point, using \eqref{bra2}, it is easy to recognize in view of \eqref{Ro} in Lemma \ref{L:RS} that
\[
\mathfrak M = \mathcal R(f).
\]
 To complete the proof of \eqref{bochner} it now suffices to:
 \begin{itemize}
 \item[1)] use the equation \eqref{S} in Lemma \ref{L:RS};
 \item[2)] recognize by a computation that, in a local horizontal frame, the square of the Hilbert-Schmidt norm of the horizontal Hessian $\nabla_\mathcal{H}^2 f$ is given by
\begin{align}\label{hessian}
\| \nabla_\mathcal{H}^2 f \|^2 & = \sum_{\ee=1}^d \left( f_{,\ell\ell} -\sum_{i=1}^d
\omega_{i\ee}^\ee X_i f \right)^2 +  2 \sum_{1 \le \ee<j \le d}
\left( f_{,j\ell} -\sum_{i=1}^d \frac{\omega_{ij}^\ee
+\omega_{i\ee}^j}{2} X_i f \right)^2.
\end{align}
\end{itemize}

\end{proof}

\subsubsection{\textbf{The generalized curvature-dimension inequality}}\label{SSS:generalizedCD}

In this final part of Section \ref{SS:subR} we establish the main result of the whole section, namely Theorem \ref{T:cd} below. This result shows that, under suitable geometric bounds, see \eqref{riccibounds} below, which are natural in sub-Riemannian geometry (by this we mean that they are satisfied by large classes of significant examples), the sub-Riemannian manifold $\bM$, with its canonical sub-Laplacian $L$ and the Lie subalgebra of transverse symmetries $\mathcal V$, satisfies the curvature-dimension inequality in \eqref{cdi}.

We need to introduce the last intrinsic first-order differential quadratic form, which in a local horizontal frame $\{X_1,...,X_d\}$, we defined  as
\[
\mathcal{T}(f)=\sum_{i=1}^d \| \text{T}(X_i, \nabla_{\mathcal{H}} f) \|^2.
\]
A computation shows that in a  local adapted frame we obtain:
\begin{equation}\label{To}
\mathcal{T} (f)= \sum_{j=1}^d \sum_{m=1}^\di \left(\sum_{i=1}^d
\gamma_{ij}^{m} X_i f \right)^2.
\end{equation}

It is worth remarking that, as we have already observed, in the Riemannian case $\nabla$ is the Levi-Civita connection. As a consequence, in such case $\mathcal T(f) = 0$ for every $f\in C^\infty(\bM)$. 

\begin{theorem}\label{T:cd}
Suppose that there exist constants $\rho_1
\in \mathbb{R}$, $\rho_2
>0$ and $\kappa \ge 0$ such that for every $f\in C^\infty(\bM)$:
\begin{equation}\label{riccibounds}
\begin{cases}
\mathcal{R}(f) \ge \rho_1 \Gamma (f) +\rho_2 \Gamma^Z (f),
\\
\mathcal{T}(f) \le \kappa \Gamma (f).
\end{cases}
\end{equation}
Then, the sub-Riemannian manifold $\bM$ satisfies the generalized curvature-dimension inequality \emph{CD}$(\rho_1,\rho_2,\kappa,d)$ in \eqref{cdi} with respect to the sub-Laplacian $L$ and the differential form $\Gamma^Z$.
\end{theorem}

\begin{proof}
We need to show that for every $f\in C^\infty(\bM)$ and any $\nu >0$ one has: 
\[
\Gamma_{2}(f)+\nu \Gamma^Z_{2}(f) \ge \frac{1}{d} (Lf)^2 + \left( \rho_1 -\frac{\kappa}{\nu}\right)  \Gamma (f) + \rho_2 \Gamma^Z (f).
\]
Let $\{X_1,...,X_d,Z_1,...,Z_\di\}$ be a local adapted frame. 
From \eqref{L} and \eqref{X0} and Schwarz inequality we find
\[
Lf = \sum_{\ee=1}^d \left(X_\ee^2 f - \sum_{i=1}^d
\omega^\ee_{i\ee} X_i f\right) \leq \sqrt d \left(\sum_{\ee=1}^d
\left(X_\ee^2 f - \sum_{i=1}^d \omega^\ee_{i\ee} X_i
f\right)^2\right)^{1/2}
\]
This inequality and \eqref{hessian} readily give
\[
\frac{1}{d} (Lf)^2 \le ||\nabla^2_{\mathcal H} f||^2.
\]
From this estimate and from \eqref{bochner} in Theorem \ref{T:bochner} we obtain
\begin{align}\label{1d}
\frac{1}{d} (Lf)^2 & \le \Gamma_2(f) - \mathcal R(f) + \mathcal S(f)
\\
& \le \Gamma_2(f) - \rho_1 \Gamma(f) - \rho_2 \Gamma^Z(f) + \mathcal S(f),
\notag
\end{align}
where in the last inequality we have used the lower bound on $\mathcal R(f)$ in the hypothesis \eqref{riccibounds}. Using \eqref{S2} and the Cauchy-Schwarz inequality we now find for every $\nu>0$
\begin{align*}
|\mathcal{S}(f)| & \le 2 \left(\sum_{j=1}^d \sum_{m=1}^\di \left(\sum_{i=1}^d \gamma_{ij}^{m} X_i f\right)^2\right)^{1/2} \left(\sum_{j=1}^d \sum_{m=1}^\di \left(X_j Z_{m} f\right)^2\right)^{1/2}
\\
& = 2 \mathcal T(f)^{1/2} \Gamma_2^Z(f)^{1/2} \le \frac{1}{\nu} \mathcal T(f) + \nu \Gamma^Z_2(f),
\end{align*}
where in the second to the last equality we have used \eqref{vbf} and \eqref{To}.
Substituting the latter inequality in \eqref{1d} we find
\[
\frac{1}{d} (Lf)^2 \leq \Gamma_2(f) + \nu \Gamma^Z_{2}(f) +
\frac{1}{\nu} \mathcal T(f) - \rho_1 \Gamma(f) - \rho_2 \Gamma^Z(f).
\]
At this point it suffices to use the bound from above on $\mathcal T(f)$ in \eqref{riccibounds} to reach the desired conclusion.

\end{proof}

The next result shows that, remarkably, the generalized curvature-dimension inequality \eqref{cdi} in Definition \ref{D:cdi} is equivalent to the geometric bounds \eqref{riccibounds} above.

\begin{theorem}\label{P:tres_beau}
Suppose that there exist constants $\rho_1
\in \mathbb{R}$, $\rho_2
>0$ and $\kappa \ge 0$ such that $\M$ satisfy the generalized curvature-dimension inequality \emph{CD}$(\rho_1,\rho_2,\kappa,d)$. Then, $\M$ satisfies the geometric bounds \eqref{riccibounds}. As a consequence of this fact and of Theorem \ref{T:cd} we conclude that
\[
\emph{CD}(\rho_1,\rho_2,\kappa,d) \Longleftrightarrow \begin{cases}
\mathcal{R}(f) \ge \rho_1 \Gamma (f) +\rho_2 \Gamma^Z (f),
\\
\mathcal{T}(f) \le \kappa \Gamma (f).
\end{cases}
\]
\end{theorem}

\begin{proof}
 Let us fix $x_0 \in \mathbb{M}$, $u \in \mathcal{H}_{x_0}(\mathbb{M})$ and $v \in \mathcal{V}_{x_0}(\mathbb{M})$.  Let also $\nu >0$. Let $\{X_1,...,X_d,Z_1,...,Z_\di\}$ be a local adapted frame around $x_0$.
We claim that we can find a function $f\in C^\infty(\M)$ such that:
\begin{itemize}
\item[(i)] $\nabla_\mathcal{H} f (x_0)=u$,
\item[(ii)] $\nabla_\mathcal{V} f (x_0)=v$,
\item[(iii)] $\nabla^2_\mathcal{H} f (x_0)=0$,
\item[(iv)] $X_jZ_m f (x_0)=\frac{1}{\nu} \sum_{i=1}^d \gamma_{ij}^m(x_0) u_i $.
\end{itemize}
To see this, we denote as before by $\nabla^R$ the Levi-Civita connection of the Riemannian metric on $\bM$. Since $\{X_1,...,X_d,Z_1,...,Z_\di\}$ is a local frame, we can find a local chart $(U,\phi)$ at $x_0$, such that $\phi(0)=x_0$ and in $U$ we have $X_j=\frac{\partial }{\partial x_j}$, $j=1,...,d$, $Z_m=\frac{\partial }{\partial z_m}$, $m=1,...,\di$. We first observe that there exists a function $f_1\in C^\infty(\M)$ such that
\[
\begin{cases}
\nabla^R f_1 (x_0)=u+v,
\\
\nabla^R \nabla^Rf_1 (x_0)=0.
\end{cases}
\]
For the explicit construction of such function $f_1$, see for instance the proof of Theorem 3.1 and Lemma 3.2 in \cite{SVR}. 
 We can also find a function $f_2\in C^\infty(\M)$ such that 
\[
\begin{cases}
\nabla^R f_2 (x_0)=0,
\\
X_jZ_m f_2 (x_0) = \frac{1}{\nu} \sum_{i=1}^d \gamma_{ij}^m(x_0) u_i-X_jZ_mf_1(x_0).
\end{cases}
\]
Indeed, it will suffice to take as $f_2$ the function that in the local coordinates 
\[
(x,z) =(x_1,...,x_d,~z_1,...,z_\di)\]
is expressed in the form 
\[
f_2(x,z) =\sum_{j=1}^d \sum_{m=1}^\di \left(  \frac{1}{\nu} \sum_{i=1}^d \gamma_{ij}^m(x_0) u_i-X_jZ_mf_1(x_0) \right)  x_j z_m.
\]
It is readily verified that such $f_2$ satisfies the two above conditions.
With this being done, we now take $f=f_1+f_2$. It is clear that such $f$ satisfies (i)-(iv) above.
Now, using CD$(\rho_1,\rho_2,\kappa,d)$ on the function $f$, in combination with (i)-(iii) above, we find at the point $x_0$,
\[
\Gamma_2(f)(x_0)+\nu \Gamma^Z_2(f)(x_0) \ge \left(\rho_1 -\frac{\kappa}{\nu}  \right) \|u \|^2 + \rho_2 \|v \|^2.
\]
But, from \eqref{bochner} in Theorem \ref{T:bochner} and (iii) we have
\begin{align*}\label{bochnerbis}
\Gamma_{2}(f)(x_0)= \mathcal{R}(f)(x_0)+\mathcal{S}(f)(x_0).
\end{align*}
By \eqref{S2} and (iii) we find
\begin{align*}
\mathcal{S}(f)(x_0) & = - 2 \sum_{i,j=1}^d \sum_{m=1}^\di \gamma_{ij}^{m}(x_0) X_j Z_{m} f(x_0) X_i f(x_0) = - \frac{2}{\nu} \sum_{j =1}^d \sum_{m=1}^\di \sum_{i,\ell=1}^d \gamma_{ij}^{m}(x_0)\gamma^m_{\ell j}(x_0) u_\ell u_i
\\
& = - \frac{2}{\nu} \sum_{j =1}^d \sum_{m=1}^\di  \left(\sum_{i=1}^d \gamma_{ij}^m(x_0) u_i\right)^2 = - \frac{2}{\nu} \mathcal T(f)(x_0),
\end{align*}
where in the last equality we have used \eqref{To}.
On the other hand, \eqref{vbf} and (iii) give
\begin{equation*}\label{vbfbis}
\Gamma^Z_{2}(f) (x_0)= \| \nabla_\mathcal{H} \nabla_\mathcal{V} f(x_0) \|^2 = \frac{1}{\nu^2} \sum_{j=1}^d \sum_{m=1}^\di \left(\sum_{i=1}^d \gamma_{ij}^m(x_0) u_i\right)^2 = \frac{1}{\nu^2} \mathcal T(f)(x_0), 
\end{equation*}
where in the last equality we have used \eqref{To} again. In conclusion,
\[
\Gamma_2(f)(x_0)+\nu \Gamma^Z_2(f)(x_0)=\mathcal{R}(f)(x_0)+\mathcal{S}(f)(x_0)+\nu \| \nabla_\mathcal{H} \nabla_\mathcal{V} f(x_0) \|^2=\mathcal{R}(f)(x_0) -\frac{1}{\nu} \mathcal{T}(f)(x_0).
\]
We thus infer from CD$(\rho_1,\rho_2,\kappa,d)$ 
\[
\mathcal{R}(f)(x_0) -\frac{1}{\nu} \mathcal{T}(f)(x_0) \ge  \left(\rho_1 -\frac{\kappa}{\nu}  \right) \|u \|^2 + \rho_2 \|v \|^2
\]
We finally note that \eqref{R2} in Lemma \ref{L:RS} gives 
\begin{align*}
\mathcal R(f) (x_0)& = \sum_{k,\ee=1}^d \bigg\{\bigg(\sum_{j=1}^d
\sum_{m=1}^\di \gamma_{kj}^{m} \delta_{jm}^\ee\bigg) +
\sum_{j=1}^d
(X_\ee\omega^j_{kj} - X_j\omega^k_{\ee j}) \\
& + \sum_{i,j=1}^d \omega_{ji}^i \omega^\ee_{k j} - \sum_{i=1}^d
\omega_{k i}^i \omega_{\ee i}^i + \frac{1}{2} \sum_{1\le i<j\le d}
\bigg(\omega^\ee_{ij} \omega^k_{ij} - (\omega_{\ee j}^i +\omega_{\ee
i}^j)(\omega^i_{kj} + \omega^j_{ki})\bigg)\bigg\}u_k u_l
\notag\\
&   +  \sum_{k=1}^d \sum_{m=1}^\di \bigg(\sum_{\ell,j=1}^d
\omega_{j\ee}^\ee \gamma_{kj}^{m} + \sum_{1\le \ee<j\le d}
\omega^k_{\ee j} \gamma^{m}_{\ee j}  - \sum_{j=1}^d X_j
\gamma^{m}_{kj}\bigg) v_m u_k  \notag\\
& + \frac{1}{2} \sum_{1\le \ee<j\le d}\bigg(\sum_{m=1}^\di
\gamma^{m}_{\ee j} v_m\bigg)^2, \notag \\
&:=\mathcal R(u,v),
\end{align*}
and that \eqref{To} gives
\begin{align*}
\mathcal{T} (f)(x_0)&= \sum_{j=1}^d \sum_{m=1}^\di \left(\sum_{i=1}^d
\gamma_{ij}^{m}u_i\right)^2\\
&: = \mathcal T(u).
\end{align*}
In conclusion, we have proved that for every  $u \in \mathcal{H}_{x_0}(\mathbb{M})$ and $v \in \mathcal{V}_{x_0}(\mathbb{M})$ and  $\nu >0$,
\[
\mathcal R(u,v) -\frac{1}{\nu}\mathcal T(u) \ge  \left(\rho_1 -\frac{\kappa}{\nu}  \right) \|u \|^2 + \rho_2 \|v \|^2.
\]
By first letting $\nu \to \infty$, we obtain 
\[
\mathcal R(u,v) \ge  \rho_1  \|u \|^2 + \rho_2 \|v \|^2.
\]
If instead we let $ \nu \to 0$, we find 
$\mathcal T(u) \le k\|u \|^2$. This completes the proof.
\end{proof}

\subsection{Carnot groups of step two}\label{SS:carnot}

Carnot groups of step 2 provide  a natural reservoir of sub-Riemannian manifolds with transverse symmetries. 
Let $\mathfrak{g}$ be a graded nilpotent Lie
algebra of step two.  This means that $\mathfrak{g}$ admits a splitting $\mathfrak g = V_1 \oplus V_2$, where $[V_1,V_1] = V_2$, and $[V_1,V_2]=\{0\}$. We endow $\mathfrak{g}$ with an inner product $\langle\cdot, \cdot \rangle$ with respect to which the decomposition $V_1 \oplus V_2$ is orthogonal.
  We denote by $e_1,...,e_d$  an orthonormal basis of $V_1$ and by $\varepsilon_1,...,\varepsilon_\di$ an orthonormal basis of $V_2$. Let $\bG$ be the connected and simply connected graded nilpotent Lie group associated with $\mathfrak{g}$. 
  Left-invariant vector fields in $V_2$ are seen to be transverse sub-Riemannian Killing vector fields of the horizontal distribution given by $V_1$.  The geometric assumptions  of the previous section are thus satisfied.

Let $L_x(y) = x y$ be the operator of
left-translation on $\bG$, and indicate with $dL_x$ its
differential. We
indicate with $X_j(x) = dL_x(e_j)$, $j=1,\cdots,d$ and $Z_m(x)=dL_x(\varepsilon_m)$, $m=1,\cdots,\di,$ the
corresponding system of left-invariant vector fields on $\bG$. Using the Baker-Campbell-Hausdorff formula, we see that in exponential coordinates
\[
X_i=\frac{\partial}{\partial x_i}-\frac{1}{2}  \sum_{m=1}^\di \sum_{\ell=1}^d \gamma^m_{i\ell} x_\ell Z_m
\]
where $\gamma_{i\ell}^m=\langle [e_i,e_\ell],\varepsilon_m \rangle$
are the group constants. From the latter equation we see that
\begin{equation}\label{commgroup}
[X_i,X_j] =  \sum_{m=1}^\di \gamma^m_{ij} Z_m.
\end{equation}
We note that $X_1,...,X_d,Z_1,...,Z_\di$ is a global adapted frame on $\bG$. 

A canonical sub-Laplacian on $\bG$ is given by 
\[
L = \sum_{i=1}^d X_i^2.  
\]
If we endow $\bG$ with a bi-invariant Haar measure $\mu$, then $X_i^* = - X_i$, see e.g. \cite{Fo}. Therefore, $L$ is symmetric with respect to $\mu$.

In the present setting we have
\[
\Gamma(f) = \sum_{i=1}^d (X_i f)^2,\ \ \ \ \ \ \ \ \ \Gamma^Z(f) = \sum_{m=1}^\di (Z_m f)^2.
\]
If we use Lemma \ref{L:RS}, then we easily see that
\[
\mathcal{R}(f) = \frac{1}{4} \sum_{i,j=1}^d \left(\sum_{m=1}^\di \gamma^m_{ij} Z_{m}f\right)^2.
\]
From this expression it is clear that
\[
\mathcal R(f) \ge \rho_2 \Gamma^Z(f),
\]
with
\begin{equation}\label{rho2G}
 \rho_2=\inf_{\| z \|=1} \frac{1}{4} \sum_{i,j=1}^d \left(\sum_{m=1}^\di \gamma^m_{ij} z_m \right)^2.
\end{equation}
Furthermore, from  \eqref{To} one has
\[
\mathcal{T} (f)= \sum_{j=1}^d \sum_{m=1}^\di \left(\sum_{i=1}^d
\gamma_{ij}^{m} X_i f \right)^2,
\]
and therefore
\[
\mathcal T(f) \le \kappa \Gamma(f),
\]
with
\begin{equation}\label{kappaG}
\kappa=\sup_{\|x\|=1}  \sum_{j=1}^d \sum_{m=1}^\di \left(\sum_{i=1}^d
\gamma_{ij}^{m} x_i \right)^2.
\end{equation}
From these considerations in view of Theorem \ref{T:cd} we immediately obtain the following result.

\begin{proposition}\label{P:carnotCD}
Let $\bG$ be a Carnot group of step two, with $d$ being the dimension of the horizontal layer of its Lie algebra. Then, $\bG$ satisfies the generalized curvature-dimension inequality \emph{CD}$(0,\rho_2, \kappa, d)$ (with respect to any sub-Laplacian $L$ on $\bG$),  
with $\rho_2$  and $\kappa$ respectively given by \eqref{rho2G} and \eqref{kappaG}. 
\end{proposition}
 
In particular, in our framework, every Carnot group of step two is a \textit{sub-Riemannian manifold with nonnegative Ricci tensor}.

\subsubsection{\textbf{Groups of Heisenberg type}}\label{SSS:ht}
  
A significant class of Carnot groups of step two is that of groups of Heisenberg type. Such groups constitute a generalization of the Heisenberg group and they carry a natural complex structure. Groups of Heisenberg type (aka H-type groups) were first introduced by Kaplan \cite{Kaplan} in connection with the study of hypoellipticity and they were further developed in \cite{CDKR}, where the authors characterized those groups of H-type which arise as the nilpotent component in the Iwasawa decomposition of simple Lie groups of real rank one. In a Carnot group of step two $\bG$ consider the map $J:V_2 \to \text{End}(V_1)$ defined for every $\eta\in V_2$ by 
\[
<J(\eta)\xi,\xi'> = <[\xi,\xi'],\eta>,\ \ \ \ \ \xi, \xi'\in V_1, \eta\in V_2.
\]
Then, $\bG$ is said of H-type if $J(\eta)$ is an orthogonal map on $V_1$ for every $\eta\in V_2$ such that $||\eta||=1$. When $\bG$ is of H-type we thus have for $\xi, \xi'\in V_1$, $\eta\in V_2$, 
\[
<J(\eta)\xi,J(\eta)\xi'> = ||\eta||^2 <\xi,\xi'>.
\]
The $J$ map induces a complex structure since in every group of H-type one has for every $\eta, \eta'\in V_2$,
\[
J(\eta)J(\eta') + J(\eta')J(\eta) = - 2 <\eta,\eta'> I,
\]
see \cite{CDKR}. In particular,
\[
J(\eta) ^2= - ||\eta||^2 I.
\]
Since in a Carnot group of step two  we always have $[e_i,e_j] = \sum_{s=1}^\di \gamma_{ij}^s \ve_s$, we obtain 
\[
<J(\ve_m)e_i,e_j> = <[e_i,e_j],\ve_m> = \gamma_{ij}^m.
\]
When $\bG$ is of H-type we thus find for $z = \sum_{m=1}^\di z_m \ve_m$,
\begin{align*}
\frac{1}{4} \sum_{i,j=1}^d \left(\sum_{m=1}^\di \gamma^m_{ij} z_m \right)^2=\frac{1}{4} \sum_{i,j=1}^d <J(z)e_i,e_j>^2=\frac{d}{4} ||z||^2.
\end{align*}
In view of \eqref{rho2G} we conclude that, when $\bG$ is of H-type, then $\rho_2=\frac{d}{4}$. Also, for $x = \sum_{i=1}^d x_i e_i$ one has,
\begin{align*}
 \sum_{j=1}^d \sum_{m=1}^\di \left(\sum_{i=1}^d
\gamma_{ij}^{m} x_i \right)^2= \sum_{j=1}^d \sum_{m=1}^\di \langle J(\varepsilon_m)x , e_j \rangle^2= \sum_{m=1}^\di \|J(\varepsilon_m) x \|^2=\di \| x \|^2.
\end{align*}
In view of \eqref{kappaG} we conclude $\kappa=\di$. Combining these considerations with Proposition \ref{P:carnotCD}, we have thus proved the following result.
\begin{proposition}\label{P:HgroupsCD}
Let  $\bG$  be a group of $H$-type. Then,  $\bG$ satisfies the generalized curvature-dimension inequality \emph{CD}$(0,\frac{d}{4}, \di, d)$ with respect to any sub-Laplacian $L$.
\end{proposition}

\subsection{CR Sasakian manifolds}\label{SS:sasakian}

Another interesting class of sub-Riemannian manifolds with transverse symmetries is given by the class of CR Sasakian manifolds. For all the known results cited in this section we refer the reader to the monograph \cite{CR}. Let $\mathbb{M}$ be a non degenerate CR manifold of real
hypersurface type and dimension $d +1$, where $d= 2n$. Let $\theta$
be a pseudo-hermitian form on $\mathbb{M}$ with respect to which the Levi
form $L_\theta$ is positive definite. The kernel of $\theta$ determines the horizontal bundle $\mathcal H$. Denote now  by $Z$ the Reeb vector field on $\bM$, i.e., the characteristic
direction of $\theta$. It is an immediate consequence of Theorem 1.3 on p. 25 in \cite{CR} that the canonical connection $\nabla$ introduced in Section \ref{SSS:cancon} coincides with the Tanaka-Webster connection on $\bM$. The sub-Laplacian $L$ introduced in Section \ref{SS:subR} is then the classical CR sub-Laplacian, see Definition 2.1 on p. 111 of \cite{CR}.

Like in the Riemannian case, the pseudo-hermitian torsion with respect to $\nabla$ is 
\[
\text{T}(X,Y) = \nabla_X Y - \nabla_Y X - [X,Y].
\] 

\begin{definition}\label{D:sasakian}
The \emph{CR} manifold $(\bM,\theta)$ is called Sasakian if the pseudo-hermitian torsion vanishes, in the sense that
\[
\emph{T}(Z,X) = 0,
\]
for every $X\in \mathcal H$.
\end{definition}

In every Sasakian manifold the Reeb vector field $Z$ is a sub-Riemannian Killing vector field (see Theorem 1.5 on p. 42 and Lemma 1.5 on p. 43  in \cite{CR}).  In this situation, the bilinear forms $\mathcal{R}, \mathcal{T}$ take a particularly nice form. Indeed, in the Sasakian case, the torsion T of the Tanaka-Webster connection is given, for horizontal vector fields $X$ and $Y,$ by
\[
\text{T}(X,Y)=\langle J X, Y \rangle Z,
\]
where $J$ is the complex structure on $\bM$. Since $\nabla J=0$,  we  obtain from \eqref{Ro} in Definition \ref{D:RS},
\begin{equation}\label{Rsasakian}
\mathcal{R}(f)=\text{Ric}(\nabla_\mathcal{H} f , \nabla_\mathcal{H} f ) +\frac{1}{4}\left( \sum_{l,k=1}^d \langle J X_l, X_k \rangle^2  \right) (Zf)^2.
\end{equation}
Since
\[
 \sum_{l,k=1}^d \langle J X_l, X_k \rangle^2= \sum_{k=1}^d \| J X_k \|^2=d,
\]
we conclude from \eqref{Rsasakian} that 
\[
\mathcal{R}(f)=\text{Ric}(\nabla_\mathcal{H} f , \nabla_\mathcal{H} f ) +\frac{d}{4}\Gamma^Z(f).
\]
Also, from \eqref{To}
\[
\mathcal{T}(f)=\sum_{i=1}^d \langle J \nabla_\mathcal{H} f, X_i \rangle^2= \| J \nabla_\mathcal{H} f \|^2=\Gamma(f).
\]
As a straightforward consequence of these considerations we thus obtain from Theorem \ref{T:cd}.

\begin{theorem}\label{T:sasakian}
Assume that the Tanaka-Webster Ricci tensor is bounded from below by $\rho_1 \in \mathbb{R}$ on smooth functions, that is for every $f\in C^\infty(\bM)$ 
\[
\emph{Ric}(\nabla_\mathcal{H} f , \nabla_\mathcal{H} f ) \ge \rho_1 \|  \nabla_\mathcal{H} f \|^2.
\]
Then, the Sasakian manifold $\bM$ satisfies the generalized curvature-dimension inequality  \emph{CD}$(\rho_1,\frac{d}{4},1,d)$, with $d = 2n$.
\end{theorem}

 \begin{remark}\label{S:hopf}
The example of CR Sasakian manifolds, together with that  of H-type groups studied in Section \ref{SSS:ht}, suggests the existence of an interesting class of sub-Riemannian manifolds with transverse symmetries. Indeed, returning to the setting and notations of Section \ref{SS:subR}, for $Z \in \mathcal{V}$ consider the map $J(Z)$ defined  on the horizontal bundle $\mathcal H$ by
\[
 \langle J(Z) X, Y \rangle = \langle Z, T(X,Y) \rangle.
\] 
Suppose that $J(Z)$ is orthogonal for every $Z\in \mathcal V$ such that $\| Z \|=1$, and that furthermore $\sum_{k=1}^d \nabla_{X_k} J(Z) =0$. In that case, similarly to the case of groups of the H-type case and Sasakian manifolds we can prove that, if the horizontal Ricci curvature of the canonical connection $\nabla$ is bounded from below by $\rho_1$, then $\bM$ satisfies the generalized curvature-dimension inequality \emph{CD}$(\rho_1,\frac{d}{4}, \di, d)$. An example of such structure is given by the Hopf fibration $\mathbb{S}^7 \to \mathbb{S}^4$ and, more generally, by the so-called  $3$ Sasakian manifolds (see  \cite{BGa} for an account on these geometric structures).
\end{remark}

\subsection{Principal bundles over Riemannian manifolds}\label{SS:ortho}

 Sub-Riemannian structures  with transverse symmetries also naturally arise in the context of principal bundles over Riemannian manifolds. 
 Let $(\mathbb{M},g)$ be a $C^\infty$ connected Riemannian manifold with dimension $d$.  Let us consider the orthonormal frame bundle $
\mathcal{O} \left( \mathbb{M}\right) $ over $\mathbb{M}$. The kernel of the Levi-Civita connection form defines the distribution $\mathcal{H}$ of horizontal directions. If the Riemannian curvature form is non-degenerate this distribution is two-step bracket generating (see for instance Chapter 3 in \cite{baudoin}). The set of vertical directions is then given by the vector fields that are tangent to the fibers of the bundle projection. It is then easily seen that in such case $\mathcal{V} \simeq \mathfrak{so}_d (\mathbb{R}),$ and therefore that the vertical bundle is generated by the sub-Riemannian  Killing vector fields of the horizontal bundle. We therefore have an example of a sub-Riemannian manifold with transverse symmetries. In this example the geometric quantities introduced in Section 2.2 may be interpreted in terms of the geometry of $\bM$.

First, let us observe that we can  find a globally defined adapted frame. For each
$x \in \mathbb{R}^d$ we can define  a horizontal vector field $H_x$
on $\mathcal{O}\left( \mathbb{M}\right)$ by the property that at
each point $u \in \mathcal{O} ( \mathbb{M} )$, $H_{x} (u)$ is the
horizontal lift of $u (x)$ from $u$. If  $(e_1,...,e_d)$ is the
canonical basis of $\mathbb{R}^d$,  the fundamental horizontal
vector fields are then defined by
\begin{equation*}
H_i= H_{e_i}.
\end{equation*}
 Now, for every $M \in \mathfrak{o}_d
(\mathbb{R})$ (space of $d
\times d$ skew-symmetric matrices), we can define a vertical vector field $V_M$ on
$\mathcal{O} \left( \mathbb{M}\right)$ by
\[
(V_M F)(u) = \lim_{t \rightarrow 0} \frac{F \left( u e^{tM}
\right) -F(u)}{t},
\]
where $u \in \mathcal{O} \left( \mathbb{M}\right)$ and $F:
\mathcal{O} \left( \mathbb{M}\right) \rightarrow \mathbb{R}$. If
$E_{ij}$, $1 \le i <j \le d$ denote the canonical basis of
$\mathfrak{o}_d (\mathbb{R})$ ($E_{ij}$ is the matrix whose
$(i,j)$-th entry is $1/2$, $(j,i)$-th entry is $-1/2$, and all other
entries are zero), then the fundamental vertical vector fields are
given by
\[
V_{ij}= V_{E_{ij}}.
\]
It can be shown that we have the following Lie bracket relations:
\[
[H_i,H_j]=-2\sum_{k<l} \Omega_{ij}^{kl} V_{kl},
\]
\[
[H_i,V_{jk}]=-\delta_{ij} \frac{1}{2} H_k +\delta_{ik} \frac{1}{2} H_j,
\]
where $\delta_{ij}=1$ if $i=j$ and $0$ otherwise,  and where $\Omega$ is the Riemannian curvature form:
\[
\Omega (X,Y)(u)=u^{-1} R(\pi_* X , \pi_* Y)u,~~X,Y \in
\mathrm{T}_u \mathcal{O} \left( \mathbb{M}\right),
\]
$R$ denoting the Riemannian curvature tensor on $\mathbb{M}$ and $\pi$ the canonical projection $\mathcal{O}\left( \mathbb{M}\right) \rightarrow \mathbb{M}$.

In this structure, the sub-Laplacian $L$ is the so-called horizontal Bochner Laplace operator. 
It is by definition the  diffusion operator on $\mathcal{O}\left( \mathbb{M}\right)$ given by
\begin{equation*}
\Delta _{\mathcal{O}\left( \mathbb{M}\right)
}=\sum_{i=1}^{d}H_{i}^{2}.
\end{equation*}
Its fundamental property is
that it is the lift of the Laplace-Beltrami operator $\Delta_{
\mathbb{M}}$ of $\mathbb{M}$. That is, for every smooth $f:
\mathbb{M} \rightarrow \mathbb{R}$,
\[
\Delta _{\mathcal{O}\left( \mathbb{M}\right)} (f \circ
\pi)=(\Delta _{ \mathbb{M}} f)\circ \pi.
\]

The canonical sub-Riemannian connection  $\nabla$ is easily expressed in terms of the Ehresman bundle connection. Let us  recall (see for instance Chapter 3 in \cite{baudoin}) that the Ehresmann connection form $\alpha$  on $\mathcal{O} \left(
\mathbb{M}\right)$   is the unique
skew-symmetric matrix $\alpha$ of one forms on $\mathcal{O} \left(
\mathbb{M}\right)$ such that:
\begin{enumerate}
\item $\alpha (X) =0$ if and only if $X \in \mathcal{H} \mathcal{O} ( \mathbb{M}
)$;
\item $V_{\alpha (X) }=X$ if and only if $X \in \mathcal{V} \mathcal{O} ( \mathbb{M}
)$,
\end{enumerate}
where $ \mathcal{H} \mathcal{O} ( \mathbb{M})$ denotes the horizontal bundle and $\mathcal{V} \mathcal{O} ( \mathbb{M}
)$ the vertical bundle. It is then easily verified that for a vector field $Y$ on $ \mathcal{O} ( \mathbb{M})$,
\[
\nabla_Y H_i=\sum_{k=1}^d \alpha_j^k (Y) H_k.
\]
Let us observe that if $X,Y$ are smooth horizontal vector fields then we have for the torsion:
\[
\text{T}(X,Y)=V_{\Omega (X,Y)}.
\]
We then obtain after straightforward computations
\begin{align*}
\mathcal{R} (f,f)= &\text{Ric}^* (\nabla_\mathcal{H} f ,\nabla_\mathcal{H} f )  - \sum_{j,k=1}^d V_{(\nabla_{H_j} \Omega) (H_j,H_k)} f H_kf   +\frac{1}{4}
\left( V_{\Omega(H_j,H_k)} f\right)^2,
\end{align*}
where for horizontal vector fields $X$ and $Y$,
\[
\text{Ric}^* (X,Y)= \text{Ric} (\pi_* X,\pi_*Y),
\]
and Ric denotes the Ricci tensor of $\mathbb{M}$. In the same vein we have
\[
\mathcal{T}(f)=\sum_{i=1}^d \| \text{T}(H_i, \nabla_{\mathcal{H}} f) \|^2=\sum_{i=1}^d \| V_{\Omega (H_i, \nabla_{\mathcal{H}} f)} \|^2.
\]
 We then observe that the expression of $\mathcal{R}$ simplifies if for every horizontal vector field $X$,
 \[
 \sum_{j=1}^d (\nabla_{H_j} \Omega) (H_j, X)=0.
\]
Using the second Bianchi identity, it is seen that this latter condition is equivalent to the fact that the Ricci tensor of $\mathbb{M}$ is a Codazzi tensor, that is for any vector fields $X,Y,Z$ on $\bM$,
\[
(\nabla_X \text{Ric}) (Y, Z)=(\nabla_Y \text{Ric}) (X, Z).
\]
As a conclusion we then obtain from Theorem \ref{T:cd}.

\begin{proposition}
Let $(\bM,g)$ be a $C^\infty$ connected Riemannian manifold with dimension $d$. Assume that:
\begin{enumerate}
\item $ \emph{Ric}$ is a Codazzi tensor;
\item There exists $\rho_1 \ge 0$ such that $\emph{Ric} \ge \rho_1$;
\item There exists $\rho_2 >0$ such that for every $U \in \mathfrak{so}_d(\mathbb{R})$,
\[
\sum_{i,j=1}^d  \langle  \Omega (H_j,H_k), U \rangle^2 \ge 4 \rho_2 \| U \|^2;
\]
\item There exists $\kappa \ge 0$ such that for every horizontal vector field $X.$
\[
\sum_{i=1}^d \| \Omega (H_i, X) \|^2 \le \kappa \|X\|^2,
\]
\end{enumerate}
Then, the horizontal Bochner operator of $\mathcal{O}(\bM)$ satisfies the generalized curvature-dimension inequality \emph{CD}$(\rho_1,\rho_2,\kappa,d)$.
\end{proposition}

The previous assumptions are readily satisfied in the case of spaces with constant curvature and, after some standard computations,  we obtain the following result.

\begin{corollary}
Let $(\bM,g)$ be a $C^\infty$ connected Riemannian manifold with dimension $d$ and constant curvature $K \neq 0$. The sub-Riemannian structure of $\mathcal{O}(\bM)$ satisfies the generalized curvature dimension inequality \emph{CD}$\left((d-1)K,\frac{d}{4}K^2,\frac{d(d-1)}{2} K^2,d\right).$
 \end{corollary}

Actually, more general principal bundles provide examples of   sub-Riemannian structures with transverse symmetries.
Let $\pi: (\bM,g) \to (\bM',g')$ be the projection of a principal fibre bundle with structure group a compact, semisimple Lie group $\bG$  with dimension $\di$ equipped with its bi-invariant metric given by the Cartan-Killing form. We suppose that $\pi$ is a Riemannian submersion with totally geodesic fibres isometric to $\bG$. We denote by $\theta$ the one-form of the principal connection corresponding to the  horizontal distribution $\mathcal{H}$.  If $\mathcal{H}$ is bracket generating, then we have an example of a sub-Riemannian structure with transverse symmetries. 

Let
\[
A_X Y=( \tilde{\nabla}_{X_\mathcal{H}} Y_\mathcal{H} )_\mathcal{V}+( \tilde{\nabla}_{X_\mathcal{H}} Y_\mathcal{V} )_\mathcal{H},
\]
be the O'Neill's tensor of the submersion. When $X$ and $Y$ are horizontal vector fields we have $\text{T}(X,Y)=-2A_X Y$, where, as usual,  $\text{T}$ denotes the torsion of the canonical sub-Riemannian connection. As a consequence, $A$ is the skew-symmetrization of $-\frac{1}{2} \text{T}$. The connection form $\theta$ is a Yang-Mills connection if for every horizontal vector field $X$, the vertical component of
\[
 \sum_{\ee=1}^d(\nabla_{X_\ee} \emph{T}) (X_\ee,X)
 \]
is zero (see for instance \cite{falcitelli}, p. 146). 
As a consequence of Theorem \ref{T:cd}, we then obtain the following result.
 
\begin{proposition}
Let us assume that:
\begin{enumerate}
\item $\theta$ is a Yang-Mills connection;
\item  there exists $\rho_1 \ge 0$ such that $\emph{Ric}' \ge \rho_1$ where  $\emph{Ric}' $ is the Ricci tensor of $(\bM',g')$;
\item there exists $\rho_2 >0$ such that for every vertical vector field $Z$,
\[
\sum_{i=1}^d  \| A_{X_i}  Z \|^2 \ge  \rho_2 \| Z \|^2;
\]
\item there exists $\kappa \ge 0$ such that for every horizontal vector field $X$,
\[
\sum_{m=1}^{\di}  \| A_{X} Z_i \|^2 \le \kappa \|X\|^2.
\]
\end{enumerate}
Then, the sub-Riemannian structure on  $\bM$  given by the submersion $\pi: (\bM,g) \to (\bM',g')$ satisfies the generalized curvature dimension inequality \emph{CD}$(\rho_1,\rho_2,\kappa,d)$.
\end{proposition}

\begin{remark}
If $\bG$ is simple, then by $\mathbf{Ad}$ invariance, it is seen that  
\[
\sum_{i=1}^d  \| A_{X_i}  Z \|^2= \frac{\| A \|^2}{\di} \| Z \|^2 ,
\]
and
\[
\sum_{m=1}^{\di}  \| A_{X} Z_i \|^2 \le \| A \|^2  \|X\|^2.
\]
\end{remark}

 \section{Second derivatives estimates}

In this section, in the context of sub-Riemannian manifolds with transverse symmetries, we develop some basic tools to obtain bounds on the second derivatives that will later be used. 

Let $\M$ be a sub-Riemannian manifold with transverse symmetries as in the previous section. If $X_1,\cdots,X_d$ is a local frame of horizontal vector fields, we define the tensor
\[
\delta T (V)= 
 \sum_{\ee=1}^d(\nabla_{X_\ee} \emph{T}) (X_\ee,V)
\]
Motivated by the examples of the previous section, we set the following definition:

\begin{definition}
The sub-Riemannian manifold $\M$ is said to be of Yang-Mills type, if for every horizontal vector field $X$,
\[
\delta T(X)=0.
\]
\end{definition}

For instance, Riemannian manifolds, CR Sasakian manifolds and Carnot groups of step 2 are examples of sub-Riemannian manifolds with transverse symmetries of Yang-Mills type.

\begin{proposition}\label{improveCD}
Suppose that $\M$ is of Yang-Mills type and that there exist constants $\rho_1
\in \mathbb{R}$, $\rho_2
>0$ and $\kappa \ge 0$ such that 
\[
\begin{cases}
\mathcal{R}(f) \ge \rho_1 \Gamma (f) +\rho_2 \Gamma^Z (f),
\\
\mathcal{T}(f) \le \kappa \Gamma (f),
\end{cases}
\]
 hold for every $f\in C^\infty(\bM)$.
Then, for $f\in C^\infty(\bM)$, and $\nu >0$, one has
\[
\Gamma(\Gamma(f)) \le  4 \Gamma(f) \left( \Gamma_{2}(f)+\nu \Gamma^Z_{2}(f)  -   \left( \rho_1 -\frac{\kappa}{\nu}\right)  \Gamma (f)  \right),
\]
and
\[
\Gamma( \Gamma^Z(f)) \le 4  \Gamma^Z(f) \Gamma_2^Z(f).
\]
\end{proposition} 

\begin{proof}
Let $f\in C^\infty(\bM)$ and $x_0 \in \M$. We assume that $ \nabla_\mathcal{H} f(x_0)\neq 0$, otherwise the inequality is straightforward. We can find a local adapted frame $\{ X_1,\cdots ,X_d, Z_1,\cdots,Z_\di \}$ in the neighborhood of $x_0$ such that
\[
X_1 f=0,\cdots, X_d f=\| \nabla_\mathcal{H} f \| .
\]
In this frame, we have
\[
\Gamma(\Gamma(f))= 4\left(  \sum_{i=1}^d (X_i X_d f)^2\right) \Gamma(f),
\]
and
\[
\| \nabla_\mathcal{H}^2 f \|^2 = \sum_{\ee=1}^d \left( X^2_\ee f-\sum_{i=1}^d
\omega_{i\ee}^\ee X_i f \right)^2 +  2 \sum_{1 \le \ee<j \le d}
\left( \frac  {X_j X_\ee + X_\ee X_j}{2}f -\sum_{i=1}^d \frac{\omega_{ij}^\ee
+\omega_{i\ee}^j}{2} X_i f \right)^2.
\]
By observing that $X_j X_\ee f=0 $ if $\ee \neq d$ and $X_j X_d f=\omega_{jd}^d X_d f +\sum_{m=1}^\di \gamma_{jd}^m Z_mf$ we easily reach the conclusion that
\begin{align*}
\Gamma(\Gamma(f)) -4 \| \nabla_\mathcal{H}^2 f \|^2 \Gamma(f)  & \le 2 \Gamma(f) \sum_{\ee=1}^d  \left( \sum_{m=1}^\di \gamma_{\ee d}^m Z_m f\right)^2 \\
 & \le 2 \Gamma(f) \sum_{1\le \ee<j\le d}\bigg(\sum_{m=1}^\di
\gamma^{m}_{\ee j} Z_{m} f\bigg)^2
\end{align*}
Now, from \eqref{bochner} in Theorem \ref{T:bochner} we have
\[
\Gamma_{2}(f)=\| \nabla_\mathcal{H}^2 f \|^2 +\mathcal{R}(f)+\mathcal{S}(f).  
\]
From this identity and the proof of Theorem \ref{T:cd}, we obtain 
for every $\nu>0$ 
\begin{equation}\label{cd2}
\Gamma_{2}(f)+\nu \Gamma^Z_{2}(f) \ge ||\nabla^2_{\mathcal H} f||^2  -\frac{\kappa}{\nu}  \Gamma (f) + \mathcal{R}(f).
\end{equation}
Therefore we have
\[
\Gamma(\Gamma(f)) \le  4 \Gamma(f) \left( \Gamma_{2}(f)+\nu \Gamma^Z_{2}(f)  +\frac{\kappa}{\nu}  \Gamma (f)-\mathcal{R}(f) +\frac{1}{2}  \sum_{1\le \ee<j\le d}\bigg(\sum_{m=1}^\di
\gamma^{m}_{\ee j} Z_{m} f\bigg)^2  \right).
\]
From the Yang-Mills assumption we have
\begin{align*}
 & \mathcal{R}(f) -\frac{1}{2}  \sum_{1\le \ee<j\le d}\bigg(\sum_{m=1}^\di
\gamma^{m}_{\ee j} Z_{m} f\bigg)^2\\
=& \sum_{k,\ee=1}^d \bigg\{\bigg(\sum_{j=1}^d
\sum_{m=1}^\di \gamma_{kj}^{m} \delta_{jm}^\ee\bigg) +
\sum_{j=1}^d
(X_\ee\omega^j_{kj} - X_j\omega^k_{\ee j}) \\
& + \sum_{i,j=1}^d \omega_{ji}^i \omega^\ee_{k j} - \sum_{i=1}^d
\omega_{k i}^i \omega_{\ee i}^i + \frac{1}{2} \sum_{1\le i<j\le d}
\bigg(\omega^\ee_{ij} \omega^k_{ij} - (\omega_{\ee j}^i +\omega_{\ee
i}^j)(\omega^i_{kj} + \omega^j_{ki})\bigg)\bigg\}X_k f X_\ee f,
\end{align*}
and thus
\[
\mathcal{R}(f) -\frac{1}{2}  \sum_{1\le \ee<j\le d}\bigg(\sum_{m=1}^\di
\gamma^{m}_{\ee j} Z_{m} f\bigg)^2 \ge \rho_1 \Gamma(f).
\]
Puttings things together, we conclude
\[
\Gamma(\Gamma(f)) \le  4 \Gamma(f) \left( \Gamma_{2}(f)+\nu \Gamma^Z_{2}(f)  -   \left( \rho_1 -\frac{\kappa}{\nu}\right)  \Gamma (f)  \right).
\]
The proof of
\[
\Gamma( \Gamma^Z(f)) \le 4  \Gamma^Z(f) \Gamma_2^Z(f).
\]
is easy and let to the reader.
\end{proof}

In the sequel of this  section we assume that $\M$ is complete and that there exist constants $\rho_1
\in \mathbb{R}$, $\rho_2
>0$ and $\kappa \ge 0$ such that \eqref{riccibounds} hold for every $f\in C^\infty(\bM)$.
The assumed completeness  of $\M$ implies that the Hypothesis \ref{A:exhaustion} be satisfied, that is that there exists an increasing
sequence $h_k\in C^\infty_0(\bM)$   such that $h_k\nearrow 1$ on
$\bM$, and
 \[
||\Gamma (h_k)||_{\infty} +||\Gamma^Z (h_k)||_{\infty}  \to 0,\ \ \text{as} \ k\to \infty.
\]
Following an argument of Strichartz \cite{Strichartz}, (Theorem 7.3
p. 246 and p. 261), this implies that the operators $L$ and $L+L^Z$ are both essentially self-adjoint on the space $C^\infty_0(\bM)$, where we have let
\[
L^Z=-\sum_{m=1}^\di Z_m^* Z_m.
\]
In the sequel, we will denote by $\mathcal{D}(L)$ the domain of the self-adjoint extension of $L$.
\begin{lemma}\label{L:LLZ}
The operators $L$ and $L + L^Z$ spectrally commute, that is for any bounded Borel function $\Psi: (-\infty,0] \rightarrow \mathbb{R}$ and any $f \in L^2(\M)$,
\[
\Psi( L) \Psi (L+L^Z) f=   \Psi (L+L^Z) \Psi( L) f.
\]
\end{lemma}

\begin{proof}
Let $ f \in C_0^\infty(\mathbb{M})$. We first observe that
\begin{equation}\label{pos}
\int_\mathbb{M} \Gamma^Z(f, Lf) d\mu \le 0.
\end{equation}
To see this we note that, thanks to Lemma \ref{L:comm}, we have
\[
2 \int_\M \Gamma^Z(f,Lf) d\mu = \int_M L\Gamma^Z(f) d\mu - 2 \sum_{m=1}^\di \int_\M \Gamma(Z_m f) d\mu = - 2 \sum_{m=1}^\di \int_\M \Gamma(Z_m f) d\mu \le 0.
\]
Next, we observe that for every $f, g\in C^\infty_0(\M)$ we have
\[
0 = \int_\M L^Z(fg)d\mu = \int_\M f L^Z g d\mu + \int_\M g L^Z f d\mu + 2 \int_\M \Gamma^Z(f,g) d\mu.
\]
With $f\in C^\infty_0(\M)$ and $g = Lf$, this gives
\begin{align*}
- 2 \int_\M \Gamma^Z(f,Lf) d\mu = 2 \int_\M Lf L^Z f d\mu.
\end{align*}
In view of \eqref{pos} this gives for any $f\in C^\infty_0(\M)$
\[
\int_\M Lf L^Z f d\mu \ge 0.
\]
In turn, this implies for all $f\in C^\infty_0(\M)$
\begin{equation}\label{boundkj}
\int_\M (Lf)^2 d\mu \le \int_\M (Lf +L^Zf)^2 d\mu.
\end{equation}
But then, the inequality \eqref{boundkj} continues to be true for $f \in \mathcal{D}(L+L^Z)$.
Let now $f \in \mathcal{D}(L)$ and consider  the function,
\[
\phi(x,t)=LQ_t f(x),
\]
where $Q_t$ is the heat semigroup associated with $L+L^Z$. Since $L$ and $L+L^Z$ commute on smooth functions (see Lemma \ref{L:comm}), we easily see that $\phi$ solves the heat equation
\[
\frac{\partial \phi}{\partial t}=(L+L^Z)\phi,
\]
with initial condition $\phi(x,0)=Lf(x)$. From \eqref{boundkj}, we have that for every $t \ge 0$, $\int_\mathbb{M} \phi(x,t)^2 d\mu <\infty$. Thus by uniqueness in $L^2$ of solutions of the heat equation, we conclude that $\phi(x,t)=LQ_t f(x)=Q_t Lf(x)$. By a similar argument,  we may prove that for every $f \in L^2(\mathbb{M})$, $s,t \ge 0$
\[
P_sQ_t f =Q_t P_s f,
\]
which implies that $L$ and $L+L^Z$ spectrally commute, see Reed and Simon \cite{reed1} (Chapter 8, Section 5).

\end{proof}

\begin{lemma}\label{spectralp}
There is a constant $C=C(\rho_1,\rho_2,\kappa)>0$ such that for every smooth function $f$ belonging to $\mathcal{D}(L^2)$, one has
\[
0 \le - \int_\mathbb{M} \Gamma^Z(f, Lf) d\mu \le C  \| f \|^2_{\mathcal{D}(L^2)} ,
\]
where
\[
  \| f \|^2_{\mathcal{D}(L^2)} =\int_\mathbb{M} f^2 d\mu+ \int_\mathbb{M} (L^2f)^2 d\mu.
 \]

\end{lemma}

\begin{proof}

  From Theorem \ref{T:cd} we have for every $\nu>0$
\[
\Gamma_2(f) +\nu \Gamma_2^Z(f) \ge \left( \rho_1-\frac{\kappa}{\nu} \right) \Gamma(f)+\rho_2\Gamma^Z(f).
\]
Since
\[
2\Gamma_2(f)=L \Gamma(f)-2\Gamma(f,Lf),
\]
and
\[
2\Gamma^Z_2(f)=L \Gamma^Z(f)-2\Gamma^Z(f,Lf),
\]
we deduce by an integration over $\mathbb{M}$ that for $\nu >0$,
\[
\int_\mathbb{M} (Lf)^2 d\mu +\nu \int_\mathbb{M} Lf L^Z f d\mu \ge \left( \rho_1-\frac{\kappa}{\nu} \right) \int_\mathbb{M} \Gamma(f) d\mu+\rho_2\int_\mathbb{M}\Gamma^Z(f) d\mu.
\]
(One should keep in mind that, since $f\in C_0^\infty(\M)$, we have $\int_\M L\Gamma(f)d\mu = \int_\M L\Gamma^Z(f)d\mu = 0$, and that $-\int_\M \Gamma(f,Lf)d\mu = \int_\M (Lf)^2 d\mu$, $-\int_\M \Gamma^Z(f,Lf)d\mu = \int_\M Lf L^Z f d\mu$.)
The latter inequality can be re-written as
\[
\int_\mathbb{M} (Lf)^2 d\mu +\nu \int_\mathbb{M} Lf  L^Z f d\mu \ge \left( \rho_1-\frac{\kappa}{\nu} \right) \int_\mathbb{M} (-Lf)f d\mu+\rho_2\int_\mathbb{M} (-L^Z f)f  d\mu.
\]
From  Lemma \ref{L:comm}, the diffusion operators $L$ and $L+L^Z$ spectrally commute, therefore from the spectral theorem, 
 there is a measure space $(\Omega, \alpha)$, a unitary map $U: \mathbf{L}_{\alpha}^2 (\Omega,\mathbb{R})  \rightarrow L^2 ( \mathbb{M})$ and   real valued measurable functions $\lambda$ and $\lambda^Z$ on $\Omega$ such that for $x \in \Omega$,
\[
U^{-1} L U g (x)=-\lambda(x) g(x), 
\]
\[
U^{-1} L^Z U g (x)=-\lambda^Z(x) g(x).
\]
From the previous inequality, we obtain
\[
\| \lambda U^{-1} f \|^2_{L^2_\alpha} +\nu \langle \lambda U^{-1}f , \lambda^Z  U^{-1}f \rangle_{L^2_\alpha}  \ge \left( \rho_1-\frac{\kappa}{\nu} \right) \langle \lambda U^{-1}f ,   U^{-1}f \rangle_{L^2_\alpha} +\rho_2 \langle \lambda^Z U^{-1}f ,   U^{-1}f \rangle_{L^2_\alpha}.
\]
Since it holds for every smooth and compactly supported functions, we deduce that for every $\nu >0$, we have almost everywhere with respect to $\alpha$,
\[
\lambda^2(x) +\nu \lambda^Z(x) \lambda(x)  \ge \left( \rho_1-\frac{\kappa}{\nu} \right)\lambda(x)+\rho_2 \lambda^Z(x).
\] 
In particular, by choosing
\[
\nu=\rho_2 (\lambda(x)+1)^{-1},
\]
we obtain the following inequality on the spectral measures
\begin{align}\label{spectrald}
\frac{\rho_2 \lambda^Z}{\lambda+1} \le -\left( \rho_1-\frac{\kappa}{\rho_2}\right)\lambda +\left(1+\frac{\kappa}{\rho_2}\right)\lambda^2.
\end{align}
As a consequence, for any $f \in \mathcal{D}(L^2)$,
\begin{align}\label{boundZ}
\rho_2 \int_{\M} (-L^Z f) f d\mu \le&  -\left( \rho_1-\frac{\kappa}{\rho_2}\right)\left(  \int_{\M} (-L f) f d\mu+  \int_{\M} (L f)^2 d\mu\right) \\
& +\left(1+\frac{\kappa}{\rho_2}\right)\left(  \int_{\M} (L f)^2 d\mu+  \int_{\M} (-L f)(L^2f) d\mu\right) \notag
\end{align}
By denoting
\[
\mathbf{R}=\rho_2 (-L+Id)^{-1},
\]
we also deduce from \ref{spectrald} that for every $f \in \mathcal{D}(L)$,
\[
\rho_2 \int_\mathbb{M} (-L^Z f) (\mathbf{R} f) d\mu \le -\left( \rho_1-\frac{\kappa}{\rho_2} \right) \int_\mathbb{M} -f Lf d\mu+\left(1+\frac{\kappa}{\rho_2}\right)\int_\mathbb{M} (Lf)^2  d\mu.
\]
By using now the above inequality with $-Lf+f$ instead of $f$, and using (\ref{boundZ}) we obtain the desired inequality.

\end{proof}

\begin{remark}
The previous proof also shows the following inclusion of domains:
\[
\mathcal{D} (L^2) \subset \mathcal{D}(L+L^Z) \subset \mathcal{D}(L).
\]
\end{remark}

As a consequence of the previous inequality , we obtain the following useful a priori bounds.

\begin{proposition}\label{a:priori:bounds}
There exists a positive constant $C=C(\rho_1,\rho_2,\kappa)>0$ such that for every smooth function $f$ belonging to $\mathcal{D}(L^2)$,
\[
\int_\bM \Gamma^Z(f) d\mu \le C \| f \|^2_{\mathcal{D}(L^2)},
\]
\[
\int_\bM \Gamma^Z_2 (f) d\mu \le C \| f \|^2_{\mathcal{D}(L^2)},
\]
\[
\int_\bM \left( \Gamma_{2}(f)+ \Gamma^Z_{2}(f)  -   \left( \rho_1 -\kappa \right)  \Gamma (f) \right) d\mu \le C \| f \|^2_{\mathcal{D}(L^2)}.
\]

\end{proposition}

\begin{proof}
Let $f \in C^\infty_0(\bM)$. According to Lemma  \ref{spectralp},  we have 
\[
\int_\bM \Gamma_2^Z(f) d\mu=  - \int_\mathbb{M} \Gamma^Z(f, Lf) d\mu \le  C_1 \| f \|^2_{\mathcal{D}(L^2)}.
\]
Then, we get
\[
\int_\bM \Gamma_2(f) d\mu=  - \int_\mathbb{M} \Gamma(f, Lf) d\mu = \int_{\bM} (Lf)^2 d\mu \le  \| f \|^2_{\mathcal{D}(L^2)},
\]
and 
\[
\int_\mathbb{M} (Lf)^2 d\mu +\nu \int_\mathbb{M} Lf L^Z f d\mu \ge \left( \rho_1-\frac{\kappa}{\nu} \right) \int_\mathbb{M} \Gamma(f,f) d\mu+\rho_2\int_\mathbb{M}\Gamma^Z(f,f) d\mu.
\]
which implies
\[
\int_\bM \Gamma^Z(f) d\mu \le  C_2 \| f \|^2_{\mathcal{D}(L^2)}.
\]
Putting things together, we conclude that for $f \in C_0^\infty(\bM)$,
\[
\int_\bM \Gamma^Z(f) d\mu \le C \| f \|^2_{\mathcal{D}(L^2)},
\]
\[
\int_\bM \Gamma^Z_2 (f) d\mu \le C \| f \|^2_{\mathcal{D}(L^2)},
\]
\[
\int_\bM \left( \Gamma_{2}(f)+ \Gamma^Z_{2}(f)  -   \left( \rho_1 -\kappa \right)  \Gamma (f) \right) d\mu \le C \| f \|^2_{\mathcal{D}(L^2)}.
\]
The inequalities are then extended to the smooth functions of $\mathcal{D}(L^2)$ by using the essential self-adjointness of $L$ which implies the density of $ C_0^\infty(\bM)$ in   $\mathcal{D}(L^2)$ and the  same  arguments as in Bakry \cite{bakry-CRAS,bakry-stflour}. The details are let to the reader.
\end{proof}


\section{The heat semigroup and parabolic comparison theorems}\label{S:comp}

We now return to the general framework described in the introduction. Hereafter in this paper, $\bM$ will be a $C^\infty$ connected manifold endowed with a smooth measure $\mu$ and a smooth, locally subelliptic operator $L$ satisfying $L1=0$ and \eqref{sa}. We indicate with $\Gamma(f)$ the quadratic differential form defined by \eqref{gamma} and  denote by $d(x,y)$ the canonical distance \eqref{di} associated with such form. As we have said in the introduction throughout this paper we assume that $(\M,d)$ be a complete metric space. Furthermore, we assume that $\bM$ be endowed with another smooth bilinear differential form, indicated with $\Gamma^Z$, satisfying \eqref{Zleib} above.
We thus have, in particular, $\Gamma^Z(1) = 0$. As stated in the introduction, we assume that $\Gamma^Z(f) \ge 0$ for every $f\in C^\infty(\M)$.

From \eqref{sa} we have that, as an operator defined on $C^\infty_0(\bM)$, $L$ is
symmetric with respect to the measure $\mu$ and non-positive: for $f
\in C^\infty_0(\bM)$, $<Lf,f> \le 0$.

Then, 
following an argument of Strichartz \cite{Strichartz}, Theorem 7.3
p. 246 and p. 261, by using the completeness of $(\bM,d)$, we conclude  that $L$ is essentially self-adjoint
on $C^\infty_0(\bM)$.
As a consequence, $L$ admits a unique self-adjoint extension (its Friedrichs extension). We shall continue to denote such extension by $L$. The domain of this extension shall be denoted by $\mathcal{D}(L)$. 

Hereafter, for $1\le p \le \infty$ we will write $L^p(\M)$ instead of $L^p(\M,\mu)$. If $L=-\int_0^{\infty} \lambda dE_\lambda$ denotes the spectral
decomposition of $L$ in $L^2 (\bM)$, then by definition, the
heat semigroup $(P_t)_{t \ge 0}$ is given by $P_t= \int_0^{\infty}
e^{-\lambda t} dE_\lambda$. It is a one-parameter family of bounded operators on
$L^2 (\bM)$. Since the quadratic form $\mathcal Q(f) = -<f,Lf>$ is a Dirichlet
form in the sense of Fukushima \cite{Fu}, we deduce 
that $(P_t)_{t \ge 0}$ is a sub-Markov semigroup: it transforms
positive functions into positive functions and satisfies
\begin{equation}\label{submarkov}
P_t 1 \le 1.
\end{equation}
This property implies in particular \begin{equation}\label{sminfty}
||P_tf||_{L^1(\bM)} \le ||f||_{L^1(\bM)},\ \ \
||P_tf||_{L^\infty(\bM)} \le ||f||_{L^\infty(\bM)},
\end{equation}
and therefore by the Theorem of Riesz-Thorin
\begin{equation}\label{smp}
||P_tf||_{L^p(\bM)} \le ||f||_{L^p(\bM)},\ \  1\le p\le \infty.
\end{equation}

From the spectral definition of $P_t$,  it is clear that for every $t>0$, and every $f \in L^2(\bM)$, $P_tf \in \mathcal{D}_\infty (L)=\cap_{k \ge 1} \mathcal{D}(L^k)$.
Moreover, it can be shown as in \cite{LI}:

\begin{proposition}\label{uniquenessLp}
The unique solution of the Cauchy problem
\[
\begin{cases}
\frac{\p u}{\p t} - Lu = 0,
\\
u(x,0) = f(x),\ \ \ \   f\in L^p(\bM),1<p<\infty,
\end{cases}
\]
that satisfies $\| u(\cdot,t) \|_p <\infty$, is given by $u(x,t)=P_t f(x)$.
\end{proposition}
Due to the hypoellipticity of $L$ the function $(x,t) \rightarrow P_t f(x)$ is
smooth on $\mathbb{M}\times (0,\infty) $ and
\[ P_t f(x)  = \int_{\mathbb M} p(x,y,t) f(y) d\mu(y),\ \ \ f\in
C^\infty_0(\mathbb M),\] where $p(x,y,t) > 0$ is the so-called heat
kernel associated to $P_t$. Such function is smooth and it is symmetric, i.e., \[ p(x,y,t)
= p(y,x,t). \]
 By the
semi-group property for every $x,y\in \bM$ and $0<s,t$, we have
\begin{align}\label{sgp}
p(x,y,t+s) & = \int_\bM p(x,z,t) p(z,y,s) d\mu(z)  
\\
& = \int_\bM p(x,z,t)
p(y,z,s) d\mu(z) = P_s(p(x,\cdot,t))(y).
\notag
\end{align}

We first establish a global comparison theorem in $L^2$.

\begin{proposition}\label{P:missing_key}
Suppose that $\M$ satisfy the Hypothesis \ref{A:exhaustion}. Let $T>0$. Let $u,v: \mathbb{M}\times [0,T] \to \mathbb{R}$ be  smooth functions such that:
\begin{itemize}
\item[(i)]  For every $t \in [0,T]$, $u(\cdot,t) \in L^2(\bM)$ and $\int_0^T \| u(\cdot,t)\|_2 dt <\infty$;
\item[(ii)]  $\int_0^T \| \sqrt{\Gamma(u) (\cdot,t)} \|_p dt <\infty$ for some $1 \le p \le \infty$;
\item[(iii)]  For every $t \in [0,T]$, $v(\cdot,t) \in L^q(\bM)$ and $\int_ 0^T \| v(\cdot,t ) \|_q dt <\infty$ for some $1 \le q \le \infty$. 
\end{itemize}
 If the  inequality 
\[
Lu+\frac{\partial u}{\partial t} \ge v,
\]
holds on $\mathbb{M}\times [0,T]$, then we have
\[
P_T u(\cdot,T)(x) \ge u(x,0) +\int_0^T P_s v(\cdot,s)(x) ds.
\]
\end{proposition}

\begin{proof}
Let $f,g \in C_0^\infty (\mathbb{M})$, $f,g \ge 0$. We claim that we must have
\begin{align}\label{cp1}
& \int_\mathbb{M} g P_T(fu(\cdot,T)) d\mu - \int_\mathbb{M} g f u(x,0) d\mu
\ge   -  \|\sqrt{\Gamma(f)}\|_\infty \int_0^T  \int_\mathbb{M}   (P_t g) \sqrt{\Gamma(u)}d\mu dt
\\
& -  \| \sqrt{\Gamma(f)} \|_\infty \int_0^T \| \sqrt{\Gamma(P_t g) }\|_2 \| u(\cdot,t) \|_2   dt  +  \int_\mathbb{M} g \int_0^T  P_t( f v(\cdot,t)) d\mu dt,
\notag
\end{align}
where for every $1\le p \le \infty$ and a measurable $F$, we have let $||F||_p = ||F||_{L^p(\M)}$.
To establish \eqref{cp1} we consider the function
\[
\phi(t)=\int_\mathbb{M} g P_t (fu(\cdot,t)) d\mu.
\]
Differentiating $\phi$ we find
\begin{align*}
\phi'(t)& =\int_\mathbb{M} g P_t \left(L( fu) + f\frac{\partial u}{\partial t} \right) d\mu \\
 &= \int_\mathbb{M} g P_t \left((L f) u+2 \Gamma (f,u) +f Lu + f\frac{\partial u}{\partial t} \right) d\mu \\
 &\ge \int_\mathbb{M} g P_t \left((L f) u+2 \Gamma (f,u)  \right) d\mu+\int_\mathbb{M} g P_t( f v) d\mu.
\end{align*}
Since
\begin{align*}
\int_\mathbb{M} g P_t \left((L f) u\right) d\mu &= \int_\mathbb{M}  (P_t g) (L f) u d\mu \\
 &= -\int_\mathbb{M}  \Gamma( f, u(P_t g)) d\mu  \\
 &=-\left( \int_\mathbb{M}  P_t g \Gamma( f, u)+ u \Gamma(f,P_t g) d\mu\right),
\end{align*}
we obtain
\[
\phi'(t) \ge \int_\M P_t g \Gamma(f,u) d\mu - \int_\M u \Gamma(f,P_t g) d\mu + \int_\M g P_t(fv) d\mu.
\]
Now, we can bound
\[
\left| \int_\mathbb{M}  (P_t g) \Gamma( f, u) d\mu\right| 
 \le    \| \sqrt{\Gamma(f) } \|_\infty \int_\mathbb{M}  (P_t g)  \sqrt{\Gamma( u)} d\mu,
\]
and for a.e. $t\in [0,T]$ the integral in the right-hand side is finite in view of the assumption (ii) above.
We have thus obtained
\begin{align*}
\phi'(t)  \ge -  \| \sqrt{\Gamma(f) } \|_\infty  \int_\mathbb{M}  (P_t g)  \sqrt{\Gamma(u)} d\mu- \int_\mathbb{M} u \Gamma(f , P_t g) d\mu+ \int_\mathbb{M} g P_t( f v(\cdot,t)) d\mu.
\end{align*}
As a consequence, we find
\begin{align*}
 & \int_\mathbb{M} g P_T (fu(\cdot,T)) d\mu-\int_\mathbb{M} gfu(x,0)d\mu  \\
 \ge&  -  \| \sqrt{\Gamma(f) } \|_\infty   \int_0^T  \int_\mathbb{M}   (P_t g) \sqrt{\Gamma(u)} d\mu dt -  \int_0^T \int_\mathbb{M}  u \Gamma\left(f,P_t g\right) d\mu dt + \int_0^T  \int_\mathbb{M} g  P_t(f v(\cdot,t)) d\mu dt
 \\
 \ge &  -  \| \sqrt{\Gamma(f) } \|_\infty   \int_0^T \int_\mathbb{M} (P_t g) \sqrt{\Gamma(u)} d\mu dt - \int_0^T \| u(\cdot,t) \|_2 \| \Gamma(f, P_t g) \|_2 dt  + \int_\mathbb{M} g \int_0^T  P_t( f v(\cdot,t)) dt d\mu \\
 \ge &    -  \| \sqrt{\Gamma(f) } \|_\infty   \int_0^T  \int_\mathbb{M} (P_t g) \sqrt{\Gamma(u)} d\mu dt -  \| \sqrt{\Gamma(f)} \|_\infty \int_0^T \| u(\cdot,t) \|_2  \| \sqrt{\Gamma(P_t g) }\|_2 dt \\ +  & \int_\mathbb{M} g \int_0^T  P_t(f v(\cdot,t)) dt d\mu,
\end{align*}
which proves \eqref{cp1}.
Let now $h_k\in C^\infty_0(\M)$ be a sequence as in Hypothesis \ref{A:exhaustion}.                
Using $h_k$ in place of $f$ in \eqref{cp1}, and letting $k \to \infty$, gives
\begin{align*}
  \int_\mathbb{M} g P_T (u(\cdot,T)) d\mu - \int_\mathbb{M} g u(x,0)d\mu  
 \ge   \int_\mathbb{M} g \int_0^T  P_t(v(\cdot,t)) dt d\mu.
\end{align*}
We observe that the assumption on $v$ and Minkowski's integral inequality guarantee that the function $x\to \int_0^T  P_t(v(\cdot,t))(x) dt$ belongs to $L^q(\M)$. We have in fact
\begin{align*}
\left(\int_\mathbb{M} \left|\int_0^T  P_t(v(\cdot,t)) dt\right|^q d\mu\right)^{\frac 1q} & \le \int_0^T \left| \int_\mathbb{M} \left|P_t(v(\cdot,t))\right|^q d\mu\right|^{\frac 1q} dt \le \int_0^T \left| \int_\mathbb{M} \left|v(\cdot,t)\right|^q d\mu\right|^{\frac 1q} dt
\\
& \le T^{\frac{1}{q'}}  \left(\int_0^T \int_\mathbb{M} \left|v(\cdot,t)\right|^q d\mu dt \right)^{\frac 1q} <\infty.
\end{align*}
Since this must hold for every non negative $g \in C_0^\infty (\mathbb{M})$, we conclude that
\[
P_T(u(\cdot,T))(x) \ge u(x,0) +\int_0^T P_s (v(\cdot,s))(x) ds,
\]
which completes the proof.
\end{proof}

The next theorem shows that Hypothesis  \ref{A:regularity} is redundant on complete sub-Riemannian manifolds with transverse symmetries of Yang-Mills type  if the  sub-Laplacian $L$ satisfies the generalized curvature dimension inequality.

\begin{theorem}\label{T:on revient tojour a son premiere amour}
Let $L$ be the sub-Laplacian on a complete sub-Riemannian manifold with transverse symmetries of Yang-Mills type. Suppose that $L$ satisfies \emph{CD}$(\rho_1,\rho_2,\kappa,d)$, for some $\rho_1\in \R$. Then, the Hypothesis \ref{A:regularity} is satisfied. 
\end{theorem}

\begin{proof}
Let $f \in C_0^\infty(\bM)$ and consider the functional
\[
\Phi(t)=\sqrt{ \Gamma^Z(P_{T-t} f)}.
\]
We first assume that $(x,t)\to \Gamma^Z(P_t f)(x)>0$ on $\M\times [0,T]$. From Proposition \ref{a:priori:bounds} we have $\Phi(t) \in L^2(\bM)$. Moreover $\Gamma(\Phi)(t)=\frac{\Gamma(\Gamma^Z(P_{T-t} f))}{4 \Gamma^Z(P_{T-t} f)}$. So from Proposition \ref{improveCD},  we have $\Gamma( \Phi)(t) \le \Gamma_2^Z(P_{T-t} f)$.  Therefore, again from Proposition \ref{a:priori:bounds} , we deduce that  $\Gamma( \Phi)(t) \in L^1(\bM)$. Next, we easily compute that
\[
\frac{\partial \Phi}{\partial t}+ L\Phi =\frac{\Gamma^Z_2(P_{T-t} f)}{\sqrt{\Gamma^Z(P_{T-t} f)}}-\frac{\Gamma(\Gamma^Z(P_{T-t} f))}{4 \Gamma^Z(P_{T-t} f)^{3/2} } .
\]
Thus, from Proposition \ref{improveCD}, we obtain that
\[
\frac{\partial \Phi}{\partial t}+ L\Phi \ge 0.
\]
We can then use Proposition \ref{P:missing_key} to infer that 
\[
\sqrt{ \Gamma^Z(P_{T} f)} \le P_T \left(\sqrt{\Gamma^Z (f)} \right).
\]
This implies that for every $t \ge 0$, $\Gamma^Z(P_t f) \in L^p(\bM)$ for every $1 \le p \le \infty$. If $(x,t) \to \Gamma^Z(P_t f)(x)$ vanishes on $\M\times [0,T]$, we consider the functional 
\[
\Phi(t)=g_\varepsilon (\Gamma^Z(P_{T-t} f) ),
\]
where, for $0 < \varepsilon <1$,
\begin{align*}
g_\varepsilon (y)=\sqrt{ y+\varepsilon^2}-\varepsilon.
\end{align*}
Since $\Phi(t) \in L^2(\M)$, an argument similar to that above (details are let to the reader) shows that
\[
g_\varepsilon (\Gamma^Z(P_{T} f) )\le  P_T \left( g_\varepsilon( \Gamma^Z (f)) \right).
\]
Letting $\varepsilon \to 0$, we conclude that 
\[
\sqrt{ \Gamma^Z(P_{T} f)} \le P_T \left(\sqrt{\Gamma^Z (f)} \right).
\]
Proving that $(x,t)\to \Gamma(P_t f)(x)$ is bounded is similar. For $\alpha \in \mathbb{R}$, we  consider the functional
\[
\Psi(t)=e^{-\alpha(T-t)} \left( \sqrt{ \Gamma(P_{T-t} f)}+\Gamma^Z(P_{T-t} f)\right), 
\]
and first assume that $(x,t)\to \Gamma(P_t f)(x)$ does not vanish on $\M \times [0,T]$.  From the previous inequality,  Proposition \ref{a:priori:bounds} and Proposition \ref{improveCD}, it is seen that  $\Psi(t) \in L^2(\bM)$ and $\sqrt{\Gamma( \Psi)(t)} \in L^1(\bM)+L^2(\bM)$. Moreover,
 \[
\frac{\partial \Psi}{\partial t}+ L\Psi =e^{-\alpha(T-t)} \left(\frac{\Gamma_2(P_{T-t} f)}{\sqrt{\Gamma(P_{T-t} f)}}-\frac{\Gamma(\Gamma(P_{T-t} f))}{4 \Gamma(P_{T-t} f)^{3/2} }+2\Gamma^Z_2(P_{T-t}f)  \right)+\alpha \Phi.
\]
According to  Proposition \ref{improveCD} we have for every $f\in C^\infty(\bM)$, and $\nu >0$,
\[
\Gamma(\Gamma(f)) \le  4 \Gamma(f) \left( \Gamma_{2}(f)+\nu \Gamma^Z_{2}(f)  -   \left( \rho_1 -\frac{\kappa}{\nu}\right)  \Gamma (f)  \right),
\]
Choosing $\nu=2\sqrt{\Gamma(f)}$ gives 
\[
\frac{\Gamma_2(f)}{\sqrt{\Gamma( f)}}-\frac{\Gamma(\Gamma(f))}{4 \Gamma( f)^{3/2}}+2\Gamma^Z_2(f) \ge \rho_1 \sqrt{\Gamma(f)}  -\frac{\kappa}{2}.
\]
We deduce 
 \[
\frac{\partial \Psi}{\partial t}+ L\Psi  \ge e^{-\alpha(T-t)} \left((\alpha+ \rho_1 )\sqrt{\Gamma(P_{T-t} f )}+\alpha \Gamma^Z(P_{T-t} f) \right) -\frac{\kappa}{2}e^{-\alpha(T-t)}.
\]
Therefore, by choosing $\alpha$ large enough we obtain
\[
\frac{\partial \Psi}{\partial t}+ L\Psi  \ge-\frac{\kappa}{2}e^{-\alpha(T-t)}.
\]
As a consequence of Proposition \ref{P:missing_key}, we find
\[
 \sqrt{ \Gamma(P_{T} f)}+\Gamma^Z(P_{T} f) \le e^{\alpha T} \left(  P_T(\sqrt{\Gamma}(f)) + P_T(\Gamma^Z(f)) \right)+\frac{\kappa}{2} e^{\alpha T} \int_0^T (P_s 1) ds.
\]
Since $P_s 1 \le 1$, we conclude therefore:
\[
 \sqrt{ \Gamma(P_{T} f)}+\Gamma^Z(P_{T} f) \le e^{\alpha T} \left(  P_T(\sqrt{\Gamma}(f)) + P_T(\Gamma^Z(f)) \right)+\frac{\kappa}{2} Te^{\alpha T}.
\]
This implies that $(x,t)\to \Gamma(P_{t} f)(x) +\Gamma^Z(P_t f)(x)\in L^\infty(\M \times [0,T])$. If $(x,t) \to\Gamma(P_t f)(x)$ does  vanish on $\M\times [0,T]$, then we consider the $C^\infty$ approximation of the square root as above.

\

We now prove that $P_t 1=1$, that is that $P_t$ is stochastically complete. A first consequence of the fact that for every $f \in C_0^\infty(\M)$, and $T \ge 0$,    $(x,t)\to \Gamma(P_{t} f)(x) +\Gamma^Z(P_t f)(x)\in L^\infty(\M \times [0,T])$ is that in   Proposition \ref{P:missing_key} we can now allow $u$ to be in $L^1$. 
More precisely, under the very same assumptions as in Proposition \ref{P:missing_key} where $(i)$ is replaced by: For every $t \in [0,T]$, $u(\cdot,t) \in L^1(\bM)$ and $\int_0^T \| u(\cdot,t)\|_1 dt <\infty$, we still have the conclusion 
\[
P_T(u(\cdot,T))(x) \ge u(x,0) +\int_0^T P_s(v(\cdot,s))(x) ds.
\]
The proof of this fact is identical to that of Proposition \ref{P:missing_key}.  With the notations of this proof, $\Gamma(P_{\cdot} g) \in L^\infty([0,T]\times\bM)$  is used to obtain the following bound  
\[
\left|\int_0^T \int_\mathbb{M} u \Gamma\left(f,P_t g\right) d\mu dt\right|  \le \|\sqrt{\Gamma(f)} \|_\infty  \int_0^T \|\sqrt{\Gamma(P_t g)}\|_\infty \| u(\cdot,t)\|_1 dt.
\]
This leads to an inequality where \eqref{cp1} is replaced by
\begin{align}\label{cp2}
& \int_\mathbb{M} g P_T(fu(\cdot,T)) d\mu - \int_\mathbb{M} g f u(x,0) d\mu
\ge   -  \|\sqrt{\Gamma(f)}\|_\infty \int_0^T  \int_\mathbb{M}   (P_t g) \sqrt{\Gamma(u)}d\mu dt
\\
& -  \| \sqrt{\Gamma(f)} \|_\infty \int_0^T \|\sqrt{\Gamma(P_t g)}\|_\infty \| u(\cdot,t)\|_1 dt  +  \int_\mathbb{M} g \int_0^T  P_t( f v(\cdot,t)) d\mu dt.
\notag
\end{align}
At this point the argument proceeds exactly as in the conclusion of the proof of Proposition \ref{P:missing_key}.

With this $L^1$ comparison result in hands,  we can now come back to the stochastic completeness problem.  Let    $f \in C_0^\infty(\mathbb M)$ and consider the functional
\[
u(x,t)=e^{\alpha (T-t)}\left( \Gamma(P_{T-t} f)(x)+\Gamma^Z(P_{T-t} f)(x)\right).
\]
We have
\[
Lu(x,t)= e^{\alpha (T-t)}\left(L \Gamma(P_{T-t} f)(x)+L\Gamma^Z(P_{T-t} f)(x)\right),
\]
and
\[
\frac{\partial u}{\partial t}(x,t)= -\alpha u (x,t) -2e^{\alpha (T-t)}\left( \Gamma(P_{T-t} f, LP_{T-t})(x)+\Gamma^Z(P_{T-t} f, LP_{T-t})(x)\right).
\]
Therefore we have
\[
Lu(x,t)+\frac{\partial u}{\partial t}(x,t) =-\alpha u (x,t) +2 e^{\alpha (T-t)}\left( \Gamma_2(P_{T-t} f)(x)+\Gamma_2^Z(P_{T-t} f)(x)\right).
\]
By using now the inequality  \emph{CD}$(\rho_1,\rho_2,\kappa,d)$ with $\nu=1$, we obtain
\[
Lu(x,t)+\frac{\partial u}{\partial t}(x,t) \ge  e^{\alpha (T-t)}\left((2(\rho_1-\kappa)-\alpha) \Gamma(P_{T-t} f)(x)+(2\rho_2-\alpha)\Gamma^Z(P_{T-t} f)(x)\right).
\]
By choosing $ \alpha \le 2\min\{\rho_2, \rho_1-\kappa\}$, we thus get 
\[
Lu(x,t)+\frac{\partial u}{\partial t}(x,t) \ge 0,
\]
and we conclude by using the $L^1$ version of  Proposition \ref{P:missing_key2} that
\begin{align}\label{pkliu}
\Gamma (P_t f)+\Gamma^Z (P_t f) \le  e^{-\alpha t} \left( P_t \Gamma(f)+ P_t \Gamma^Z(f)\right).
\end{align}
We are now ready for the final argument leading to the stochastic completeness.
Let $f,g \in  C^\infty_0(\mathbb M)$, by \eqref{sa} and \eqref{gamma} we have
\begin{align*}
\int_{\bM} (P_t f -f) g d\mu = \int_0^t \int_{\bM}\left(
\frac{\partial}{\partial s} P_s f \right) g d\mu ds= \int_0^t
\int_{\bM}\left(L P_s f \right) g d\mu ds=- \int_0^t \int_{\bM}
\Gamma ( P_s f , g) d\mu ds.
\end{align*}
By means of Cauchy-Schwarz inequality and \eqref{pkliu}, we find
\begin{equation}\label{P1}
\left| \int_{\bM} (P_t f -f) g d\mu \right| \le \left(\int_0^t
e^{-\frac{\alpha s}{2}} ds\right) \sqrt{ \| \Gamma (f) \|_\infty + \|
\Gamma^Z (f) \|_\infty } \int_{\bM}\Gamma (g)^{\frac{1}{2}}d\mu.
\end{equation}
We now apply \eqref{P1} with $f = h_k$, where $h_k$ is the sequence whose existence is postulated in 
the Hypothesis  \eqref{A:exhaustion}, and then let $k\to \infty$.
By Beppo Levi's monotone convergence theorem we have $P_t
h_k(x)\nearrow P_t 1(x)$ for every $x\in \bM$. We conclude that the
left-hand side of \eqref{P1} converges to $\int_{\bM} (P_t 1 -1) g d\mu$. Since in view of the Hypothesis \eqref{A:exhaustion} the right-hand side converges to zero, we reach the conclusion
\[
\int_{\bM} (P_t 1 -1) g d\mu=0,\ \ \ g\in C^\infty_0(\bM).
\]
It follows that $P_t 1 =1$.

\end{proof}

We point out that the stochastic completeness of the heat semigroup  is classically equivalent to the uniqueness in the Cauchy
problem for initial data in $L^\infty(\bM)$. Following the classical approach (see for instance Theorem 8.18 in \cite{Gri}), we in fact obtain:

\begin{proposition}\label{P:ucp}
Suppose that $\bM$ satisfy Hypothesis \ref{A:exhaustion}, Hypothesis \ref{A:regularity}. Then, for every $f\in L^\infty(\bM)$ the Cauchy problem
\[
\begin{cases}
Lu - u_t = 0, \ \ \  \text{in}\ \M \times (0,\infty),
\\
u(x,0) = f(x),\ \ \ \  f\in L^\infty(\bM),
\end{cases}
\]
admits a unique bounded solution, given by $u(x,t)=P_t f(x)$.
\end{proposition}

We state the following $L^\infty$ global parabolic comparison theorem that will be easier to use that Proposition \ref{P:missing_key} because it does not require a priori bounds on the derivatives.

\begin{proposition}\label{P:missing_key2}
Suppose that $\bM$ satisfy  Hypothesis \ref{A:regularity}. Let $T>0$. Let $u,v: \mathbb{M}\times [0,T] \to \mathbb{R}$ be  smooth functions such that  for every $T>0$,  $\sup_{t \in [0,T]} \| u(\cdot,t)\|_\infty <\infty$, $\sup_{t \in [0,T]} \| v(\cdot,t)\|_\infty <\infty$; If the  inequality 
\[
Lu+\frac{\partial u}{\partial t} \ge v
\]
holds on $\mathbb{M}\times [0,T]$, then we have
\[
P_T(u(\cdot,T))(x) \ge u(x,0) +\int_0^T P_s(v(\cdot,s))(x) ds.
\]

\end{proposition}

\begin{proof}
Let $(X^x_t)_{t \ge 0}$ be the diffusion Markov process with semigroup $(P_t)_{t \ge 0}$ and started at $x \in \M$ (see for instance Chapter 7 in \cite{ Fu} for the construction of such process). From $P_t1=1$, we deduce that  $(X^x_t)_{t\ge 0}$ has an infinite lifetime. We have then for $t \ge 0$,
\[
u\left( X^x_t, t \right)=u\left( x,0\right)+\int_0^t \left( Lu+\frac{\partial u}{\partial t}\right)(X^x_s,s) ds +M_t,
\]
where $(M_t)_{t \ge0 }$ is a local martingale. From the assumption one obtains
\[
u\left( X^x_t, t \right) \ge u\left( x,0\right)+\int_0^t v(X^x_s,s) ds +M_t.
\]
Let now $(T_n)_{n \in \mathbb{N}}$ be an increasing sequence of stopping times such that almost surely $T_n \to +\infty$ and $(M_{t\wedge T_n})_{t \ge 0}$ is a martingale.
From the previous inequality, we find
\[
\mathbb{E}\left( u\left( X^x_{t\wedge T_n}, t\wedge T_n \right) \right) \ge u\left( x,0\right)+\mathbb{E}\left( \int_0^{t\wedge T_n} v(X^x_s,s) ds\right).
\]
By using the dominated convergence theorem, we conclude
\[
\mathbb{E}\left( u\left( X^x_{t}, t \right) \right) \ge u\left( x,0\right)+\mathbb{E}\left( \int_0^{t} v(X^x_s,s) ds\right),
\]
which yields the conclusion.
\end{proof}

For later use, we also finally record the following gradient bounds that are consequences of Hypothesis \ref{A:regularity}.

\begin{corollary}\label{C:expdecay}
Suppose that $L$ satisfies \emph{CD}$(\rho_1,\rho_2,\kappa,d)$, for some $\rho_1\in \R$ and that Hypothesis \ref{A:regularity} is satisfied. There exists $\alpha \in \mathbb{R}$ ($ \alpha \le 2\min\{\rho_2, \rho_1
-\kappa\}$ will do), such that for   every   $f \in C^\infty_0(\M)$, one has
\begin{equation}\label{pointwiseCaccioppoli}
\Gamma (P_t f)+\Gamma^Z (P_t f) \le  e^{-\alpha t} \left( P_t
\Gamma(f)+ P_t \Gamma^Z(f)\right).
\end{equation}
As a consequence,  for every $f \in C_0^\infty(\mathbb
M)$ and  $1\le p\le \infty$ one obtains
\begin{equation}\label{LpCaccioppoli}
\| \Gamma (P_t f) \|_{L^p(\bM)} \le e^{-\alpha t} \left( \| \Gamma (f)
\|_{L^p(\bM)} +  \| \Gamma^Z (f) \|_{L^p(\bM)} \right), \quad t \ge
0.
\end{equation}
and
\begin{equation}\label{LpVerticalCaccioppoli}
\| \Gamma^Z (P_t f) \|_{L^p(\bM)} \le e^{-\alpha t} \left( \| \Gamma (f)
\|_{L^p(\bM)} +  \| \Gamma^Z (f) \|_{L^p(\bM)} \right), \quad t \ge
0.
\end{equation}
\end{corollary}

\begin{proof}

The proof is identical to the proof of (\ref{pkliu}) except that we now use Proposition \ref{P:missing_key2}.
\end{proof}

\section{Entropic variational inequalities}\label{S:EI}

Our objective in this section is proving a fundamental variational inequality which will play a pervasive role in our study, see Theorem \ref{T:source} below. 
We begin with some preliminary results. Henceforth, we will indicate $C_b^\infty(\mathbb M) = C^\infty(\M)\cap L^\infty(\M)$.

\begin{lemma}\label{L:derivatives}
Let $f \in C^\infty_b(\mathbb{M})$, $f > 0$ and $T>0$, and consider the functions
\[
\phi_1 (x,t)=(P_{T-t} f) (x)\Gamma (\ln P_{T-t}f)(x),
\]
\[
\phi_2 (x,t)= (P_{T-t} f)(x) \Gamma^Z (\ln P_{T-t}f)(x),
\]
which are defined on $\M\times (-\infty,T)$.  We have
\[
L\phi_1+\frac{\partial \phi_1}{\partial t} =2 (P_{T-t} f) \Gamma_2 (\ln P_{T-t}f). 
\]
If, furthermore, the Hypothesis \eqref{A:main_assumption} is valid, then 
\[
L\phi_2+\frac{\partial \phi_2}{\partial t} =2 (P_{T-t} f) \Gamma_2^Z (\ln P_{T-t}f).
\]
\end{lemma}

\begin{proof} Let for simplicity $g(x,t) = P_{T-t} f(x)$. A simple computation gives
\[
\frac{\p \phi_1}{\p t} = g_t \Gamma(\ln g) + 2 g \Gamma(\ln g,\frac{g_t}{g}).
\]
On the other hand,
\[
L\phi_1 = Lg \Gamma(\ln g) + g L \Gamma(\ln g) + 2 \Gamma(g,\Gamma(\ln g)).
\]
Combining these equations we obtain
\[
L\phi_1 + \frac{\p \phi_1}{\p t} = g L\Gamma(\ln g) +  2\Gamma(g,\Gamma(\ln g)) + 2 g \Gamma(\ln g,\frac{g_t}{g}).
\]
From \eqref{gamma2} we see that
\begin{align*}
2 g \Gamma_2(\ln g) & = g (L \Gamma(\ln g) - 2 \Gamma(\ln g,L(\ln g)))
\\
& = g L\Gamma(\ln g) - 2 g \Gamma(\ln g,L(\ln g)).
\end{align*}
Observing that
\[
L(\ln g) = - \frac{\Gamma(g)}{g^2} - \frac{g_t}{g},
\]
we conclude   that
\[
L\phi_1+\frac{\partial \phi_1}{\partial t} =2 (P_{T-t} f) \Gamma_2 (\ln P_{T-t}f). 
\]

In the same vein, we obtain
\[
L\phi_2 + \frac{\p \phi_2}{\p t} = g L\Gamma^Z(\ln g) +  2\Gamma(g,\Gamma^Z(\ln g)) + 2 g \Gamma^Z(\ln g,\frac{g_t}{g}).
\]
On the other hand, this time using \eqref{gamma2Z}, we find 
\begin{align*}
2 g \Gamma_2^Z(\ln g) & = g (L \Gamma^Z(\ln g) - 2 \Gamma^Z(\ln g,L(\ln g)))
\\
& = g L\Gamma^Z(\ln g) + 2 g \Gamma^Z(\ln g,\frac{\Gamma(g)}{g^2}) + 2 g \Gamma^Z(\ln g,\frac{g_t}{g}).
\end{align*}
From this latter equation it is now clear that, if the Hypothesis \eqref{A:main_assumption} is valid, then
\[
L\phi_2 + \frac{\p \phi_2}{\p t} =  2 g \Gamma_2^Z(\ln g).
\]
This concludes the proof.

\end{proof}

We now turn to our most important variational inequality.  Given a function $f\in C^\infty_b(\M)$ and $\ve>0$, we let $f_\varepsilon=f+\varepsilon$. 

Suppose that $T>0$, and $x\in \M$ be given. For a function $f\in  C^\infty_b(\M)$ with $f \ge 0$ we define for $t\in [0,T]$,
\[
\Phi_1 (t)=P_t \left( (P_{T-t} f_\varepsilon) \Gamma (\ln P_{T-t}f_\varepsilon) \right),
\]
\[
\Phi_2 (t)=P_t \left( (P_{T-t} f_\varepsilon) \Gamma^Z (\ln P_{T-t}f_\varepsilon) \right).
\]

\begin{theorem}\label{T:source}
Suppose that the Hypothesis  \ref{A:exhaustion},   \ref{A:main_assumption},  \ref{A:regularity} be satisfied and that the curvature-dimension inequality \eqref{cdi} holds for $\rho_1\in \R$. Let $a, b \in C^1([0,T],[0,\infty))$, $\gamma \in C((0,T),\R)$ be such that
$a'+2\rho_1 a -2\kappa \frac{a^2}{b}-\frac{4a\gamma}{d}$, $b'+2\rho_2 a$, $a\gamma, a\gamma^2$
be continuous functions on $[0,T]$. Given $f \in C_0^\infty(\M)$, with $f\ge 0$, we have
\begin{align*}
 & a(T) P_T \left(  f_\varepsilon \Gamma (\ln f_\varepsilon) \right) +b(T) P_T \left(  f_\varepsilon \Gamma^Z (\ln f_\varepsilon) \right)-a(0)(P_{T} f_\varepsilon) \Gamma (\ln P_{T}f_\varepsilon)-b(0)\Gamma^Z (\ln P_{T}f_\varepsilon) 
 \\
 & \ge  \int_0^T \left(a'+2\rho_1 a -2\kappa \frac{a^2}{b}-\frac{4a\gamma}{d} \right)\Phi_1  ds +\int_0^T(b'+2\rho_2 a) \Phi_2  ds
 \\
 & +\left(\frac{4}{d}\int_0^T a\gamma ds\right)LP_{T} f_\varepsilon -\left(\frac{2 }{d}\int_0^T a\gamma^2ds\right)P_T f_\varepsilon.
\end{align*}
\end{theorem}

\begin{proof}
Let $f \in C^\infty(\mathbb{M})$, $ f \ge 0$.
Consider the function
\[
\phi (x,t)=a(t)(P_{T-t} f) (x)\Gamma (\ln P_{T-t}f)(x)+b(t)(P_{T-t} f) (x) \Gamma^Z (\ln P_{T-t}f)(x).
\]
Applying Lemma \ref{L:derivatives} and the curvature-dimension inequality \eqref{cdi}, we obtain
\begin{align*}
  L\phi+\frac{\partial \phi}{\partial t} &
=a' (P_{T-t} f) \Gamma (\ln P_{T-t}f)+b' (P_{T-t} f) \Gamma^Z (\ln P_{T-t}f) \\
&+2a (P_{T-t} f) \Gamma_2 (\ln P_{T-t}f)+2b (P_{T-t} f) \Gamma_2^Z (\ln P_{T-t}f) \\
&\ge  \left(a'+2\rho_1 a -2\kappa \frac{a^2}{b}\right)(P_{T-t} f) \Gamma (\ln P_{T-t}f)
\\
&  +(b'+2\rho_2 a) (P_{T-t} f)  \Gamma^Z (\ln P_{T-t}f) \\
&+\frac{2a}{d}  (P_{T-t} f) (L(\ln P_{T-t} f))^2. 
\end{align*}
But, 
\[
(L(\ln P_{T-t} f))^2 \ge 2\gamma L(\ln P_{T-t}f) -\gamma^2,
\]
and
\[
 L(\ln P_{T-t}f)=\frac{LP_{T-t}f}{P_{T-t}f} -\Gamma(\ln P_{T-t} f ).
 \]
Therefore,
 \begin{align*}
L\phi+\frac{\partial \phi}{\partial t}  & \ge \left(a'+2\rho_1 a -2\kappa \frac{a^2}{b}-\frac{4a\gamma}{d} \right) (P_{T-t} f) \Gamma (\ln P_{T-t}f)
\\
& +(b'+2\rho_2 a)  (P_{T-t} f) \Gamma^Z (\ln P_{T-t}f) \\
&+\frac{4a\gamma}{d} LP_{T-t} f - \frac{2a\gamma^2}{d} P_{T-t} f.
\end{align*}
If now $f \in C_0^\infty(\M)$, $f\ge 0$, we obtain the same differential inequality if we use $f_\varepsilon$ instead of $f$ throughout. At that point we apply Proposition \ref{P:missing_key2} to reach the desired conclusion.

\end{proof}

The following corollary is of particular importance.

\begin{corollary}\label{C:variational}
Under the same assumptions of Theorem \ref{T:source}, let $b:[0,T]\to [0,\infty)$ be a non-increasing $C^2$ function such that, with 
\begin{equation}\label{babygamma}
\gamma\overset{def}{=}\frac{d}{4} \left(\frac{b''}{b'} +\frac{\kappa}{\rho_2} \frac{b'}{b} +2\rho_1\right),
\end{equation}
the functions $b'\gamma, b'\gamma^2$ be continuous on $[0,T]$. Then, we have  for $f \in C_0^\infty(\M)$,
\begin{align}\label{var}
  - &\frac{b'(T)}{2\rho_2}P_T \left(  f_\varepsilon \Gamma (\ln f_\varepsilon) \right) +b(T)P_T \left(  f_\varepsilon \Gamma^Z (\ln f_\varepsilon) \right)
 \\
  + & \frac{b'(0)}{2\rho_2} (P_{T} f_\varepsilon) \Gamma (\ln P_{T}f_\varepsilon) -b(0) \Gamma^Z (\ln P_{T}f_\varepsilon)
 \notag\\
\ge  &   -\left( \frac{2}{d\rho_2}\int_0^T b' \gamma ds\right) LP_Tf_\varepsilon + \left(\frac{1}{d\rho_2}\int_0^T  b'  \gamma^2 ds\right) P_T f_\varepsilon.
\notag
\end{align}
\end{corollary}

\begin{proof}
We choose $a:[0,T]\to [0,\infty)$  of class $C^1$ so that 
\[
b'+2\rho_2 a=0.
\]
With this choice, and with $\gamma$ defined by \eqref{babygamma}, we obtain
\[
a'+2\rho_1 a -2\kappa \frac{a^2}{b} -\frac{4a\gamma}{d}  =0.
\]
Applying Theorem \ref{T:source} with these $a, b$ and $\gamma$, we immediately reach the desired conclusion. 

\end{proof}

\section{Li-Yau type estimates}

In this section, we extend the celebrated Li-Yau inequality in
\cite{LY} to the heat semigroup associated with the subelliptic
operator $L$. Let us mention that, in this setting, related
inequalities were obtained by Cao-Yau \cite{Cao-Yau}. However, these
authors work locally and the geometry of the manifold does not enter in their study.
Instead, our analysis in based on the entropic inequalities established in Section \ref{S:EI} and, consequently, it hinges crucially on our curvature-dimension inequality \eqref{cdi}. As we have shown in the discussion of the examples in Section \ref{S:appendix}, such inequality is deeply connected to the sub-Riemannian geometry of the manifold. We have mentioned in the introduction that, even when specialized to the Riemannian case, the ideas in this section provide a new, more elementary approach of the Li-Yau inequalities. For this aspect we refer the reader to the paper \cite{BGjga}.

\begin{theorem}[Gradient estimate]\label{T:ge}
Assume the  Hypothesis  \ref{A:exhaustion},  \ref{A:main_assumption}, \ref{A:regularity}  and that the curvature-dimension inequality \eqref{cdi} be satisfied for $\rho_1\in \R$. Let  $f \in C_0^\infty(\M)$, $f  \ge 0$, $f \not\equiv 0$, then the following inequality holds for $t>0$:
\[
\Gamma (\ln P_t f) +\frac{2 \rho_2}{3}  t \Gamma^Z (\ln P_t f) \le
\left(1+\frac{3\kappa}{2\rho_2}-\frac{2\rho_1}{3} t\right)
\frac{LP_t f}{P_t f} +\frac{d\rho_1^2}{6} t-\frac{\rho_1 d}{2}\left(
1+\frac{3\kappa}{2\rho_2}\right) +\frac{d\left(
1+\frac{3\kappa}{2\rho_2}\right)^2}{2t}.
\]
\end{theorem}

\begin{proof}
We apply Corollary \ref{C:variational}, in which we choose $b(t)=(T-t)^3$. With such choice, \eqref{babygamma} gives:
\[
\gamma (t)=
\frac{d}{2}\left(\rho_1
 - \frac{1}{T-t}\left(1 + \frac{3\kappa}{2\rho_2}\right)\right),
\]
and thus $b'\gamma, b'\gamma^2\in C([0,t]),\R)$. Simple calculations give
\[
\int_0^T b'(t) \gamma(t) dt =-\frac{\rho_1 d}{2} T^3 +\frac{3d}{4} \left( 1+ \frac{3\kappa}{2\rho_2}\right) T^2,
\]
and
\[
\int_0^T b'(t) \gamma(t)^2 dt =-\frac{3 d^2}{16} \left(
\frac{4\rho^2_1 }{3} T^3+ 4 \left( 1+ \frac{3\kappa}{2\rho_2}\right)
^2 T -4\rho_1\left( 1+ \frac{3\kappa}{2\rho_2}\right) T^2 \right).
\]
Using the latter two equations in \eqref{var} and letting $\varepsilon \to 0$, by the arbitrariness of $T>0$ we obtain the desired conclusion.

\end{proof}

\begin{remark}\label{R:rho1zero}
We notice that when $\rho_1\ge \rho_1'$, then one trivially has that:
\[
\emph{CD}(\rho_1,\rho_2,\kappa,d)\  \Longrightarrow\  \emph{CD}(\rho_1',\rho_2,\kappa,d).
\]
As a consequence of this observation, when \eqref{cdi} holds with $\rho_1>0$, then also \text{CD}$(0,\rho_2,\kappa,d)$ is true. Therefore, when $\rho_1>0$, Theorem \ref{T:ge} gives in particular for $f\in C_0^\infty(\M)$, $f\ge 0$,
\begin{align}\label{liyaupositifzero}
\Gamma (\ln P_t f) +\frac{2 \rho_2}{3}  t \Gamma^Z (\ln P_t f) \le
\left(1+\frac{3\kappa}{2\rho_2}\right) \frac{LP_t f}{P_t f} +\frac{d\left(
1+\frac{3\kappa}{2\rho_2}\right)^2}{2t}.
\end{align}
However, this inequality is not optimal when $\rho_1>0$. It leads to a optimal Harnack inequality only when
$\rho_1 = 0$. Sharper
bounds in the case $\rho_1>0$ will be obtained in
\eqref{gamma_bound} of Proposition \ref{P:bound_kernel} below by a different choice of the function $b(t)$ in Corollary \ref{C:variational}.
\end{remark}

\begin{remark}\label{R:ultrac}
Throughout the remainder of the paper the symbol $D$ will only be used with the following meaning:
 \begin{equation}\label{D} D =  d  \left( 1+
\frac{3\kappa}{2\rho_2}\right). \end{equation}
With this notation, observing that the left-hand side of \eqref{liyaupositifzero} is always nonnegative, and that $LP_t f = \p_t P_t f$, when $\rho_1 \ge 0$ we obtain
\begin{align}\label{ultracontractivity}
\p_t(\ln (t^{D/2} P_t f(x))) \ge 0.
\end{align}
By integrating \eqref{ultracontractivity} from $t<1$ to $1$ leads to the following on-diagonal bound for the heat kernel,
\begin{align}\label{ultracontractivity2}
p (x,x,t) \le \frac{1}{t^{D/2}} p(x,x,1).
\end{align}

The constant $\frac{D}{2}$ in \eqref{ultracontractivity2} is not optimal, in
general, as the example of the heat semigroup on a Carnot group shows. In such case, in fact, one can argue as
in \cite{Fo} to show that the heat kernel $p(x,y,t)$ is homogeneous
of degree $-\frac{Q}{2}$ with respect to the non-isotropic group dilations,
where $Q$ indicates the corresponding
homogeneous dimension of the group. From such homogeneity of $p(x,y,t)$,
one obtains the estimate
\begin{align*}
p (x,x,t) \le \frac{1}{t^{Q/2}} p(x,x,1),
\end{align*}
which, unlike \eqref{ultracontractivity2}, is best possible.
In the sub-Riemannian setting it does not seem easy to obtain sharp geometric constants by using only
the curvature-dimension inequality \eqref{cdi}. This aspect is quite different
from the Riemannian case, for which the \emph{CD}$(\rho_1,n)$ inequality
\eqref{CDi} does provide sharp geometric constants (see
\cite{bakry-tata}, \cite{ledoux-zurich}). However, in such case our bound \eqref{ultracontractivity2} is sharp as well, since if $d = n=$ dim$(\M)$, and $\kappa = 0$, then \eqref{D} gives $D = n$.
\end{remark}

\section{A parabolic Harnack inequality}\label{S:harnack}

In this section we generalize the celebrated Harnack inequality in
\cite{LY} to solutions of the heat equation $Lu - u_t = 0$
on $\bM$ which are in the form $u(x,t) = P_t f(x)$,
for some $f\in C_b^\infty(\bM)$, $f\ge 0$. Theorem
\ref{T:harnack} below should be seen as a generalization of (i) of
Theorem 2.2 in \cite{LY}, in the case of a zero potential $q$. One
should also see the paper \cite{Cao-Yau}, where the authors deal
with subelliptic operators on a compact manifold. As we have
mentioned, these authors do not obtain bounds which depend 
on the sub-Riemannian geometry of the underlying manifold. 

\begin{theorem}\label{T:harnack}
Assume the Hypothesis \ref{A:exhaustion}, \ref{A:main_assumption}, \ref{A:regularity}  and that the curvature-dimension inequality \eqref{cdi} be satisfied for $\rho_1\ge 0$. Given $(x,s), (y,t)\in \bM\times (0,\infty)$, with
$s<t$, one has for any $f\in
C_b^\infty(\bM)$, $f\ge 0$,
\begin{equation}\label{beauty}
P_s f(x) \le P_t f(y) \left(\frac{t}{s}\right)^{\frac{D}{2}} \exp\left(
\frac{D}{d} \frac{d(x,y)^2}{4(t-s)} \right).
\end{equation}
\end{theorem}

\begin{proof}
Let $f\in C_0^\infty(\M)$ be as in the statement of the theorem, and for every $(x,t)\in \bM\times
(0,\infty)$ consider $u(x,t) = P_t f(x)$ . Since $Lu = \frac{\p u}{\p t}$, in terms of $u$
the inequality \eqref{liyaupositifzero} can be reformulated as
\[ \Gamma (\ln u) +\frac{2 \rho_2}{3} t \Gamma^Z (\ln u) \le
(1+\frac{3\kappa}{2\rho_2}) \frac{\p \log u}{\p t}  +
\frac{d\left( 1+\frac{3\kappa}{2\rho_2}\right)^2}{2t}.
\]
Recalling \eqref{D}, this implies in particular, 
\begin{equation}\label{liyaupositif2}
- \frac{\p \ln u}{\p t} \le - \frac{d}{D} \Gamma
(\ln u)  +\frac{D}{2t}.
\end{equation}

We now fix two points $(x,s), (y,t)\in \bM\times (0,\infty)$, with
$s<t$. Let $\gamma(\tau)$, $0\le \tau \le T$ be a subunit path such
that $\gamma(0) = y$, $\gamma(T) = x$ (for the definition of subunit path see \cite{FP1}). Consider the path in $\bM\times
(0,\infty)$ defined by
\[
\alpha(\tau) = \left(\gamma(\tau),t + \frac{s-t}{T}\tau\right),\ \ \
\ 0\le \tau\le T,
\]
so that $\alpha(0) = (y,t)$, $\alpha(T) = (x,s)$. We have
\begin{align*}
\ln \frac{u(x,s)}{u(y,t)}& = \int_0^T \frac{d}{d\tau} \ln
u(\alpha(\tau)) d\tau \\
& \le \int_0^T \left[\Gamma(\ln u(\alpha(\tau)))^{\frac{1}{2}}
- \frac{t-s}{T} \frac{\p \ln u}{\p t}(\alpha(\tau))\right] d \tau.
\end{align*}
Applying \eqref{liyaupositif2} for any $\epsilon >0$ we find
\begin{align*}
\log \frac{u(x,s)}{u(y,t)}& \le T^{\frac{1}{2}} \left(\int_0^T
\Gamma(\ln u)(\alpha(\tau)) d\tau\right)^{\frac{1}{2}} -
\frac{t-s}{T} \int_0^T \frac{\p \ln u}{\p t}(\alpha(\tau)) d \tau
\\
& \le \frac{1}{2\epsilon} T + \frac{\epsilon}{2} \int_0^T \Gamma(\ln
u)(\alpha(\tau)) d\tau - \frac dD \frac{t-s}{T}
\int_0^T \Gamma(\ln u)(\alpha(\tau)) d\tau
\\
&   - \frac{D(s-t)}{2T}
\int_0^T \frac{d\tau}{t + \frac{s-t}{T} \tau}.
\end{align*}
If we now choose $\epsilon >0$ such that
\[
\frac{\epsilon}{2} = \frac dD\frac{t-s}{T},
\]
we obtain from the latter inequality
\[
\log \frac{u(x,s)}{u(y,t)} \le
\frac Dd\frac{\ell_s(\gamma)^2}{4(t-s)}  +
\frac{D}{2}
\ln\left(\frac{t}{s}\right),
\]
where we have denoted by $\ell_s(\gamma)$ the subunitary length of
$\gamma$. If we now minimize over all subunitary paths joining $y$
to $x$, and we exponentiate, we obtain
\[
u(x,s) \le u(y,t) \left(\frac{t}{s}\right)^{\frac{D}{2}} \exp\left(\frac Dd \frac{d(x,y)^2}{4(t-s)} \right).
\]
This proves
\eqref{beauty} when $f \in C_0^\infty(\M)$. We can then extend the result to  $f \in C_b^\infty(\bM)$ by considering the approximations $h_n P_\tau f \in C_0^\infty(\M)$ , where $h_n \in C_0^\infty(\M)$, $h_n \ge 0$, $h_n \to_{n \to \infty} 1$ and let $n \to \infty$ and $\tau \to 0$.

\end{proof}

The following result represents an important consequence of Theorem
\ref{T:harnack}.

\begin{corollary}\label{C:harnackheat}
Suppose that the Hypothesis \ref{A:exhaustion}, \ref{A:main_assumption},  \ref{A:regularity} be valid, and that the curvature-dimension inequality \eqref{cdi} be satisfied for $\rho_1\ge 0$. Let $p(x,y,t)$ be the heat kernel on $\bM$. For every $x,y, z\in
\bM$ and every $0<s<t<\infty$ one has
\[
p(x,y,s) \le p(x,z,t) \left(\frac{t}{s}\right)^{\frac{D}{2}}
\exp\left(\frac{D}{d} \frac{d(y,z)^2}{4(t-s)} \right).
\]
\end{corollary}

\begin{proof}
Let $\tau >0$ and $x\in \bM$ be fixed. By the hypoellipticity of $L - \p_t$, we know that $p(x,\cdot,\cdot +
\tau)\in C^\infty(\bM \times (-\tau,\infty))$, see \cite{FSC}. From \eqref{sgp} we
have
\[
p (x,y,s+\tau)=P_s (p(x,\cdot,\tau))(y)
\]
and
\[
p (x,z,t+\tau)=P_t (p(x,\cdot,\tau))(z)
\]
Since we cannot apply Theorem \ref{T:harnack} directly to $u(y,t) =
P_t(p(x,\cdot,\tau))(y)$, we consider 
$u_n(y,t) = P_t(h_n p(x,\cdot,\tau))(y)$, where $h_n\in
C^\infty_0(\bM)$, $0\le h_n\le 1$, and $h_n\nearrow  1$. From
\eqref{beauty} we find
\[
P_s (h_np(x,\cdot,\tau))(y) \le P_t (h_np(x,\cdot,\tau))(z)
\left(\frac{t}{s}\right)^{\frac{D}{2}} \exp\left(\frac{D}{d}
\frac{d(y,z)^2}{4(t-s)} \right)
\]
Letting $n \to \infty$, by Beppo Levi's monotone convergence theorem
we obtain
\[
p (x,y,s+\tau) \le p (x,z,t+\tau)
\left(\frac{t}{s}\right)^{\frac{D}{2}} \exp\left(\frac{D}{d}
\frac{d(y,z)^2}{4(t-s)} \right).
\]
The desired conclusion follows by letting $\tau \to 0$.

\end{proof}

\section{Off-diagonal Gaussian upper bounds for $p(x,y,t)$}\label{S:gaussianub}

Suppose that the assumption of Theorem \ref{T:harnack} are in force. Fix $x\in \bM$ and
$t>0$. Applying Corollary \ref{C:harnackheat} to $(y,t)\to p(x,y,t)$
for every $y\in B(x,\sqrt t)$ we find
\[
p(x,x,t) \le  2^{\frac{D}{2}} e^{\frac{D}{4d}}\ p(x,y,2t) =
C(\rho_2,\kappa,d) p(x,y,2t).
\]
Integration over $B(x,\sqrt t)$ gives
\[
p(x,x,t)\mu(B(x,\sqrt t)) \le C(\rho_2,\kappa,d) \int_{B(x,\sqrt
t)}p(x,y,2t)d\mu(y) \le C(\rho_2,\kappa,d),
\]
where we have used $P_t1\le 1$. This gives the on-diagonal upper
bound 
\begin{equation}\label{odub} 
p(x,x,t) \le \frac{C(\rho_2,\kappa,d)}{\mu(B(x,\sqrt t))}.
\end{equation}

The aim of this section is to establish the following off-diagonal
upper bound for the heat kernel. Before doing so, let us observe that from the general theory of Markov semigroups, if the volume doubling property is assumed, then  the on-diagonal bound \eqref{odub}  implies an off-diagonal bound (see for instance \cite{CS}). However, in our framework, the volume doubling property is only proved in the sequel paper \cite{BBG} which relies on the results in the present paper. Therefore, and we think this is interesting in itself, to prove the off-diagonal upper bound, we completely bypass the use of uniform volume estimates and instead rely in an essential way on the scale invariant parabolic Harnack inequality.

\begin{theorem}\label{T:ub}
Assume the Hypothesis \ref{A:exhaustion}, \ref{A:main_assumption},  \ref{A:regularity},  and that the curvature-dimension inequality \eqref{cdi} be satisfied for $\rho_1\ge 0$. For any $0<\epsilon <1$
there exists a constant $C(\rho_2,\kappa,d,\epsilon)>0$, which tends
to $\infty$ as $\epsilon \to 0^+$, such that for every $x,y\in \bM$
and $t>0$ one has
\[
p(x,y,t)\le \frac{C(d,\kappa,\rho_2,\epsilon)}{\mu(B(x,\sqrt
t))^{\frac{1}{2}}\mu(B(y,\sqrt t))^{\frac{1}{2}}} \exp
\left(-\frac{d(x,y)^2}{(4+\epsilon)t}\right).
\]
\end{theorem}

\begin{proof}
We suitably adapt here an idea in \cite{Cao-Yau} for the case of a
compact manifold without boundary. Since, however, we allow the
manifold $\bM$ to be  non-compact, we need to take care of this
aspect. Corollary \ref{C:expdecay} will prove crucial in this
connection. Given $T>0$, and $\alpha>0$ we fix $0<\tau \le
(1+\alpha)T$. For a function $\psi\in C^\infty_0(\bM)$, with $\psi
\ge 0$, in $\bM \times (0,\tau)$ we consider the function
\[
f(y,t) = \int_{\bM} p(y,z,t) p(x,z,T) \psi(z) d\mu(z),\ \ \ x\in
\bM.
\]
Since $f = P_t(p(x,\cdot,T)\psi)$, it satisfies the Cauchy problem
\[
\begin{cases}
Lf - f_t = 0 \ \ \ \ \text{in}\ \bM \times (0,\tau),
\\
f(z,0) = p(x,z,T)\psi(z),\ \  \ z\in \bM.
\end{cases}
\]
Notice that by the hypoellipticity of $L-\p_t$ we know $y\to
p(x,y,T)$ is in $C^\infty(\bM)$, and therefore $p(x,\cdot,T)\psi\in
L^\infty(\bM)$. Moreover, \eqref{smp} gives \[
||P_t(p(x,\cdot,T)\psi)||^2_{L^2(\bM)} \le
||p(x,\cdot,T)\psi||^2_{L^2(\bM)} = \int_{\bM} p(x,z,T)^2 \psi(z)
d\mu(z) <\infty, \]
and therefore
\begin{equation}\label{L2f}
\int_0^\tau \int_{\bM} f(y,t)^2 d\mu(z)dt \le \tau \int_{\bM}
p(x,z,T)^2 \psi(z) d\mu(z)dt <\infty.
\end{equation}
Invoking \eqref{pointwiseCaccioppoli} in Corollary \ref{C:expdecay}
we have
\[
\Gamma (f)(z,t) \le  e^{-\alpha t} \left( P_t
\Gamma(p(x,\cdot,T)\psi)(z) + P_t
\Gamma^Z(p(x,\cdot,T)\psi)(z)\right).
\]
This allows to conclude
\begin{equation}\label{L2gradf}
\int_0^\tau \int_{\bM} \Gamma(f)(z,t)^2 d\mu(z)dt <\infty.
\end{equation}
We now consider a function $g\in
C^1([0,(1+\alpha)T],\text{Lip}_d(\bM))\cap L^\infty(\bM\times
(0,(1+\alpha)T))$ such that
\begin{equation}\label{ineg}
- \frac{\p g}{\p t} \ge \frac{1}{2}\Gamma(g),\ \ \text{on}\ \bM \times
(0,(1+\alpha)T).
\end{equation}

Since
\[ (L-\frac{\p}{\p t})f^2 = 2 f(L-\frac{\p}{\p t})f + 2 \Gamma(f) =
2 \Gamma(f), \] multiplying this identity by $h_n^2(y) e^{g(y,t)}$,
where $h_n$ is a sequence as in Hypothesis \ref{A:exhaustion}, and
integrating by parts, we obtain
\begin{align*}
0 & = 2 \int_0^\tau \int_{\bM} h_n^2 e^g \Gamma(f) d\mu(y) dt -
\int_0^\tau \int_{\bM}h_n^2 e^g (L-\frac{\p}{\p t})f^2 d\mu(y) dt
\\
& = 2 \int_0^\tau \int_{\bM} h_n^2 e^g \Gamma(f) d\mu(y) dt + 4
\int_0^\tau \int_{\bM} h_n e^g f \Gamma(h_n,f) d\mu(y) dt \\
& + 2 \int_0^\tau\int_{\bM}h_n^2 e^g f \Gamma(f,g)d\mu(y) dt -
\int_0^\tau \int_{\bM} h_n e^g f^2 \frac{\p g}{\p t} d\mu(y) dt
\\
& -   \int_{\bM} h_n e^g f^2 d\mu(y)\bigg|_{t=0} + \int_{\bM} h_n
e^g f^2 d\mu(y)\bigg|_{t=\tau}
\\
& \ge 2 \int_0^\tau \int_{\bM} h_n^2 e^g \left(\Gamma(f) +
\frac{f^2}{4} \Gamma(g) + f \Gamma(f,g)\right) d\mu(y) dt + 4 \int_0^\tau \int_{\bM} h_n
e^g f \Gamma(h_n,f) d\mu(y) dt
\\
& + \int_{\bM} h_n e^g f^2 d\mu(y)\bigg|_{t=\tau} -   \int_{\bM} h_n
e^g f^2 d\mu(y)\bigg|_{t=0},
\end{align*}
where in the last inequality we have made use of the assumption
\eqref{ineg} on $g$. From this we conclude
\[
\int_{\bM} h_n e^g f^2 d\mu(y)\bigg|_{t=\tau} \le \int_{\bM} h_n e^g
f^2 d\mu(y)\bigg|_{t=0} - 4 \int_0^\tau \int_{\bM} h_n e^g f
\Gamma(h_n,f) d\mu(y) dt.
\]
We now claim that
\[
\underset{n\to \infty}{\lim} \int_0^\tau \int_{\bM} h_n e^g f
\Gamma(h_n,f) d\mu(y) dt = 0.
\]
To see this we apply Cauchy-Schwarz inequality which gives
\begin{align*}
& \left|\int_0^\tau \int_{\bM} h_n e^g f \Gamma(h_n,f) d\mu(y)
dt\right|\le \left(\int_0^\tau \int_{\bM} h_n^2 e^g f^2 \Gamma(h_n)
d\mu(y) dt\right)^{\frac{1}{2}} \left(\int_0^\tau \int_{\bM} e^g
\Gamma(f) d\mu(y) dt\right)^{\frac{1}{2}}
\\
& \le \left(\int_0^\tau \int_{\bM} e^g f^2 \Gamma(h_n) d\mu(y)
dt\right)^{\frac{1}{2}} \left(\int_0^\tau \int_{\bM} e^g \Gamma(f)
d\mu(y) dt\right)^{\frac{1}{2}} \to 0,
\end{align*}
as $n\to \infty$, thanks to \eqref{L2f}, \eqref{L2gradf}. With the
claim in hands we now let $n\to \infty$ in the above inequality
obtaining
\begin{equation}\label{CYpsi}
\int_{\bM}  e^{g(y,\tau)} f^2(y,\tau) d\mu(y) \le \int_{\bM}
e^{g(y,0)} f^2(y,0) d\mu(y).
\end{equation}
At this point we fix $x\in \bM$ and for $0<t\le\tau$ consider the
indicator function $\mathbf 1_{B(x,\sqrt t)}$ of the ball $B(x,\sqrt
t)$. Let $\psi_k\in C^\infty_0(\bM)$, $\psi_k \ge 0$, be a sequence
such that $\psi_k \to \mathbf 1_{B(x,\sqrt t)}$ in $L^2(\bM)$, with
supp$\ \psi_k\subset B(x,100\sqrt t)$. Slightly abusing the notation
we now set \[ f(y,s) = P_s(p(x,\cdot,T)\mathbf{1}_{B(x,\sqrt t)})(y)
= \int_{B(x,\sqrt t)} p(y,z,s) p(x,z,T) d\mu(z). \] Thanks to the
symmetry of $p(x,y,s) = p(y,x,s)$, we have
\begin{equation}\label{psquare}
 f(x,T) = \int_{B(x,\sqrt t)} p(x,z,T)^2 d\mu(z).
\end{equation}

Applying \eqref{CYpsi} to $f_k(y,s) = P_s(p(x,\cdot,T)\psi_k)(y)$,
we find
\begin{equation}\label{CYpsik}
\int_{\bM}  e^{g(y,\tau)} f^2_k(y,\tau) d\mu(y) \le \int_{\bM}
e^{g(y,0)} f^2_k(y,0) d\mu(y).
\end{equation}
At this point we observe that as $k\to \infty$
\begin{align*}
& \left|\int_{\bM}  e^{g(y,\tau)} f^2_k(y,\tau) d\mu(y) - \int_{\bM}
e^{g(y,\tau)} f^2(y,\tau) d\mu(y)\right|
\\
& \le 2 ||e^{g(\cdot,\tau)}||_{L^\infty(\bM)}
||p(x,\cdot,T)||_{L^2(\bM)} ||p(x,\cdot,\tau)||_{L^\infty(B(x,110
\sqrt t))} ||\psi_k - \mathbf 1_{B(x,\sqrt t)}||_{L^2(\bM)} \to 0.
\end{align*}
By similar considerations we find
\begin{align*}
& \left|\int_{\bM}  e^{g(y,0)} f^2_k(y,0) d\mu(y) - \int_{\bM}
e^{g(y,0)} f^2(y,0) d\mu(y)\right|
\\
& \le 2 ||e^{g(\cdot,0)}||_{L^\infty(\bM)}
||p(x,\cdot,T)||_{L^\infty(B(x,110 \sqrt t))} ||\psi_k - \mathbf
1_{B(x,\sqrt t)}||_{L^2(\bM)} \to 0.
\end{align*}
Letting $k\to \infty$ in \eqref{CYpsik} we thus conclude that the
same inequality holds with $f_k$ replaced by $f(y,s) =
P_s(p(x,\cdot,T)1_{B(x,\sqrt t)})(y)$. This implies in particular
the basic estimate
\begin{align}\label{be}
& \underset{z\in B(x,\sqrt t)}{\inf}\ e^{g(z,\tau)} \int_{B(x,\sqrt
t)} f^2(z,\tau) d\mu(z)
\\
& \le \int_{B(x,\sqrt t)} e^{g(z,\tau)} f^2(z,\tau) d\mu(z) \le
\int_{\bM} e^{g(z,\tau)} f^2(z,\tau) d\mu(z) \notag\\
& \le \int_{\bM} e^{g(z,0)} f^2(z,0) d\mu(z) = \int_{B(y,\sqrt t)}
e^{g(z,0)} p(x,z,T)^2 d\mu(z) \notag\\
& \le \underset{z\in B(y,\sqrt t)}{\sup}\ e^{g(z,0)} \int_{B(y,\sqrt
t)} p(x,z,T)^2 d\mu(z). \notag
\end{align}

At this point we choose in \eqref{be}
\[ g(y,t) = g_x(y,t) = -
\frac{d(x,y)^2}{2((1+2\alpha) T - t)}.
\]
Using the fact that $\Gamma(d)\le 1$, one can easily check that
\eqref{ineg} is satisfied for this $g$. Taking into account that
\[
\underset{z\in B(x,\sqrt t)}{\inf}\ e^{g_x(z,\tau)} = \underset{z\in
B(x,\sqrt t)}{\inf}\ e^{-\frac{d(x,z)^2}{2((1+2\alpha)T- \tau)}} \ge
e^{\frac{-t}{2((1+2\alpha)T- \tau)}},
\]
if we now choose $\tau = (1+\alpha)T$, then from the previous
inequality and from \eqref{psquare} we conclude that
\begin{equation}\label{lemmaub}
\int_{B(x,\sqrt t)} f^2(z,(1+\alpha)T) d\mu(z) \le
\left(\underset{z\in B(y,\sqrt t)}{\sup}\
e^{-\frac{d(x,z)^2}{2(1+2\alpha)T} + \frac{t}{2\alpha T}}\right)
\int_{B(y,\sqrt t)} p(x,z,T)^2 d\mu(z).
\end{equation}
We now apply Theorem \ref{T:harnack} which gives for every $z\in
B(x,\sqrt t)$
\[
f(x,T)^2 \le f(z,(1+\alpha)T)^2
(1+\alpha)^{d(1+\frac{3\kappa}{2\rho_2})}
e^{\frac{t(1+\frac{3\kappa}{2\rho_2})}{2\alpha T}}.
\]
Integrating this inequality on $B(x,\sqrt t)$ we find
\[
\left(\int_{B(y,\sqrt t)} p(x,z,T)^2 d\mu(z)\right)^2 = f(x,T)^2 \le
\frac{(1+\alpha)^{d(1+\frac{3\kappa}{2\rho_2})}
e^{\frac{t(1+\frac{3\kappa}{2\rho_2})}{2\alpha T}}}{\mu(B(x,\sqrt
t)) } \int_{B(x,\sqrt t)} f^2(z,(1+\alpha)T) d\mu(z).
\]
If we now use \eqref{lemmaub} in the last inequality we obtain
\begin{align*}
& \int_{B(y,\sqrt t)} p(x,z,T)^2 d\mu(z) \le
\frac{(1+\alpha)^{d(1+\frac{3\kappa}{2\rho_2})}
e^{\frac{t(1+\frac{3\kappa}{2\rho_2})}{2\alpha T}}}{\mu(B(x,\sqrt
t))} \left(\underset{z\in B(y,\sqrt t)}{\sup}\
e^{-\frac{d(x,z)^2}{2(1+2\alpha)T} + \frac{t}{2\alpha T}}\right).
\end{align*}
Choosing $T = (1+\alpha)t$ in this inequality we find
\begin{align}\label{ub2}
& \int_{B(y,\sqrt t)} p(x,z,(1+\alpha)t)^2 d\mu(z) \le
\frac{(1+\alpha)^{d(1+\frac{3\kappa}{2\rho_2})}
e^{\frac{(1+\frac{3\kappa}{2\rho_2})}{2\alpha(1+\alpha)}+
\frac{1}{2\alpha (1+\alpha)}}}{\mu(B(x,\sqrt t))}
\left(\underset{z\in B(y,\sqrt t)}{\sup}\
e^{-\frac{d(x,z)^2}{2(1+2\alpha)(1+\alpha)t} + \frac{1}{2\alpha
(1+\alpha)}}\right).
\end{align}
We now apply Corollary \ref{C:harnackheat} obtaining for every $z\in
B(y,\sqrt t)$
\[
p(x,y,t)^2 \le p(x,z,(1+\alpha)t)^2 (1+\alpha)^{d\left(
1+\frac{3\kappa}{2\rho_2}\right)} \exp\left(\frac{
1+\frac{3\kappa}{2\rho_2}}{2\alpha } \right).
\]
Integrating this inequality in $z\in B(y,\sqrt t)$, we have
\[
\mu(B(y,\sqrt t)) p(x,y,t)^2 \le (1+\alpha)^{d\left(
1+\frac{3\kappa}{2\rho_2}\right)} e^{\frac{
1+\frac{3\kappa}{2\rho_2}}{2\alpha }} \int_{B(y,\sqrt t)}
p(x,z,(1+\alpha)t)^2 d\mu(z).
\]
Combining this inequality with \eqref{ub2} we conclude
\[
p(x,y,t) \le \frac{(1+\alpha)^{d(1+\frac{3\kappa}{2\rho_2})}
e^{\frac{(1+\frac{3\kappa}{2\rho_2})(2+\alpha)}{4\alpha(1+\alpha)}+
\frac{3}{4\alpha (1+\alpha)}}}{\mu(B(x,\sqrt
t))^{\frac{1}{2}}\mu(B(y,\sqrt
t))^{\frac{1}{2}}}\left(\underset{z\in B(y,\sqrt t)}{\sup}\
e^{-\frac{d(x,z)^2}{2(1+2\alpha)(1+\alpha)t}}\right).
\]
If now $x\in B(y,\sqrt t)$, then
\[
d(x,z)^2 \ge (d(x,y) - \sqrt t)^2 > d(x,y)^2 - t,
\]
and therefore
\[
\underset{z\in B(y,\sqrt t)}{\sup}\
e^{-\frac{d(x,z)^2}{2(1+2\alpha)(1+\alpha)t}} \le
e^{\frac{1}{2(1+2\alpha)(1+\alpha)}}
e^{-\frac{d(x,y)^2}{2(1+2\alpha)(1+\alpha)t}}.
\]
If instead $x\not\in B(y,\sqrt t)$, then for every $\delta >0$ we
have
\[
d(x,z)^2 \ge (1-\delta) d(x,y)^2  - (1+ \delta^{-1}) t
\]
Choosing $\delta = \alpha/(\alpha+1)$ we find
\[
d(x,z)^2 \ge \frac{d(x,y)^2}{1+\alpha}  - (2 + \alpha^{-1}) t,
\]
and therefore
\[
\underset{z\in B(y,\sqrt t)}{\sup}\
e^{-\frac{d(x,z)^2}{2(1+2\alpha)(1+\alpha)t}} \le
e^{-\frac{d(x,y)^2}{2(1+2\alpha)(1+\alpha)^2 t} + \frac{2 +
\alpha^{-1}}{2(1+2\alpha)(1+\alpha)}}
\]
For any $\epsilon >0$ we now choose $\alpha>0$ such that
$2(1+2\alpha)(1+\alpha)^2 = 4+\epsilon$ to reach the desired
conclusion.

\end{proof}

\section{A generalization of Yau's Liouville theorem}

In his seminal 1975 paper \cite{Yau}, by using gradient estimates,
Yau proved his celebrated Liouville theorem that there exists no
non-constant positive harmonic function on a complete Riemannian
manifold with non-negative Ricci curvature. The aim of this section
is to extend Yau's theorem to the sub-Riemannian setting of this
paper. An interesting point to keep in mind here is that, even in
the Riemannian setting, our approach gives a new proof of Yau's
theorem which is not based on delicate tools from Riemann geometry
such as the Laplacian comparison theorem \eqref{lct} for the
geodesic distance. However, due to the nature of our proof at the
moment we are only able to deal with harmonic functions bounded from
two sides, whereas in \cite{Yau} the author is able to treat
functions satisfying a one-side bound. In the sequel paper \cite{BBG} we will remove this restriction.

We begin with a Harnack type inequality for the operator $L$.

\begin{theorem}\label{T:harnackL}
Assume the Hypothesis \ref{A:exhaustion}, \ref{A:main_assumption},  \ref{A:regularity}  and that the curvature-dimension inequality \eqref{cdi} be satisfied for $\rho_1\ge 0$. Let $0\le f\le M$ be a harmonic function on $\bM$,
then there exists a constant $C = C(\rho_2,\kappa,d)>0$ such that
for any $x_0\in \bM$ and any $r>0$ one has
\[
\underset{B(x_0,r)}{\sup} f  \le  C \underset{B(x_0,r)}{\inf} f.
\]
\end{theorem}

\begin{proof}
We know that $f\in C^\infty_b(\bM)$, and $f\ge 0$. Applying
Theorem \ref{T:harnack} to the function $u(x,t) = P_t f(x)$, we
obtain for $x,y\in B(x_0,r)$
\[
P_s f(x) \le P_t f(y) \left(\frac{t}{s}\right)^{\frac{D}{2}}
\exp\left(\frac{D r^2}{d(t-s)}\right),\ \ \ 0<s<t<\infty.
\]
At this point we observe that, thanks to the assumption $Lf = 0$,
the functions $u(x,t) = P_t f(x)$ and $v(x,t) = f(x)$ solve the same
Cauchy problem on $\bM$. By Proposition \ref{P:ucp} we must have
$P_t f(x) = f(x)$ for every $x\in \bM$ and every $t>0$. Therefore,
taking $s = r^2, t = 2r^2$, the latter inequality gives
\[
f(x) \le  \left(\sqrt 2 e^{\frac{1}{d}}\right)^D\ f(y), \ \ \ x,y\in
B(x_0,r).
\]
\end{proof}

\begin{theorem}[of Cauchy-Liouville type]\label{T:liouville}
Under the same assumptions of Theorem \ref{T:harnackL}, there exist no bounded
solutions to $Lf=0$ on $\bM$, other than the constants.
\end{theorem}

\begin{proof}
Suppose $a\le f\le b$ on $\bM$. Consider the function $g = f -
\underset{\bM}{\inf}\ f$. Clearly, $0\le g \le M = b-a$. If we apply
Theorem \ref{T:harnackL} to $g$ we find for any $x_0\in \bM$ and
$r>0$
\[
\underset{B(x_0,r)}{\sup} g  \le  C \underset{B(x_0,r)}{\inf} g.
\]
Letting $r\to \infty$ we reach the conclusion $\underset{\bM}{\sup}\
f = \underset{\bM}{\inf}\ f$, hence $f\equiv $ const.

\end{proof}

\section{A sub-Riemannian Bonnet-Myers theorem}\label{S:myer}

Let  $(\mathbb{M},g)$ be a complete, connected Riemannian manifold
of dimension $n\ge 2$. It is well-known that if for some $\rho_1>0$ the Ricci tensor of
$\mathbb{M}$ satisfies the bound 
\begin{equation}\label{ricbd}
\text{Ric} \geq (n-1) \rho_1, 
\end{equation} then $\bM$ is compact, with a finite
fundamental group, and diam$(\bM) \le \frac{\pi}{\sqrt{\rho_1}}$.
This is the celebrated Myer's theorem, which strengthens Bonnet's
theorem. Like the latter, Myer's theorem is usually proved by using
Jacobi vector fields (see e.g. Theorem 2.12 in \cite{Chavel}).

A different approach is based on the curvature-dimension inequality CD$((n-1)\rho_1,n)$,
which as we have seen, one obtains from \eqref{ricbd} (see \eqref{CDi}). When $n>2$, in the paper
\cite{ledoux-zurich} (see also \cite{Bakry-Ledoux1}) Ledoux uses ingenious non-linear methods, based on the study of the partial differential equation
\[
c(f^{p-1}-f)=-\Delta f, \ \ \ \ \    1\le p \le
\frac{2n}{n-2},
\]
 to deduce  from the curvature-dimension
inequality CD$((n-1)\rho_1,n)$ the following Sobolev
inequality
\begin{equation}\label{sobolev}
\frac{n}{(n-2)\rho_1^2} \left[ \left( \int_{\mathbb M} |f|^p d\mu
\right)^{2/p} -\int_{\mathbb M} f^2 d\mu \right] \le \int_{\mathbb
M} \Gamma (f) d\mu, \ \ f\in C^\infty_0(\mathbb M),
\end{equation}
where $\mu$ is the Riemannian measure. By a simple iteration procedure, the author shows
from \eqref{sobolev} that the diameter of $\mathbb{M}$ is finite and
bounded by $\frac{\pi}{\sqrt{\rho_1}}$. The non-linear methods in \cite{ledoux-zurich} seem difficult to extend to the framework of the present paper. 

A weaker version of the Myers theorem was proved by Bakry in \cite{bakry-stflour} by using linear methods only. We have been able to suitably adapt his approach, based on entropy-energy inequalities (a strong form of log-Sobolev inequalities).
In this section we establish the following sub-Riemannian
Bonnet-Myer's compactness theorem. 

\begin{theorem}\label{T:BM}
Assume the  Hypothesis \ref{A:exhaustion}, \ref{A:main_assumption},  \ref{A:regularity},  and that the curvature-dimension inequality \eqref{cdi} be satisfied for $\rho_1> 0$. Then, the metric space $(\mathbb{M},d)$ is compact and we have \[ \emph{diam}\ \bM \le
\frac{\pi}{\sqrt \rho_1} 2\sqrt{3} \sqrt{\left(\frac{\kappa}{\rho_2} + 1\right) D}=2\sqrt{3} \pi \sqrt{
\frac{\rho_2+\kappa}{\rho_1\rho_2} \left(
1+\frac{3\kappa}{2\rho_2}\right)d } .
\]
\end{theorem}

The proof of Theorem \ref{T:BM} will be accomplished in several steps. In the remainder of this section we will tacitly assume the hypothesis of Theorem \ref{T:BM}. 

\subsection{Global heat kernel bounds}\label{S:global}

Our first result is the following large-time exponential decay for
the heat kernel.

\begin{proposition}\label{P:bound_kernel}
Let $0< \nu < \frac{\rho_1\rho_2}{\rho_2+ \kappa}$. There exist $t_0
>0$ and $C_1>0$ such that for every $f\in C^\infty_0(\mathbb M)$, $f\ge 0$:
\[
\left| \frac{\partial }{\partial t} \ln P_t f  (x) \right| \le C_1
e^{-\nu t} , \quad \ \ \ \ x \in \mathbb{M},\ t \ge t_0.
\]
\end{proposition}

\begin{proof}
In Corollary \ref{C:variational}, we choose
\[
b(t)=(e^{-\alpha t} -e^{-\alpha T})^\beta,\ \ \ \ 0\le t\le T,
\]
with $\beta >2$ and $\alpha >0$. With such choice a simple
computation gives,
\[
\gamma (t) =\frac{d}{4}  \left( 2\rho_1 -\alpha \beta -\alpha \beta
\frac{\kappa}{\rho_2} -e^{-\alpha T} \left( \alpha(\beta -1) +
\frac{\alpha \beta\kappa}{\rho_2}\right) b(t)^{
-\frac{1}{\beta}}\right).
\]
Keeping in mind that $b(T) = b'(T) = 0$, and that $b(0) =
(1-e^{-\alpha T})^\beta$, $b'(0) = - \alpha \beta (1-e^{-\alpha
T})^{\beta-1}$, we obtain from \eqref{var}
\begin{align}\label{lb}
& - \frac{\alpha \beta (1-e^{-\alpha T})^{\beta-1}}{2\rho_2}
\Gamma(\ln P_Tf) - (1-e^{-\alpha T})^\beta \Gamma^Z(\ln P_T f)
\\
& \ge - \frac{2}{d\rho_2} \left(\int_0^T b'(t) \gamma (t) dt\right)
\frac{L P_T f}{P_T f} + \frac{1}{d\rho_2} \left(\int_0^T b'(t)
\gamma (t)^2 dt\right). \notag
\end{align} Now,
\begin{align*}
\int_0^T b'(t) \gamma (t) dt  = & - \frac{d}{4} \left( 2 \rho_1
-\alpha\beta - \alpha \beta \frac{\kappa}{\rho_2}\right) (1
-e^{-\alpha T})^\beta \\
& + \frac{d}{4}\frac{1}{1-\frac{1}{\beta}} \left( \alpha\beta
-\alpha+ \alpha \beta \frac{\kappa}{\rho_2} \right) e^{-\alpha T}(1
-e^{-\alpha T})^{\beta-1},
\end{align*}
\begin{align*}
\int_0^T b'(t) \gamma (t)^2 dt= & -\frac{d^2}{16} \left( 2 \rho_1 -\alpha\beta - \alpha \beta \frac{\kappa}{\rho_2}\right)^2 (1 -e^{-\alpha T})^\beta
 \\ &+\frac{d^2}{8} \frac{ \left( 2 \rho_1 -\alpha\beta - \alpha \beta \frac{\kappa}{\rho_2}\right) \left( \alpha\beta -\alpha+ \alpha \beta \frac{\kappa}{\rho_2} \right)}{1-\frac{1}{\beta}} e^{-\alpha T}(1 -e^{-\alpha T})^{\beta-1} \\
 &-\frac{d^2}{16}\frac{\left( \alpha\beta -\alpha+ \alpha \beta \frac{\kappa}{\rho_2} \right)^2}{ 1-\frac{2}{\beta}}e^{-2\alpha T}(1 -e^{-\alpha T})^{\beta-2}.
\end{align*}
If we choose
\[
\alpha= \frac{2 \rho_1\rho_2}{\beta(\rho_2+\kappa)},
\]
then \[ 2\rho_1 - \alpha \beta - \alpha \beta \frac{\kappa}{\rho_2}
= 0, \ \ \ \alpha \beta - \alpha + \alpha \beta
\frac{\kappa}{\rho_2} = 2\rho_1 - \alpha,\]
 and we obtain from \eqref{lb}:
\begin{align}\label{gamma_bound}
0 \le & \frac{\rho_1}{\rho_2+ \kappa} \Gamma (\ln P_T
f)+(1-e^{-\alpha T})\Gamma^Z (\ln P_T f) \le
\frac{d(2\rho_1-\alpha)}{2\rho_2 \left(1-\frac{1}{\beta} \right)}
e^{-\alpha T} \frac{LP_T f}{P_T f}
\\
& + \frac{d(2\rho_1-\alpha)^2}{16 \rho_2 \left(1-\frac{2}{\beta}
\right)} \frac{e^{-2\alpha T}}{ 1-e^{-\alpha T}}. \notag
\end{align}
Noting that $2\rho_1 - \alpha =
\frac{2\rho_1}{\beta(\rho_2+\kappa)}((\beta-1)\rho_2 +
\beta\kappa)>0$, and that $\beta
>2$ implies $\alpha < \frac{\rho_1 \rho_2}{\rho_2 + \kappa}$,
\eqref{gamma_bound} gives in particular the desired lower bound for
$\frac{\partial }{\partial t} \ln P_t f (x) $ with $\nu = \alpha$.

The upper bound is more delicate. We fix  $ 0<\eta = \frac{2
\rho_1\rho_2}{\beta(\rho_2+\kappa)}$, and with $\gamma =
2\beta\rho_1\rho_2$ we now choose in \eqref{lb}
\[
\alpha= \frac{2 \rho_1\rho_2- \gamma e^{-\eta T}
}{\beta(\rho_2+\kappa)} = \eta - \frac{\gamma e^{-\eta T}
}{\beta(\rho_2+\kappa)}.
\]
Clearly, $\alpha>0$ provided that $T$ be sufficiently large. This
choice gives
\[ 2\rho_1 - \alpha \beta - \alpha \beta \frac{\kappa}{\rho_2}
= \frac{\gamma e^{-\eta T}}{\rho_2}, \ \ \ \alpha \beta - \alpha +
\alpha \beta \frac{\kappa}{\rho_2} = 2\rho_1 - \alpha - \frac{\gamma
e^{-\eta T}}{\rho_2}.
\]
We thus have
\begin{align*}
\int_0^T b'(t) \gamma (t) dt & = - \frac{d}{4} e^{-\alpha T} (1 -
e^{-\alpha T})^{\beta-1} \left\{\frac{\gamma (1 - e^{-\alpha T})
e^{-(\eta-\alpha)T}}{\rho_2} - \frac{\beta}{\beta-1} (2 \rho_1 -
\alpha - \frac{\gamma e^{-\eta T}}{\rho_2})\right\}. \end{align*}
Noting that $e^{-(\eta-\alpha)T} = e^{-\frac{\gamma T e^{-\eta
T}}{\beta(\rho_2+\kappa)}}\to 1$, and $\alpha \longrightarrow
 \frac{2\rho_1\rho_2}{\beta(\rho_2+\kappa)}$ as $T\to \infty$, we
obtain
\[
\frac{\gamma (1 - e^{-\alpha T}) e^{-(\eta-\alpha)T}}{\rho_2} -
\frac{\beta}{\beta-1} (2 \rho_1 - \alpha - \frac{\gamma e^{-\eta
T}}{\rho_2})\ \longrightarrow\ \frac{\gamma}{\rho_2} -
\frac{\beta}{\beta-1}\left(2\rho_1  -
\frac{2\rho_1\rho_2}{\beta(\rho_2+\kappa)}\right).\] Since by our
choice of $\gamma$ we have $\frac{\gamma}{\rho_2} -
\frac{\beta}{\beta-1}\left(2\rho_1  -
\frac{2\rho_1\rho_2}{\beta(\rho_2+\kappa)}\right)>0$, it is clear
that we have \[ \int_0^T b'(t) \gamma (t) dt \le -
\frac{d}{8}\left(\frac{\gamma}{\rho_2} -
\frac{\beta}{\beta-1}\left(2\rho_1
-\frac{2\rho_1\rho_2}{\beta(\rho_2+\kappa)}\right)\right) e^{-\alpha
T} (1 - e^{-\alpha T})^{\beta-1} , \]
 provided that $T$ be large
enough. We also have
\begin{align*}
\int_0^T b'(t) \gamma (t)^2 dt= & -\frac{d^2}{16} e^{-2\alpha T} (1
-e^{-\alpha T})^{\beta-2} \bigg\{\frac{\beta}{\beta-2} (2 \rho_1 -
\alpha - \frac{\gamma e^{-\eta T}}{\rho_2})^2
\\
&  + \frac{\gamma^2}{\rho_2^2} (1 -e^{-\alpha T})^2
e^{-2(\eta-\alpha)T} - 2 \frac{\gamma}{\rho_2}\frac{\beta}{\beta -1}
(1 -e^{-\alpha T}) (2 \rho_1 - \alpha - \frac{\gamma e^{-\eta
T}}{\rho_2}) e^{-(\eta-\alpha)T}\bigg\}.
\end{align*}
Using our choice of $\gamma$ we see that, if we let $T\to \infty$,
the quantity between curly bracket in the right-hand side converges
to
\[
\frac{\beta}{\beta-2} 4\rho_1^2 \left(\frac{(\beta-1)\rho_2+\beta
\kappa}{\beta(\rho_2+\kappa)}\right)^2 + 4 \beta^2\rho_1^2 -
\frac{8\beta^2\rho_1^2}{\beta-1}\frac{(\beta-1)\rho_2+\beta
\kappa}{\beta(\rho_2+\kappa)}.
\]
This quantity is strictly positive provided that
\[
\frac{2\beta}{\beta-1}\frac{(\beta-1)\rho_2+\beta
\kappa}{\beta(\rho_2+\kappa)} < \frac{1}{\beta-2}
\left(\frac{(\beta-1)\rho_2+\beta
\kappa}{\beta(\rho_2+\kappa)}\right)^2 +  \beta,
\]
and this latter inequality is true, as one recognizes by applying
the inequality $2xy\le x^2 + y^2$. From these considerations and
from \eqref{lb} we conclude the desired upper bound for
$\frac{\partial }{\partial t} \ln P_t f (x) $.

\end{proof}

\begin{proposition}\label{harnack_spectral}
Let $0< \nu < \frac{\rho_1\rho_2}{\kappa +\rho_2}$. There exist $t_0
>0$ and $C_2>0$ such that for every $f\in C^\infty_0(\mathbb
M)$, with $f\ge 0$,
\[
e^{-C_2 e^{-\nu t} d(x,y)} \le \frac{P_t f  (x)}{P_t f (y)} \le
e^{C_2 e^{-\nu t} d(x,y)} , \quad \ \ \ x,y \in \mathbb{M},\ t \ge
t_0.
\]
\end{proposition}

\begin{proof}
If we combine \eqref{gamma_bound} (in which we take $\alpha = \nu$), with the upper bound of
Proposition \ref{P:bound_kernel}, we obtain that for $x \in
\mathbb{M}$ and $t \ge t_0$,
\[
\Gamma (\ln P_t f)(x) \le C_2^2 e^{-2\nu t},
\]
with $C_2 = \sqrt{d(2\rho_1-\nu)/2\rho_2(1-\beta^{-1})}$.
We infer that the function $u(x) = C_2^{-1} e^{\nu t} \ln P_t f(x)$,
which belongs to $C^\infty(\bM)$, is such that $||\Gamma(u)||_\infty \le 1$. From \eqref{di} we obtain that \[ |u(x) - u(y)| \le
d(x,y), \ \ \ \ \ x, y\in \bM. \] This implies the sought for
conclusion.

\end{proof}

If we now fix $x \in \mathbb{M}$, and denote by $p(x,\cdot,t)$ the
heat kernel with singularity at $(x,0)$, then according to
Proposition \ref{P:bound_kernel} we obtain for $t \ge t_0$,
\begin{align}\label{estimeenoyau}
\left| \frac{\partial \ln p(x,y,t)}{\partial t}\right|   \le  C_1
\exp\left(-\nu t \right),
\end{align}
with $0< \nu < \frac{\rho_1\rho_2}{\kappa +\rho_2}$. This shows that
$\ln p(\cdot,\cdot,t) $ converges when $t\to\infty$. Let us call
$\ln p_\infty $ this limit. Moreover, from Proposition
\ref{harnack_spectral}  the limit, $\ln p_{\infty}( x, \cdot)$ is a
constant $C(x)$. By the symmetry property $p(x,y,t)=p(y,x,t)$, so
that $C(x)$ actually does not depend on $x$. We deduce from this
that the measure $\mu$ is finite. We may then as well
suppose that $\mu$ is a probability measure, in which case
$p_{\infty}=1$. We assume this from now on.

We now can prove a global and explicit upper bound for the heat kernel $p(x,y,t)$.

\begin{proposition}\label{globalbound}
For $x,y \in \bM$ and $t>0$,
\[
p(x,y,t) \le \frac{1}{\left( 1-e^{-\frac{2\rho_1 \rho_2
t}{3(\rho_2+\kappa)}}
\right)^{\frac{d}{2}\left(1+\frac{3\kappa}{2\rho_2}\right)} }.
\]
\end{proposition}

\begin{proof}
We apply \eqref{gamma_bound} with $\beta=3$ and obtain
\begin{align}\label{bla}
\frac{\rho_1}{\rho_2+\kappa} \Gamma (\ln P_t f)+(1-e^{-\alpha
t})\Gamma^Z (\ln P_t f) & \le   \frac{\rho_1}{2\rho_2}
\frac{2\rho_2+3\kappa}{\rho_2+\kappa} e^{-\alpha t} \frac{LP_t
f}{P_t f}
\\
& + \frac{d \rho_1^2}{12 \rho_2} \left(\frac{2\rho_2+
3\kappa}{\rho_2+\kappa} \right)^2 \frac{e^{-2\alpha t}}{
1-e^{-\alpha t}}, \notag
\end{align}
where $\alpha=\frac{2\rho_1\rho_2}{3(\rho_2+\kappa)}$. We deduce
\[
\frac{\partial \ln P_tf}{\partial t} \ge -\frac{d\rho_1}{6}
\frac{2\rho_2+3\kappa}{\rho_2+\kappa}\frac{e^{-\alpha t}}{
1-e^{-\alpha t}}.
\]
By integrating from $t$ to $\infty$, we obtain
\[
-\ln p(x,y,t) \ge -\frac{d}{2} \left( 1+\frac{3\kappa}{2\rho_2}\right) \ln (1-e^{-\alpha t}).
\]
This gives the desired conclusion.

\end{proof}

\subsection{Diameter bound}

In this subsection we conclude the proof of Theorem \ref{T:BM} by
showing that the diam$\ \bM$ is bounded. The idea is to show that the operator $L$ satisfies an
entropy-energy inequality. Such inequalities have been extensively
studied by Bakry (see chapters $4$ and $5$ in \cite{bakry-stflour}).
To simplify the computations, in what follows we denote by $D$ the
number defined in \eqref{D}, and we set
\[
\alpha=\frac{2\rho_1 \rho_2 }{3(\rho_2+\kappa)}.
\]

\begin{proposition}\label{P:entropyenergy}
For $f \in L^2 (\bM)$ such that $\int_\bM f^2 d\mu =1$, we have
\[
\int_\bM f^2 \ln f^2 d\mu \le \Phi \left( \int_\bM \Gamma(f) d\mu \right),
\]
where
\[
\Phi(x)=  D \left[  \left( 1+\frac{2}{\alpha D } x\right)\ln \left(
1+\frac{2}{\alpha D} x\right)-\frac{2}{\alpha D } x  \ln \left(
\frac{2}{\alpha D} x  \right)\right].
\]
\end{proposition}

\begin{proof}
From Proposition \ref{globalbound}, for every $f \in L^2(\bM)$ we
have
\[
\|P_t f \|_\infty \le \frac{1}{\left( 1-e^{-\alpha t}
\right)^{\frac{D}{2} }} \| f \|_2.
\]
Therefore, from Davies' theorem (Theorem 2.2.3 in \cite{Davies}), for $f \in L^2
(\bM)$ such that $\int_\bM f^2 d\mu =1$, we obtain
\[
\int_\bM f^2 \ln f^2 d\mu \le 2t \int_\bM \Gamma(f) d\mu -D \ln \left( 1-e^{-\alpha t} \right), \quad t >0.
\]
By minimizing over $t$ the right-hand side of the above inequality,
we obtain
\[
\int_\bM f^2 \ln f^2 d\mu \le -\frac{2}{\alpha} x
\ln\left(\frac{2x}{2x + \alpha D}\right) +D
\ln\left(\frac{2x+\alpha D}{\alpha D}\right).
\]
where $ x = \int_\bM \Gamma(f) d\mu$. It is now an easy exercise to
recognize that the right-hand side of the latter inequality is the
same as $\Phi(x)$.

\end{proof}

With Proposition \ref{P:entropyenergy} in hands, we can finally
complete the proof of Theorem \ref{T:BM}.

\begin{proposition}
One has
\[ \emph{diam}\ \bM \le 2 \sqrt 2 \sqrt{\frac{D}{\alpha}}\pi = 2\sqrt{3} \pi \sqrt{
\frac{\rho_2+\kappa}{\rho_1\rho_2} \left(
1+\frac{3\kappa}{2\rho_2}\right)d } .
\]
\end{proposition}

\begin{proof}
The function $\Phi$ that appears in the Proposition
\ref{P:entropyenergy} enjoys the following properties:
\begin{itemize}
\item $\Phi'(x)/x^{1/2}$  and $\Phi(x)/x^{3/2}$ are integrable on $(0,\infty)$;
\item $\Phi$ is concave;
\item $\frac{1}{2}\int_0^{+\infty} \frac{\Phi(x)}{x^{3/2}}dx=\int_0^{+\infty} \frac{\Phi'(x)}{\sqrt{x}}dx =-2\int_0^{+\infty} \sqrt{x} \Phi''(x)dx <+\infty.$
\end{itemize}
We can therefore apply the beautiful Theorem 5.4 in \cite{bakry-stflour} to deduce that the diameter of $\bM$ is finite and
\[
\emph{diam}\ \bM \le-2\int_0^{+\infty} \sqrt{x} \Phi''(x)dx.
\]
Since $\Phi''(x) = - \frac{2D}{x(2x+\alpha D)}$, a routine
calculation shows
\[
-2\int_0^{+\infty} \sqrt{x} \Phi''(x)dx= \frac{\pi}{\sqrt \rho_1} 2\sqrt{3} \sqrt{\left(\frac{\kappa}{\rho_2} + 1\right) D}.
\]
\end{proof}

\begin{remark}
The constant $2\sqrt{3} \pi \sqrt{
\frac{\rho_2+\kappa}{\rho_1\rho_2} \left(
1+\frac{3\kappa}{2\rho_2}\right)d } $ is not sharp. For instance, if $\M$ is a Riemannian manifold, we can take $d = n =$ dim$(\M)$, $\kappa=0$, and we thus obtain
\[
\emph{diam}\ \bM \le  2\sqrt{3}\pi\sqrt{ \frac{n}{\rho_1}},
\]
whereas it is known from the classical Bonnet-Myer's theorem that
\[
\emph{diam}\ \bM  \le  \pi \sqrt{ \frac{n-1}{\rho_1}}.
\]
\end{remark}

\vskip 0.3in

\bigskip
\footnotesize
\noindent\textit{Acknowledgments.}
The first author supported in part by NSF Grant DMS 0907326. The second author was supported in part by NSF Grant DMS-1001317


\begin{thebibliography}{99}




\bibitem{A}
A. Agrachev, \emph{Geometry of optimal control problems and Hamiltonian systems}, SISSA, preprint series, 42/2005/M.

\bibitem{AL} A. Agrachev \& P. Lee, \emph{Generalized Ricci curvature bounds on three-dimensional contact sub-Riemannian manifolds}, To appear Math. Ann., (2014)

\bibitem{AL2} A. Agrachev, P. Lee: Bishop and Laplacian comparison theorems on three-dimensional contact subriemannian manifolds with symmetry, to appear Journal of Geometric Analysis, (2014).

\bibitem{sobolog}  C. An\'e;  S. Blach\`ere,  D. Chafa\"i, P. Foug\`eres, I. Gentil, F. Malrieu, C. Roberto, G. Scheffer, \emph{ Sur les in\'egalit\'es de Sobolev logarithmiques}. (French) [Logarithmic Sobolev inequalities] With a preface by Dominique Bakry and Michel Ledoux. Panoramas et Synth\`eses, 10. Soci\'et\'e Math\'ematique de France, Paris, 2000. xvi+217

\bibitem{bakry-CRAS} D. Bakry, \emph{Un crit\`ere de non-explosion pour certaines diffusions sur une vari\'et\'e riemannienne compl\`ete.}C. R. Acad. Sci. Paris S\'er. I Math. 303 (1986), no. 1, 23Ð26.

\bibitem{bakry-stflour}
D. Bakry, \emph{L'hypercontractivit\'e et son utilisation en
th\'eorie des semigroupes}, Ecole d'Et\'e de Probabilites de
St-Flour, Lecture Notes in Math, (1994).

\bibitem{bakry-tata}
D. Bakry, \emph{Functional inequalities for Markov
semigroups}. Probability measures on groups: recent directions and
  trends, 91--147, Tata Inst. Fund. Res., Mumbai, 2006.



\bibitem{bakry-baudoin} D. Bakry, F. Baudoin, M. Bonnefont \&
B. Qian, \emph{Subelliptic Li-Yau estimates on three dimensional
model spaces},  Potential Theory and Stochastics in Albac, Aurel Cornea Memorial Volume (2009).

\bibitem{Bakry-Emery} 
D. Bakry \& M. Emery, \emph{Diffusions hypercontractives}, S\'emin. de probabilit\'es XIX, Univ. Strasbourg, Springer, 1983.


\bibitem{Bakry-Ledoux1} D. Bakry \& M. Ledoux, \emph{
Sobolev inequalities and Myers's diameter theorem for an abstract Markov generator}.
Duke Math. J. 85 (1996), no. 1, 253Ð270. 



\bibitem{Bakry-Ledoux} D. Bakry \& M. Ledoux, \emph{A logarithmic Sobolev form of
the Li-Yau parabolic inequality}. Revista Mat. Iberoamericana
\textbf{22}~(2006), 683--702.


\bibitem{baudoin} F. Baudoin, \emph{An introduction to the geometry
of stochastic flows}, Imperial College Press, London, 2004. x+140
pp.

\bibitem{baudoin-bonnefont} F. Baudoin \& M. Bonnefont,
\emph{The subelliptic heat kernel on $\mathbb{SU}(2)$: Representations,
asymptotics and gradient bounds}, Math. Zeit., \textbf{263}, no. 3, (2009), 647-672. 
\bibitem{BBG}
F. Baudoin, M. Bonnefont \& N. Garofalo, \emph{A sub-Riemannian curvature-dimension inequality, volume doubling property and the Poincar\'e inequality},  Math. Ann. 358 (2014), 3-4, 833-860

\bibitem{BGjga}
F. Baudoin \& N. Garofalo, \emph{Perelman's entropy and doubling property on Riemannian manifolds}, Journal of Geometric Analysis, \textbf{21}~(2011), 1119-1131. 



\bibitem{BGa} C.P. Boyer, K. Galicki: \emph{ 3-Sasakian manifolds}, arXiv:hep-th/9810250



\bibitem{Cao-Yau} H.D. Cao, S.T. Yau, \emph{Gradient estimates, Harnack inequalities and estimates for
heat kernels of the sum of squares of vector fields}, Mathematische
Zeitschrift, \textbf{211}~(1992), 485-504.







\bibitem{CKS}
E. Carlen, S. Kusuoka \& D. Stroock, \emph{Upper bounds for
symmetric Markov transition functions},
 Ann. Inst. H. Poincar\'e Probab. Statist.  \textbf{23}~ (1987),  no. 2, suppl.,
 245--287.


\bibitem{Chavel}
I. Chavel, \emph{Riemannian geometry: a modern introduction},
Cambridge Tracts in Mathematics, vol. 108, Cambridge Univ. Press,
1993.
 



\bibitem{CS}
T. Coulhon \& A. Sikora, \emph{Gaussian heat kernel bounds via PhragmŽn-Lindelš?f theorem}, Proc. London Math. Soc. 3, 96 (2008) 507-544.


\bibitem{CLN}
B. Chow, P. Lu \& L. Ni, \emph{Hamilton's Ricci flow}, Graduate Studies in Mathematics, 77. American Mathematical Society, Providence, RI; Science Press, New York, 2006. xxxvi+608 pp.


\bibitem{CDKR}
M. Cowling, A. H.  Dooley,  A.  Kor\'anyi \& F. Ricci,
\emph{ $H$-type groups and Iwasawa decompositions}, Adv. in Math.,
\textbf{87}~(1991), 1-41.




\bibitem{Davies} Davies, E.B. \emph{Heat kernels and spectral theory}. Cambridge Tracts in Mathematics, 92. Cambridge University Press, Cambridge, 1989.



\bibitem{dodziuk}  
J. Dodziuk, \emph{Maximum principle for parabolic inequalities and the heat flow on open manifolds}, Indiana
Univ. Math. J. \textbf{32}~(1983), 703-716.


\bibitem{CR} 
S. Dragomir \& G. Tomassini, \emph{Differential geometry and analysis on CR manifolds}, Birkh\"auser, Vol. 246, 2006.

\bibitem{falcitelli} 
M. Falcitelli, S. Ianus \& A. M. Pastore,
\emph{Riemannian submersions and related topics.} World Scientific Publishing Co., Inc., River Edge, NJ, 2004. xiv+277 pp

\bibitem{FP1}
C. Fefferman \& D. H. Phong, \emph{Subelliptic eigenvalue problems},
Conference on harmonic analysis in honor of Antoni Zygmund, Vol. I,
II (Chicago, Ill., 1981),  590--606, Wadsworth Math. Ser.,
Wadsworth, Belmont, CA, 1983.

\bibitem{FSC}
C. L. Fefferman \& A. S\'anchez-Calle, \emph{Fundamental solutions
for second order subelliptic operators},  Ann. of Math. (2)
\textbf{124}~(1986), no. 2, 247--272.

\bibitem{Fo}
G. Folland, \emph{Subelliptic estimates and function spaces on
nilpotent Lie groups}, Ark. Math., \textbf{13}~(1975), 161-207.

\bibitem{Fried}
A. Friedman, \emph{Partial differential equations of parabolic type}, Dover, 2008.

\bibitem{Fu}
M. Fukushima,  Y. Oshima, M. Takeda, \emph{Dirichlet forms and Symmetric Markov processes},
de Gruyter Studies in Mathematics, 19, (1994).






\bibitem{GW}
R. Green \& Wu, \emph{Function theory on manifolds which possess a
pole}, Lecture Notes in Math., vol. 699, Springer-Verlag, 1979.




\bibitem{Gri} A. Grigor'yan,
\emph{Heat kernel and analysis on manifolds. }
AMS/IP Studies in Advanced Mathematics, 47. American Mathematical Society, Providence, RI; International Press, Boston, MA, 2009. xviii+482 pp.



\bibitem{Ho}
L. H\"ormander, \emph{Hypoelliptic second-order differential
equations}, Acta Math., \textbf{119}~(1967), 147-171.

\bibitem{Hu}
K. Hughen, \emph{The geometry of sub-Riemannian three-manifolds}, 1995, Duke University preprint server.










\bibitem{juillet} N. Juillet, \emph{Geometric inequalities and
    generalized Ricci bounds in the Heisenberg group}, Int. Math. Res. Not. IMRN no. 13~(2009), 2347-2373.

\bibitem{Kaplan} A Kaplan. \emph{Fundamental solutions for a class of hypoelliptic PDE generated
by composition of quadratic forms}. Trans. Amer. Math. Soc., 258(1):147Ð153,
1980





\bibitem{ledoux-zurich} M. Ledoux, \emph{The geometry of Markov
diffusion generators.} Probability theory. Ann. Fac. Sci. Toulouse
Math., (6) \textbf{9}~(2000), no. 2, 305--366.


\bibitem{LI} P. Li, \emph{Uniqueness of $L^1$ solutions for the Laplace equation and the heat equation on Riemannian manifolds}, Journ. Diff. Geom., \textbf{20}~(1984), 447-457.

\bibitem{LY}
P. Li \& S. T. Yau, \emph{On the parabolic kernel of the
Schr\"odinger operator}, Acta Math., \textbf{156}~(1986), 153-201.

\bibitem{li} 
X.D. Li, \emph{Liouville theorems for symmetric diffusion
operators on complete Riemannian manifolds}, J. Math. Pures Appl. 84, (2005), 1295-1361.



\bibitem{villani-lott} J. Lott \& C. Villani, \emph{Ricci
    curvature for metric-measure spaces via optimal transport}, Annals of Math., Vol. 169, No. 3., pp.903-991, 2009.


\bibitem{M}
T. Melcher, \emph{Hypoelliptic heat kernel inequalities on Lie groups}. Stochastic Process. Appl. 118 (2008), no. 3, 368388





\bibitem{Montgomery}
R. Montgomery, \emph{A tour of sub-Riemannian geometries, their
geodesics and applications}, Math. Surveys and Monographs, vol. 91,
Amer. Math. Soc., 2002.


\bibitem{ollivier} 
Y. Ollivier, \emph{Ricci curvature of Markov chains on metric spaces}
J. Funct. Anal. \textbf{256}~(2009), 3, 810-864.

\bibitem{Per}
G. Perelman, \emph{The entropy formula for the Ricci flow and its geometric applications}, ArXiv:math.DG/0211159.

 
\bibitem{PS1}
R. S. Phillips \& L. Sarason,  \emph{Elliptic-parabolic equations of the second order}, J. Math. Mech. \textbf{17}~(1967/1968), 891-917.




\bibitem{reed1}  M. Reed \& B. Simon, \emph{Methods of modern mathematical physics. Functional analysis}. Second edition. Academic Press, Inc. [Harcourt Brace Jovanovich, Publishers], New York, 1980.


\bibitem{SVR}
M.-K. von Renesse \& K-T. Sturm, \emph{Transport inequalities,
gradient estimates, entropy, and Ricci curvature}, Comm. Pure Appl.
Math. \textbf{58}~(2005),  no. 7, 923--940.

\bibitem{RS}
L. P. Rothschild \& E. M. Stein, \emph{Hypoelliptic differential
operators and nilpotent groups},  Acta Math.  \textbf{137}~(1976),
no. 3-4, 247--320.


\bibitem{rumin} M. Rumin, \emph{Formes diff\'erentielles sur les vari\'et\'es de contact}.  (French) [Differential forms on contact manifolds]
J. Differential Geom. \textbf{39}~(1994), no. 2, 281-330.





\bibitem{strichartz1} R. Strichartz, \emph{Analysis of the Laplacian on the complete Riemannian manifold}, Journal Func. Anal., 52, 1, 48-79, (1983).

\bibitem{Strichartz} R. Strichartz, \emph{Sub-Riemannian geometry},
Journ. Diff. Geom.,  \textbf{24}~(1986), 221-263.

\bibitem{stricorr}
R. Strichartz, \emph{Corrections to ``Sub-Riemannian geometry''
[Journ. Diff. Geom.,  \textbf{24}~(1986), 221-263]},
\textbf{30}~(2)~(1989), 595-596.


\bibitem{sturm1} K. Th. Sturm, \emph{On the geometry of metric measure spaces I},
 Acta Math., \textbf{196}, no.1, (2006), 65--131.

\bibitem{sturm2}  K. Th. Sturm, \emph{On the geometry of metric measure spaces II},
 Acta Math., \textbf{196}, no.1, (2006), 133--177.
 








\bibitem{VSC}
N. Varopoulos, L. Saloff-Coste \& T. Coulhon, \emph{Analysis and
Geometry on Groups}, Cambridge University Press, 1992.


\bibitem{Yau} S.T. Yau, \emph{Harmonic functions on complete Riemannian manifolds},
Comm. Pure Appl. Math. 28 (1975), 201--228.



\bibitem{Yau2} S.T. Yau,  \emph{On the heat kernel of a complete Riemannian manifold}.
J. Math. Pures Appl. (9) 57 (1978), no. 2, 191--201.






\end{thebibliography}
\end{document}